\documentclass{tac}
\usepackage{amssymb,amsmath,stmaryrd,mathrsfs}
\usepackage[all,2cell]{xy}
\usepackage{enumitem}
\usepackage{color}
\definecolor{darkgreen}{rgb}{0,0.45,0} 
\usepackage[pagebackref,colorlinks,citecolor=darkgreen,%
linkcolor=darkgreen]{hyperref}
\usepackage{mathtools}          
\usepackage{braket}             
\usepackage{url}                
\usepackage{xspace}             

\makeatletter
\let\ea\expandafter

\def\mdef#1#2{\ea\ea\ea\gdef\ea\ea\noexpand#1\ea%
  {\ea\ensuremath\ea{#2}\xspace}}
\def\alwaysmath#1{\ea\ea\ea\global\ea\ea\ea\let\ea\ea%
  \csname your@#1\endcsname\csname #1\endcsname
  \ea\def\csname #1\endcsname%
  {\ensuremath{\csname your@#1\endcsname}\xspace}}

\newcount\foreachcount

\def\foreachLetter#1#2#3{\foreachcount=#1
  \ea\loop\ea\ea\ea#3\@Alph\foreachcount
  \advance\foreachcount by 1
  \ifnum\foreachcount<#2\repeat}

\def\definescr#1{\ea\gdef\csname s#1\endcsname%
  {\ensuremath{\mathscr{#1}}\xspace}}
\foreachLetter{1}{27}{\definescr}
\def\definecal#1{\ea\gdef\csname c#1\endcsname%
  {\ensuremath{\mathcal{#1}}\xspace}}
\foreachLetter{1}{27}{\definecal}
\def\definebold#1{\ea\gdef\csname b#1\endcsname%
  {\ensuremath{\mathbf{#1}}\xspace}}
\foreachLetter{1}{27}{\definebold}
\def\definebb#1{\ea\gdef\csname l#1\endcsname%
  {\ensuremath{\mathbb{#1}}\xspace}}
\foreachLetter{1}{27}{\definebb}

\def\autofmt@n#1\autofmt@end{\mathrm{#1}}
\def\autofmt@b#1\autofmt@end{\mathbf{#1}}
\def\autofmt@l#1#2\autofmt@end{\mathbb{#1}\mathsf{#2}}
\def\autofmt@c#1#2\autofmt@end{\mathcal{#1}\mathit{#2}}
\def\autofmt@f#1\autofmt@end{\mathfrak{#1}}

\def\auto@drop#1{}
\def\autodef#1{\ea\ea\ea\@autodef\ea\ea\ea#1\ea%
  \auto@drop\string#1\autodef@end}
\def\@autodef#1#2#3\autodef@end{%
  \ea\def\ea#1\ea{\ea\ensuremath\ea%
    {\csname autofmt@#2\endcsname#3\autofmt@end}\xspace}}
\def\autodefs@end{blarg!}
\def\autodefs#1{\@autodefs#1\autodefs@end}
\def\@autodefs#1{\ifx#1\autodefs@end%
  \def\autodefs@next{}%
  \else%
  \def\autodefs@next{\autodef#1\@autodefs}%
  \fi\autodefs@next}

\let\meet\wedge
\newcommand{\op}{^{\mathrm{op}}}
\DeclareMathOperator\lan{Lan}
\DeclareMathOperator*\colim{colim}
\newcommand{\too}[1][]{\ensuremath{\overset{#1}{\longrightarrow}}}
\newcommand{\ot}{\ensuremath{\leftarrow}}
\let\toto\rightrightarrows
\let\into\hookrightarrow
\let\xto\xrightarrow
\let\xot\xleftarrow

\newif\ifhyperref
\@ifpackageloaded{hyperref}{\hyperreftrue}{\hyperreffalse}

\let\your@state\state
\def\state#1{\gdef\currthmtype{#1}\your@state{#1}}
\let\your@staterm\staterm
\def\staterm#1{\gdef\currthmtype{#1}\your@staterm{#1}}
\let\defthm\newtheorem
\def\currthmtype{}
\ifhyperref
\def\autoref#1{\ref*{label@name@#1}~\ref{#1}}
\else
\def\autoref#1{\ref{label@name@#1}~\ref{#1}}
\fi
\AtBeginDocument{%
  \let\old@label\label%
  \def\label#1{%
    {\let\your@currentlabel\@currentlabel%
      \edef\@currentlabel{\currthmtype}%
      \old@label{label@name@#1}}%
    \old@label{#1}}
}

\newtheorem{thm}{Theorem}[section]

\defthm{cor}{Corollary}
\defthm{prop}{Proposition}
\defthm{lem}{Lemma}
\theoremstyle{definition}
\defthm{defn}{Definition}

\let\qed\endproof
\def\qedhere{\quad\hbox{\endproofbox}}

\renewcommand{\theenumi}{(\roman{enumi})}

\let\c@equation\c@subsection
\numberwithin{equation}{section}

\@ifpackageloaded{mathtools}{%
\mathtoolsset{showonlyrefs,showmanualtags}}{}

\alwaysmath{alpha}
\alwaysmath{beta}
\alwaysmath{gamma}
\alwaysmath{Gamma}
\alwaysmath{delta}
\alwaysmath{Delta}
\alwaysmath{epsilon}
\mdef\ep{\varepsilon}
\alwaysmath{zeta}
\alwaysmath{eta}
\alwaysmath{theta}
\alwaysmath{Theta}
\alwaysmath{iota}
\alwaysmath{kappa}
\alwaysmath{lambda}
\alwaysmath{Lambda}
\alwaysmath{mu}
\alwaysmath{nu}
\alwaysmath{xi}
\alwaysmath{pi}
\alwaysmath{rho}
\alwaysmath{sigma}
\alwaysmath{Sigma}
\alwaysmath{tau}
\alwaysmath{upsilon}
\alwaysmath{Upsilon}
\alwaysmath{phi}
\alwaysmath{Pi}
\alwaysmath{Phi}
\mdef\ph{\varphi}
\alwaysmath{chi}
\alwaysmath{psi}
\alwaysmath{Psi}
\alwaysmath{omega}
\alwaysmath{Omega}
\let\al\alpha

\let\si\sigma

\let\om\omega
\let\ka\kappa
\let\la\lambda
\let\La\Lambda
\let\ze\zeta

\makeatother

\title{Exact completions and small sheaves}
\author{Michael Shulman}

\address{Department of Mathematics\\University of California, San
  Diego\\9500 Gilman Dr. \#0112\\La Jolla, CA 92093}

\eaddress{viritrilbia@gmail.com}
\copyrightyear{2012}
\thanks{The author was supported by a National Science Foundation
  postdoctoral fellowship during the writing of this paper.}

\mdef\KA{\mathfrak{K}\xspace}
\def\ff#1#2{\ensuremath{\left\{#1\right\}_{#2}}}

\mdef\un{\{1\}}
\mdef\bytil{\widetilde{\mathbf{y}}\xspace}

\autodefs{\nSITE\nLEX\nWLEX\bSet\bSET\nEx\nLEx\by\nREG}
\autodefs{\nSet\nSh\ba\fy\nReg\nSET\nSH\nEX\nCAT\nLSITE}
\autodefs{\bRing\nSpec\bg\nFam\bCat\nPRETOPOS\nRel\ncoll}
\autodefs{\nMap\nMod\nTMap\lMod\lRel\nALL\nFALL\bSUP}

\mdef\SITEk{\nSITE_{\ka}}
\mdef\LSITEk{\nLSITE_{\ka}}
\mdef\Relk{\nRel_{\ka}}
\mdef\RelK{\nRel_{\KA}}
\mdef\lrelk{\lRel_{\ka}}

\mdef\exk{\nEx_{\ka}}
\mdef\lexk{\nLEx_{\ka}}
\mdef\exK{\nEx_{\KA}}
\mdef\regk{\nReg_{\ka}}
\mdef\exu{\nEx_{\un}}
\mdef\regu{\nReg_{\un}}
\mdef\REGk{\nREG_\ka}
\mdef\EXk{\nEX_\ka}
\mdef\Modk{\nMod_\ka}
\mdef\lmodk{\lMod_\ka}
\mdef\ALLk{\nALL_{\ka}}
\mdef\FALLk{\nFALL_{\ka}}
\mdef\SUPk{\bSUP_{\ka}}
\mdef\Fk{\sF_{\ka}}
\mdef\shk{\nSh_{\ka}}
\mdef\f{\mathbf{f}\xspace}
\mdef\rto{\looparrowright}
\mdef\Rto{\mathrel{\overline{\looparrowright}}}
\mdef\lto{\rightsquigarrow}
\mdef\Astar{\sA^\star}

\def\o{^{\circ}}
\def\sb{_{\bullet}}
\def\pb{^{\bullet}}
\def\sbb{_{\scriptscriptstyle\blacklozenge}}
\def\pbb{^{\scriptscriptstyle\blacklozenge}}
\def\ssing#1{\ensuremath{\llbracket #1 \rrbracket}\xspace}

\let\To\Rightarrow
\let\lle\preceq

\setitemize[1]{leftmargin=2em}

\begin{document}
\maketitle

\begin{abstract}
  We prove a general theorem which includes most notions of ``exact
  completion'' as special cases.  The theorem is that ``\ka-ary exact
  categories'' are a reflective sub-2-category of ``\ka-ary sites'',
  for any regular cardinal \ka.  A \ka-ary exact category is an exact
  category with disjoint and universal \ka-small coproducts, and a
  \ka-ary site is a site whose covering sieves are generated by
  \ka-small families and which satisfies a solution-set condition for
  finite limits relative to \ka.

  In the unary case, this includes the exact completions of a regular
  category, of a category with (weak) finite limits, and of a category
  with a factorization system.  When $\ka=\om$, it includes the
  pretopos completion of a coherent category.  And when $\ka=\KA$ is
  the size of the universe, it includes the category of sheaves on a
  small site, and the category of small presheaves on a locally small
  and finitely complete category.  The \KA-ary exact completion of a
  large nontrivial site gives a well-behaved ``category of small
  sheaves''.

  Along the way, we define a slightly generalized notion of ``morphism
  of sites'' and show that \ka-ary sites are equivalent to a type of
  ``enhanced allegory''.  This enables us to construct the exact
  completion in two ways, which can be regarded as decategorifications
  of ``representable profunctors'' (i.e.\ entire functional relations)
  and ``anafunctors'', respectively.
\end{abstract}

\setcounter{tocdepth}{1}
\tableofcontents

\section{Introduction}
\label{sec:introduction}

In this paper we show that the following ``completion'' operations are
all instances of a single general construction.
\begin{enumerate}[noitemsep]
\item The free exact category on a category with (weak) finite limits, as
  in~\cite{cm:ex-lex,carboni:free-constr,cv:reg-exact-cplt,ht:free-regex}.
  \label{item:op1}
\item The exact completion of a regular category, as
  in~\cite{c:exreg,fs:catall,carboni:free-constr,
    cv:reg-exact-cplt,lack:exreg-inf}.
\item The pretopos completion of a coherent category, as
  in~\cite{fs:catall,ptj:elephant}, and its infinitary analogue.
\item The category of sheaves on a small site.\label{item:op2}
\item The category of small presheaves on a locally small category
  satisfying the solution-set condition for finite diagrams, as
  in~\cite{dl:lim-smallfr} (the solution-set condition makes the
  category of small presheaves finitely complete).
\end{enumerate}

The existence of a relationship between the above constructions is not
surprising.  On the one hand, Giraud's theorem characterizes
categories of sheaves as the infinitary pretoposes with a small
generating set.  It is also folklore that adding disjoint universal
coproducts is the natural ``higher-ary'' generalization of exactness;
this is perhaps most explicit in~\cite{street:family}.  Furthermore,
the category of sheaves on a small infinitary-coherent category agrees
with its infinitary-pretopos completion, as remarked
in~\cite{ptj:elephant}.  On the other hand,~\cite{ht:free-regex}
showed that the free exact category on a category with weak finite
limits can be identified with a full subcategory of its presheaf
category, and~\cite{lack:exreg-inf} showed that the exact completion
of a regular category can similarly be identified with a full
subcategory of the sheaves for its regular topology.

However, in other ways the above-listed constructions appear
different.  For instance, each has a universal property, but the
universal properties are not all the same.
\begin{enumerate}[noitemsep]
\item The free exact category on a category with finite limits is a
  left adjoint to the forgetful functor.  However, the free exact
  category on a category with \emph{weak} finite limits represents
  ``left covering'' functors.
\item The exact completion of a regular category is a reflection.
\item The pretopos completion of a coherent category is also a
  reflection.
\item The category of sheaves on a small site is the classifying topos
  for flat cover-preserving functors.
\item The category of small presheaves on a locally small category is
  its free cocompletion under small colimits.
\end{enumerate}
In searching for a common generalization of these constructions, which
also unifies their universal properties, we are led to introduce the
following new definitions, relative to a regular cardinal \ka.

\begin{defn}
  A \textbf{\ka-ary site} is a site whose covering sieves are
  generated by \ka-small families, and which satisfies a certain weak
  solution-set condition for finite cones relative to \ka (see
  \S\ref{sec:sites}).
\end{defn}

This includes the inputs to all the above constructions, as follows:

\begin{enumerate}[noitemsep]
\item The trivial topology on a category is unary precisely when the
  category has weak finite limits.
\item The regular topology on a regular category is also unary.
\item The coherent topology on a coherent category is \om-ary
  (``finitary''), and its infinitary analogue is \KA-ary, where \KA is
  the size of the universe.
\item The topology of any small site is \KA-ary.
\item The trivial topology on a large category is \KA-ary just when
  that category satisfies the solution-set condition for finite
  diagrams.
\end{enumerate}

\begin{defn}
  A \textbf{\ka-ary exact category} is a category with universally
  effective equivalence relations and disjoint universal \ka-small
  coproducts.
\end{defn}

This includes the \emph{outputs} of all the above constructions, as
follows.

\begin{enumerate}[noitemsep,leftmargin=*,widest=i)-(ii]
\item[(i)-(ii)] A unary exact category is an exact category in the
  usual sense.\setcounter{enumi}{2}
\item An \om-ary exact category is a pretopos.
\item A \KA-ary exact category is an infinitary pretopos (a category
  satisfying the exactness conditions of Giraud's theorem).
\end{enumerate}

\begin{defn}
  A \textbf{morphism of sites} is a functor $\bC\to\bD$ which
  preserves covering families and is ``flat relative to the topology
  of \bD'' in the sense of~\cite{kock:postulated,karazeris:flatness}.
\end{defn}

This is a slight generalization of the usual notion of ``morphism of
sites''.  The latter requires the functor to be ``representably
flat'', which is equivalent to flatness relative to the trivial
topology of the codomain.  The two are equivalent if the sites have
actual finite limits and subcanonical topologies.  Our more general
notion also includes all ``dense inclusions'' of sub-sites, and has
other pleasing properties which the usual one does not (see
\S\ref{sec:morphisms-sites} and \S\ref{sec:dense}).

Generalized morphisms of sites include all the relevant morphisms
between all the above inputs, as follows:

\begin{enumerate}[noitemsep]
\item A morphism between sites with trivial topology is a flat
  functor.  A morphism from a unary trivial site to an exact category
  is a left covering functor.
\item A morphism of sites between regular categories is a regular
  functor.
\item A morphism of sites between coherent categories is a coherent
  functor.
\item A morphism of sites from a small site to a Grothendieck topos
  (with its canonical topology) is a flat cover-preserving functor.  A
  morphism of sites \emph{between} Grothendieck toposes is the inverse
  image functor of a geometric morphism.
\end{enumerate}

We can now state the general theorem which unifies all the above
constructions.

\begin{theorem}\label{thm:intro-excplt}
  \ka-ary exact categories form a reflective sub-2-category of \ka-ary
  sites.  The reflector is called \textbf{\ka-ary exact completion}.
\end{theorem}

Besides its intrinsic interest, this has several useful consequences.
For instance, if \bC satisfies the solution-set condition for finite
limits, then its category of small presheaves is an infinitary
pretopos.  More generally, if \bC is a \emph{large} site which is
\KA-ary (this includes most large sites arising in practice), then its
\KA-ary exact completion is a category of ``small sheaves''.  For
instance, any scheme can be regarded as a small sheaf on the large
site of rings.  The category of small sheaves shares many properties
of the sheaf topos of a small site: it is an infinitary pretopos, it
has a similar universal property, and it satisfies a ``size-free''
version of Giraud's theorem.

Additionally, by completing with successively larger \ka, we can
obtain information about ordinary sheaf toposes with ``cardinality
limits''.  For instance, the sheaves on any small \om-ary site form a
coherent topos.

We can also find ``\ka-ary regular completions'' sitting inside the
\ka-ary exact completion, in the usual way.  This includes the
classical regular completion of a category with (weak) finite limits
as in~\cite{cm:ex-lex,cv:reg-exact-cplt,ht:free-regex}, as well as
variants such as the regular completion of a category with a
factorization system from~\cite{kelly:rel-factsys} and the relative
regular completion from~\cite{hofstra:relcpltn}.  More generally, we
can obtain the exact completions of~\cite{gl:lex-colimits} relative to
a class of lex-colimits.

Finally, our approach to proving \autoref{thm:intro-excplt} also
unifies many existing proofs.  There are three general methods used to
construct the known exact completions.
\begin{enumerate}[label=(\alph*),noitemsep]
\item Construct a bicategory of binary relations from the input,
  complete it under certain colimits, then consider the category of
  ``maps'' (left adjoints) in the result.\label{item:allpf}
\item As objects take ``\ka-ary congruences'' (many-object equivalence
  relations), and as morphisms take ``congruence functors'', perhaps
  with ``weak equivalences'' inverted.
\item Find the exact completion as a subcategory of the category of
  (pre)sheaves.
\end{enumerate}

The bicategories used in~\ref{item:allpf} are called
\emph{allegories}~\cite{fs:catall}.  In order to generalize this
construction to \ka-ary sites, we are led to the following
``enhanced'' notion of allegory.

\begin{defn}
  A \textbf{framed allegory} is an allegory equipped with a category
  of ``tight maps'', each of which has an underlying map in the
  underlying allegory.
\end{defn}

Framed allegories are a ``decategorification'' of proarrow
equipments~\cite{wood:proarrows-i}, framed
bicategories~\cite{shulman:frbi}, and \sF-categories~\cite{ls:limlax}.
We can then prove:

\begin{theorem}\label{thm:intro-frall}
  The 2-category of \ka-ary sites is equivalent to a suitable
  2-category of framed allegories.
\end{theorem}

Besides further justifying the notion of ``\ka-ary site'', this
theorem allows us to construct the exact completion of \ka-ary sites
using analogues of all three of the above methods.
\begin{enumerate}[label=(\alph*),noitemsep]
\item We can build the corresponding framed allegory, forget the
  framing to obtain an ordinary allegory, then perform the usual
  completion under appropriate colimits and consider the category of
  maps.  This construction is most convenient for obtaining the
  universal property of the exact completion.\label{item:kameth1}
\item Alternatively, we can \emph{first} complete a framed allegory
  under a corresponding type of ``framed colimit'', then forget the
  framing and consider the category of maps.  The objects of this
  framed cocompletion are \ka-ary congruences and its tight maps are
  ``congruence functors''.  In this case, the last step is equivalent
  to constructing a category of fractions of the tight
  maps.\label{item:kameth2}
\item The universal property of the exact completion, obtained
  from~\ref{item:kameth1}, induces a functor into the category of
  sheaves.  Using description~\ref{item:kameth2} of the exact
  completion, we show that this functor is fully faithful and identify
  its image.\label{item:kameth3}
\end{enumerate}

We find it convenient to describe the completion operations
in~\ref{item:kameth1} and~\ref{item:kameth2} in terms of enriched
category theory, using ideas
from~\cite{ls:ftm2,bls:weak-aspects,ls:limlax}.  This also makes clear
that~\ref{item:kameth1} is a decategorification of the ``enriched
categories and modules'' construction
from~\cite{street:cauchy-enr,ckw:axiom-mod}, while the ``framed
colimits'' in~\ref{item:kameth2} are a decategorification of those
appearing in~\cite{wood:proarrows-ii,ls:limlax}, and that all of these
are an aspect of Cauchy or \mbox{absolute}
cocompletion~\cite{lawvere:metric-spaces,street:absolute}.  The idea
of ``categorified sheaves'', and the connection to Cauchy completion,
is also explicit in~\cite{street:cauchy-enr,ckw:axiom-mod}, building
on~\cite{walters:sheaves-cauchy-1,walters:sheaves-cauchy-2}.

We hope that making these connections explicit will facilitate the
study of exact completions of higher categories.  Of particular
interest is the fact that the ``framed colimits''
in~\ref{item:kameth2} naturally produce decategorified versions of
\emph{functors}, in addition to the \emph{profunctors} resulting from
the colimits in~\ref{item:kameth1}.

\begin{remark}
  There are a few other viewpoints on exact completion in the
  literature, such as that of~\cite{cw:regcplt,cw:factreg}, which seem
  not to be closely related to this paper.
\end{remark}

\subsection{Organization}
\label{sec:organization}

We begin in \S\ref{sec:preliminaries} with some preliminary
definitions.  Then we define the basic notions mentioned above: in
\S\ref{sec:sites} we define \ka-ary sites, in
\S\ref{sec:morphisms-sites} we define morphisms of sites, and in
\S\ref{sec:regularity} we define \ka-ary regularity and exactness.  In
\S\ref{sec:fr-alleg-new} we recall the notion of allegory, define
framed allegories, and prove \autoref{thm:intro-frall}.  Then in
\S\ref{sec:exact-completion} we deduce \autoref{thm:intro-excplt}
using construction~\ref{item:kameth1}.  In \S\ref{sec:frac} and
\S\ref{sec:exact-compl-sheav} we show the equivalence with
constructions~\ref{item:kameth2} and~\ref{item:kameth3}, respectively.

We will explain the relationship of our theory to each existing sort
of exact completion as we develop the requisite technology in
\S\S\ref{sec:exact-completion}--\ref{sec:exact-compl-sheav}.  In
\S\ref{sec:post-lex-colim}, we discuss separately a couple of related
notions which require a somewhat more in-depth treatment: the
\emph{postulated colimits} of~\cite{kock:postulated} and the
\emph{lex-colimits} of~\cite{gl:lex-colimits}.  In particular, we show
that the relative exact completions of~\cite{gl:lex-colimits} can also
be generalized to (possibly large) \ka-ary sites, and we derive the
\ka-ary regular completion as a special case.

In \S\ref{sec:dense} we study \emph{dense} morphisms of sites.  There
we further justify our generalized notion of ``morphism of sites'' by
showing that every dense inclusion is a morphism of sites, and that
every geometric morphism which lifts to a pair of sites of definition
is determined by a morphism between those sites.  Neither of these
statements is true for the classical notion of ``morphism of sites''.
Moreover, the latter is merely a special case of a fact about \ka-ary
exact completions for any \ka; in particular it applies just as well
to categories of small sheaves.

\subsection{Foundational remarks}
\label{sec:foundational-remarks}

We assume two ``universes'' $\mathbb{U}_1\in \mathbb{U}_2$, and denote
by \KA the least cardinal number not in $\lU_1$, or equivalently the
cardinality of $\mathbb{U}_1$.  These universes might be Grothendieck
universes (i.e.\ \KA might be inaccessible), but they might also be
the class of all sets and the hyperclass of all classes (in a set
theory such as NBG), or they might be small reflective models of ZFC
(as in~\cite{feferman:fdns-of-ct}).  In fact, \KA might be any regular
cardinal at all; all of our constructions will apply equally well to
all regular cardinals \ka, with \KA as merely a special case.
However, for comparison with standard notions, it is helpful to have
one regular cardinal singled out to call ``the size of the universe''.

Regardless of foundational choices, we refer to objects (such as
categories) in $\mathbb{U}_1$ as \emph{small} and others as
\emph{large}, and to objects in $\mathbb{U}_2$ as \emph{moderate}
(following~\cite{street:topos}) and others as \emph{very large}.  In
particular, small objects are also moderate.  We write \bSet for the
moderate category of small sets.  \emph{All categories, functors, and
  transformations in this paper will be moderate}, with a few
exceptions such as the very large category \bSET of moderate sets.
(We do not assume categories to be locally small, however.)  But most
of our 2-categories will be very large, such as the 2-category \nCAT
of moderate categories.

\subsection{Acknowledgments}
\label{sec:acknowledgments}

I would like to thank David Roberts for discussions about anafunctors,
James Borger for the suggestion to consider small sheaves and a
prediction of their universal property, and Panagis Karazeris for
discussions about flat functors and coherent toposes.  I would also
like to thank the organizers of the CT2011 conference at which this
work was presented, as well as the anonymous referee for many helpful
suggestions.  Some of these results (the case of \un-canonical
topologies on finitely complete categories) were independently
obtained by Tomas Everaert.

\section{Preliminary notions}
\label{sec:preliminaries}

\subsection{Arity classes}

In our notions of \ka-ary site, \ka-ary exactness, etc., \ka does not
denote exactly a regular cardinal, but rather something of the
following sort.

\begin{defn}
  An \textbf{arity class} is a class \ka of small cardinal numbers
  such that:
  \begin{enumerate}[nolistsep]
  \item $1\in\ka$.\label{item:arity1}
  \item \ka is closed under indexed sums: if $\la\in\ka$ and
    $\al\colon \la \to\ka$, then $\sum_{i\in \la} \al(i)$ is also in
    \ka.\label{item:arity2}
  \item \ka is closed under indexed decompositions: if $\la\in\ka$ and
    $\sum_{i\in \la} \al(i)\in \ka$, then each $\al(i)$ is also in
    \ka.\label{item:arity3}
  \end{enumerate}
  We say that a set is \textbf{\ka-small} if its cardinality is in \ka.
\end{defn}

\begin{remark}\label{rmk:arity}
  Conditions~\ref{item:arity2} and~\ref{item:arity3} can be combined
  to say that if $\phi\colon I\to J$ is any function where $J$ is
  \ka-small, then $I$ is \ka-small if and only if all fibers of $\phi$
  are \ka-small.  Also, if we assume~\ref{item:arity3}, then
  condition~\ref{item:arity1} is equivalent to nonemptiness of $\ka$,
  since for any $\la\in\ka$ we can write $\la = \sum_{i\in\la} 1$.

  By induction,~\ref{item:arity2} implies closure under iterated
  indexed sums: for any $n\ge 2$,
  \[\sum_{i_1\in\la_1} \; \sum_{i_2\in\la_2(i_1)} \cdots
  \sum_{i_{n-1} \in\la_{n-1}(i_1,\dots,i_{n-2})} \la_n(i_1,\dots,i_{n-1})
  \]
  is in \ka if all the $\la$'s are.  Condition~\ref{item:arity1} can
  be regarded as the case $n=0$ of this (the case $n=1$ being just
  ``$\la\in\ka$ if $\la\in\ka$'').  I am indebted to Toby Bartels and
  Sridhar Ramesh for a helpful discussion of this point.
\end{remark}

\noindent
The most obvious examples are the following.
\begin{itemize}[noitemsep]
\item The set $\{1\}$ is an arity class.\label{item:cb1}
\item The set $\{0,1\}$ is an arity class.
\item For any regular cardinal $\ka\le\KA$, the set of all cardinals
  strictly less than \ka is an arity class, which we abusively denote
  also by \ka.\label{item:cb2} (We can include $\{0,1\}$ in this
  clause if we allow $2$ as a regular cardinal.)  The cases of most
  interest are $\ka=\om$ and $\ka=\KA$, which consist respectively of
  the \emph{finite} or \emph{small} cardinal numbers.
\end{itemize}

In fact, these are the only examples.  For if \ka contains any
$\la>1$, then it must be down-closed, since if $\mu\le\nu$ and $\la>1$
we can write $\nu$ as a \la-indexed sum containing \mu.  And clearly
any down-closed arity class must arise from a regular cardinal
(including possibly $2$).  So we could equally well have defined an
arity class to be ``either the set of all cardinals less than some
regular cardinal, or the set $\{1\}$''; but the definition we have
given seems less arbitrary.

\begin{remark}\label{rmk:indarity}
  For any \ka, the full subcategory $\bSet_\ka\subseteq \bSet$
  consisting of the \ka-small sets is closed under finite limits.  A
  reader familiar with ``indexed categories'' will see that all our
  constructions can be phrased using ``naively'' $\bSet_\ka$-indexed
  categories, and suspect a generalization to \bK-indexed categories
  for any finitely complete \bK.  We leave such a generalization to a
  later paper, along with potential examples such
  as~\cite{rr:colim-eff,frey:mfpo-pca}.
\end{remark}

From now on, all definitions and constructions will be relative to an
arity class \ka, whether or not this is explicitly indicated in the
notation.  We sometimes say \textbf{unary}, \textbf{finitary}, and
\textbf{infinitary} instead of \un-ary, \om-ary, and \KA-ary
respectively.

\begin{remark}\label{rmk:subsing}
  Let $x$ and $y$ be elements of some set $I$.  The
  \emph{subsingleton} $\ssing{x=y}$ is a set that contains one element
  if $x=y$ and is empty otherwise.  Then if $I$ is \ka-small, then so
  is $\ssing{x=y}$.  This is trivial unless $\ka=\un$, so we can prove
  this by cases.  Alternatively, we can observe that \ssing{x=y} is a
  fiber of the diagonal map $I\to I\times I$, and both $I$ and
  $I\times I$ are \ka-small.
\end{remark}

\subsection{Matrices and families}

We now introduce some terminology and notation for families of
morphisms.  This level of abstraction is not strictly necessary, but
otherwise the notation later on would become quite cumbersome.

By a \emph{family} (of objects or morphisms) we always mean a
\emph{small-set-indexed} family.  We will always use uppercase letters
for families and lowercase letters for their elements, such as $X =
\ff{x_i}{i\in I}$.  We use braces to denote families, although they
are not sets and in particular can contain duplicates.  For example,
the family $\{x,x\}$ has two elements.  We further abuse notation by
writing $x\in X$ to mean that there is a specified $i\in I$ such that
$x=x_i$.  We say $X = \ff{x_i}{i\in I}$ is \textbf{\ka-ary} if $I$ is
\ka-small.

If $\ff{X_i}{i\in I}$ is a family of families, we have a ``disjoint
union'' family $\bigsqcup X_i$, which is \ka-ary if $I$ is \ka-small
and each $X_i$ is \ka-ary.

\begin{defn}
  Let \bC be a category and $X$ and $Y$ families of objects of \bC.  A
  \textbf{matrix} from $X$ to $Y$, written $F\colon X\To Y$, is a
  family $F = \ff{f_{x y}}{x\in X,y\in Y}$, where each $f_{x y}$ is a
  set of morphisms from $x$ to $y$ in \bC.  If $G\colon Y\To Z$ is
  another matrix, then their composite $G F\colon X\To Z$ is
  \[ G F = \Big\{ \Set{ g f | y\in Y, g \in g_{yz}, f\in f_{x y}}
  \Big\}_{x\in X, z\in Z}. \]
\end{defn}

Composition of matrices is associative and unital.  Also, for any
family of matrices $\ff{F_i\colon X_i\To Y_i}{i\in I}$, we have a
disjoint union matrix $\bigsqcup F_i \colon \bigsqcup X_i \To
\bigsqcup Y_i$, defined by
\[ \left(\bigsqcup F_i\right)_{x y} = 
\begin{cases}
  (f_i)_{x y} &\quad x\in X_i, y\in Y_i\\
  \emptyset &\quad x\in X_i, y\in Y_j, i\neq j.
\end{cases}
\]

\begin{defn}
  A matrix $F\colon X\To Y$ is \textbf{\ka-sourced} if $X$ is \ka-ary,
  and \textbf{\ka-targeted} if $Y$ is \ka-ary.  It is
  \textbf{\ka-to-finite} if it is \ka-sourced and \om-targeted.
\end{defn}

If $F\colon X\To Y$ is a matrix and $X'$ is a subfamily of $X$, we
have an induced matrix $F|^{X'}\colon X' \To Y$.  Similarly, for a
subfamily $Y'$ of $Y$, we have $F|_{Y'}\colon X\To Y'$.

\begin{defn}
  An \textbf{array} in \bC is a matrix each of whose entries is a
  singleton.  A \textbf{sparse array} in \bC is a matrix each of whose
  entries is a subsingleton (i.e.\ contains at most one element).
\end{defn}

The composite of two (sparse) arrays $F\colon X\To Y$ and $G\colon
Y\To Z$ is always defined as a matrix.  It is a (sparse) array just
when for all $x\in X$ and $z\in Z$, the composite $g_{yz}f_{xy}$ is
independent of $y$.

We can identify objects of \bC with singleton families, and arrays
between such families with single morphisms.

\begin{defn}
  A \textbf{cone} is an array with singleton domain, and a
  \textbf{cocone} is one with singleton codomain.
\end{defn}

Cones and cocones are sometimes called \emph{sources} and \emph{sinks}
respectively, but this use of ``source'' has potential for confusion
with the source (= domain) of a morphism.  Another important sort of
sparse array is the following.

\begin{defn}
  A \textbf{functional array} is a sparse array $F\colon X\To Y$ such
  that for each $x\in X$, there is exactly one $y\in Y$ such that
  $f_{x y}$ is nonempty.
\end{defn}

Thus, if $X=\ff{x_i}{i\in I}$ and $Y=\ff{y_j}{j\in J}$, a functional
array $F\colon X\To Y$ consists of a function $f\colon I\to J$ and
morphisms $f_i\colon x_i \to y_{f(j)}$.  We abuse notation further by
writing $f(x_i)$ for $y_{f(i)}$ and $f_{x_i}$ for $f_i$, so that $F$
consists of morphisms $f_x\colon x \to f(x)$.  For instance, the
\emph{identity} functional array $X\To X$ has $f(x)=x$ and $f_x =
1_{x}$ for each $x$.

Any cocone is functional, as is any disjoint union $\bigsqcup F_i$ of
functional arrays.  Conversely, any functional array $F\colon X\To Y$
can be decomposed as
\[F = \bigsqcup F|_y\colon \bigsqcup X|_y \Longrightarrow \bigsqcup
\{y\} = Y,\] where $X|_y = \ff{x\in X}{f(x)=y}$ and each $F|_y\colon
X|_y \To y$ is the induced cocone.  (This is a slight abuse of
notation, as $F|_y$ might also refer to the sparse array $X \To y$
which is empty at those $x\in X$ with $f(x)\neq y$, but the context
will always disambiguate.)

Also, if $F\colon X\To Y$ is functional and $G\colon Y\To Z$ is any
sparse array, then the composite matrix $G F$ is also a sparse array.
If $G$ is also functional, then so is $G F$.

\begin{remark}\label{rmk:funarr-coprod}
  The category of \ka-ary families of objects in \bC and functional
  arrays is the free completion of \bC under \ka-ary coproducts.
\end{remark}

Any functor $\bg\colon \bD\to \bC$ gives rise to a family $\bg(\bD)
\coloneqq \ff{\bg(d)}{d\in \bD}$ in \bC.

\begin{defn}
  An array $F\colon X \To \bg(\bD)$ is \textbf{over $\bg$} if
  $\bg(\delta) \circ f_{d x} = f_{d' x}$ for all $\delta\colon d\to
  d'$ in \bD.
\end{defn}

This generalizes the standard notion of ``cone over a functor.''

\begin{defn}\label{def:refines}
  Let $F\colon X\To Z$ and $G\colon Y\To Z$ be arrays with the same
  target.
  \begin{enumerate}[leftmargin=*,nolistsep]
  \item We say that $F$ \textbf{factors through} $G$ or
    \textbf{refines} $G$ if for every $x\in X$ there exists a $y\in Y$
    and a morphism $h\colon x\to y$ such that $f_{z x} = g_{z y} h$
    for all $z\in Z$.  In this case we write $F\le G$.
  \item If $F\le G$ and $G\le F$, we say $F$ and $G$ are
    \textbf{equivalent}.
  \end{enumerate}
\end{defn}

Note that $F\le G$ just when there exists a functional array $H\colon
X\To Y$ such that $F = G H$.  We have a (possibly large) preorder of
\ka-sourced arrays with a fixed target, under the relation $\le$.

\subsection{\ka-prelimits}

We mentioned in the introduction that a \ka-ary site must satisfy a
solution-set condition.  We will define the actual condition in
\S\ref{sec:sites}; here we define a preliminary, closely related
notion.

\begin{defn}\label{def:prelimit}
  A \textbf{\ka-prelimit} of a functor $\bg\colon \bD\to \bC$ is a
  \ka-sourced array $T$ over $\bg$ such that every cone over $\bg$
  factors through $T$.
\end{defn}

Note that if $P\colon X\To Z$ and $Q\colon Y\To Z$ are arrays over
$Z$, then $P$ factors through $Q$ if and only if each cone
$P|^{x}\colon x\To Z$ does.  Thus, \autoref{def:prelimit} could
equally well ask that every \emph{array} over \bg factor through $T$.
That is, a \ka-prelimit of \bg is a \ka-sourced array over $\bg$ which
is $\le$-greatest among arrays over $\bg$.  Similar remarks apply to
subsequent related definitions, such as \autoref{def:lwmlim}.

In~\cite{fs:catall}, the term \emph{prelimit} refers to our \KA-prelimits.

\begin{examples}\ 
\begin{itemize}[nolistsep]
\item Since $1\in\ka$, any \textbf{limit} of $\bg$ is \emph{a
    fortiori} a \ka-prelimit.
\item Recall that a \textbf{multilimit} of $\bg$ is a set $T$ of cones
  such that for any cone $x$, there exists a unique $t\in T$ such that
  $x$ factors through $t$, and for this $t$ the factorization is
  unique.  Any \ka-small multilimit is also a \ka-prelimit.
\item Recall that a \textbf{weak limit} of $\bg$ is a cone such that
  any other cone factors through it, not necessarily uniquely.  Since
  $1\in\ka$, any weak limit is a \ka-prelimit.  Conversely,
  \un-prelimits are precisely weak limits.
\end{itemize}
\end{examples}

Note that even if a limit, multilimit, or weak limit exists, it will
not in general be the \emph{only} \ka-prelimit.  In particular, if
$\bg$ has a limit $T$, then a \ka-small family of cones over $\bg$ is
a \ka-prelimit if and only if it contains some cone whose comparison
map to $T$ is split epic (in the category of cones).

\begin{example}\label{eg:small-wmlim}
  If there is a \ka-small family which includes \emph{all} cones over
  $\bg$ (in particular, if $\ka=\KA$ and \bC and \bD are small), then
  this family is a \ka-prelimit.
\end{example}

\begin{remark}\label{rmk:gaft}
  Given \bD, a category \bC has \KA-prelimits of all \bD-shaped
  diagrams precisely when the diagonal functor $\bC\to \bC^\bD$ has a
  \emph{(co-)solution set}, as in (the dual form of) Freyd's General
  Adjoint Functor Theorem.  In fact, the crucial lemma for the GAFT
  can be phrased as ``if \bC is cocomplete and locally small, and
  $\bg\colon \bD\to \bC$ has a \KA-prelimit, then it also has a
  limit.''  See also~\cite[\S1.8]{fs:catall}
  and~\cite[Ch.~4]{ar:loc-pres}.
\end{remark}

\begin{remark}
  \KA-prelimits also appear in~\cite{dl:lim-smallfr}, though not by
  that name.  There it is proven that \bC has \KA-prelimits if and
  only if its category of small presheaves is complete.  We will
  deduce this by an alternative method in
  \S\ref{sec:exact-compl-sheav} and \S\ref{sec:dense}.
\end{remark}

A \textbf{finite \ka-prelimit} is a \ka-prelimit of a finite diagram.
Another important notion is the following.  Given a cocone $P\colon
V\To u$ and a morphism $f\colon x\to u$, for each $v\in V$ let $Q^v
\colon Y^v \To \{x,v\}$ be a \ka-prelimit of the cospan $x \xto{f} u
\xot{p_v} v$, and let $Y=\bigsqcup Y^v$.  Putting together the cocones
$Q^v|_x \colon Y^v \To x$, we obtain a cocone $Y \To x$, which we
denote $f^* P$ and call a \textbf{\ka-pre-pullback} of $P$ along $f$.
This is unique up to equivalence of cocones over $x$ (in the sense of
\autoref{def:refines}).

\subsection{Classes of epimorphisms}

We now define several types of epimorphic cocones.

\begin{defn}\label{def:eff-epi}
  Let $R\colon V \To u$ be a cocone.
  \begin{enumerate}[nolistsep]
  \item $R$ is \textbf{epic} if $f R = g R$ implies $f=g$.  (Of
    course, a cone is \textbf{monic} if it is an epic cocone in the
    opposite category.)
  \item $R$ is \textbf{extremal-epic} if it is epic, and whenever $R =
    q P$ with $q\colon z \to u$ monic, it follows that $q$ is an
    isomorphism.
  \item $R$ is \textbf{strong-epic} if it is epic, and whenever $F R =
    Q P$, for $F\colon u\To W$ a finite cone, $Q\colon z \To W$ a
    finite monic cone, and $P\colon V\To z$ any cocone, there exists
    $h\colon u \to z$ (necessarily unique) such that $h R = P$ and $Q
    h = F$.
  \item $R$ is \textbf{effective-epic} if whenever $Q\colon V\To x$ is
    a cocone such that $r_{v_1} a = r_{v_2} b$ implies $q_{v_1} a =
    q_{v_2} b$, then $Q$ factors as $h R$ for a unique $h\colon u\to
    x$.
  \end{enumerate}
\end{defn}

It is standard to show that
\begin{center}
  effective-epic $\To$ strong-epic $\To$ extremal-epic $\To$ epic.
\end{center}
Note that if \bC lacks finite products, our notion of strong-epic is
stronger than the usual one which only involves orthogonality to
\emph{single} monomorphisms.  If \bC has all finite limits, then
strong-epic and extremal-epic coincide.

Of particular importance are cocones with these properties that are
``stable under pullback''.  Since we are not assuming the existence of
actual pullbacks, defining this appropriately requires a little care.

\begin{defn}\label{def:univ}
  Let \sA be a collection of \ka-ary cocones in \bC.  We define \Astar
  to be the largest possible collection of \ka-ary cocones $P\colon
  V\To u$ such that
  \begin{enumerate}[leftmargin=*,nolistsep]
  \item if $P\in\Astar$, then $P\in\sA$, and
  \item if $P\in\Astar$, then for any $f\colon x\to u$, there exists a
    $Q\in \Astar$ such that $f Q \le P$.
  \end{enumerate}
  If \sA is the collection of cocones with some property X, we speak
  of the cocones in \Astar as being \textbf{\ka-universally X}.
\end{defn}

This is a coinductive definition.  The resulting \emph{coinduction
  principle} says that to prove $\sB\subseteq \Astar$, for some
collection of \ka-ary cocones \sB, it suffices to show that
\begin{enumerate}[label=(\alph*),leftmargin=2.5em,nolistsep]
\item if $P\in\sB$, then $P\in\sA$, and\label{item:coind1}
\item if $P\in\sB$, then for any $f\colon x\to u$, there exists a
  $Q\in \sB$ such that $f Q \le P$.\label{item:coind2}
\end{enumerate}

\begin{defn}\label{def:saturated}
  A collection \sA of \ka-ary cocones is \textbf{saturated} if
  whenever $P\in\sA$ and $P\le Q$ for a \ka-ary cocone $Q$, then also
  $Q\in\sA$.
\end{defn}

\begin{lem}
  If \sA is saturated, so is \Astar.
\end{lem}
\begin{proof}
  Let \sB be the collection of cocones $Q$ such that $P\le Q$ for some
  $P\in\Astar$.  We want to show $\sB \subseteq \Astar$, and by
  coinduction it suffices to verify~\ref{item:coind1}
  and~\ref{item:coind2} above.  Thus, suppose $Q\in\sB$, i.e.\ $P\le
  Q$ for some $P\in\Astar$.  Since \sA is saturated, and $P\in\sA$, we
  have $Q\in\sA$, so~\ref{item:coind1} holds.  And given $f$, since
  $P\in\Astar$ we have an $R\in\Astar$ (hence $R\in\sB$) with $f R \le
  P$, whence $f R \le Q$.  Thus~\ref{item:coind2} also holds.
\end{proof}

\begin{lem}\label{lem:univ}
  Suppose that \Astar is saturated (for instance, if \sA is saturated)
  and that \bC has finite \ka-prelimits.  Then a \ka-ary cocone
  $P\colon V\To u$ lies in \Astar if and only if for any $f\colon x\to
  u$, some (hence any) \ka-pre-pullback $f^* P$ lies in \sA.
\end{lem}
\begin{proof}
  Suppose $P\in\Astar$.  Then given $f$, we have some $Q\in\Astar$
  with $f Q \le P$.  Thus $Q \le f^* P$, so (by saturation)
  $f^*P\in\Astar$.  Hence, in particular, $f^*P\in\sA$.

  For the converse, let \sB be the collection of \ka-ary cocones $P$
  such that $f^* P\in\sA$ for any $f$.  By coinduction, to show that
  $\sB\subseteq \Astar$, it suffices to show~\ref{item:coind1}
  and~\ref{item:coind2} above.  Since $P$ is a \ka-pre-pullback of
  itself along $1_u$, we have~\ref{item:coind1} easily.
  For~\ref{item:coind2}, we can take $Q = f^* P$.  Then for any
  further $g\colon z\to x$, the \ka-pre-pullback $g^* (f^* P)$ is also
  a \ka-pre-pullback $(f g)^*P$, hence lies in \sA; thus $f^* P
  \in\sB$ as desired.
\end{proof}

It is easy to see that epic, extremal-epic, and strong-epic cocones
are saturated.  It seems that effective-epic cocones are not saturated
in general, but \ka-universally effective-epic cocones are, so that
\autoref{lem:univ} still applies; cf.\ for
instance~\cite[C2.1.6]{ptj:elephant}.

\section{\ka-ary sites}
\label{sec:sites}

As suggested in the introduction, a \emph{\ka-ary site} is one whose
covers are determined by \ka-small families and which satisfies a
solution-set condition.
%
%
We begin with weakly \ka-ary sites, which omit the solution-set
condition.  Recall that all categories we consider will be moderate.

\begin{defn}\label{def:site}
  A \textbf{weakly \ka-ary topology} on a category \bC consists of a
  class of \ka-ary cocones $P\colon V\To u$, called \textbf{covering
    families}, such that
  \begin{enumerate}[nolistsep]
  \item For each object $u\in \bC$, the singleton family $\{1_u\colon
    u\to u\}$ is covering.\label{item:site-sing}
  \item For any covering family $P\colon V\To u$ and any morphism
    $f\colon x\to u$, there exists a covering family $Q\colon Y \To x$
    such that $f Q \le P$.\label{item:site-pb}
  \item If $P\colon V \To u$ is a covering family and for each $v\in
    V$ we have a covering family $Q_v\colon W_v\To v$, then
    $P(\bigsqcup Q_v)\colon W \To u$ is a covering
    family.\label{item:site-comp}
  \item If $P\colon V \To u$ is a covering family and $Q\colon W\To u$
    is a \ka-ary cocone with $P\le Q$, then $Q$ is also a covering
    family.\label{item:site-sat}
  \end{enumerate}
  If \bC is equipped with a weakly \ka-ary topology, we call it a
  \textbf{weakly \ka-ary site}.
\end{defn}

\begin{remark}
  Conditions~\ref{item:site-pb} and~\ref{item:site-comp} imply that
  for any covering families $P\colon V\To u$ and $Q\colon W\To u$,
  there exists a covering family $R\colon Z\To u$ with $R\le P$ and
  $R\le Q$.
\end{remark}


\begin{remark}
  If we strengthen \ref{def:site}\ref{item:site-pb} to require
  covering families to have \emph{actual} pullbacks, as is common in
  the definition of ``Grothendieck pretopology'', then our
  \emph{weakly unary topologies} become the \emph{saturated singleton
    pretopologies} of~\cite{roberts:ana} and the
  \emph{quasi-topologies} of~\cite{hofstra:relcpltn}.
\end{remark}

We should first of all relate this definition to the usual notion of
Grothendieck topology, which consists of a collection of
\emph{covering sieves} such that
\begin{enumerate}[nolistsep,label=(\alph*)]
\item For any $u$, the maximal sieve on $u$, which consists of all
  morphisms with target $u$, is covering.\label{item:gtop1}
\item If $P$ is a covering sieve on $u$ and $f\colon v\to u$, then the
  sieve $f^{-1} P = \Set{ g\colon w\to v | f g \in P }$ is also
  covering.\label{item:gtop2}
\item If $P$ is a sieve on $u$ such that the sieve $\Set{f\colon v\to
    u | f^{-1} P \text{ is covering }}$ is covering, then $P$ is also
  covering.\label{item:gtop3}
\end{enumerate}

The general relationship between covering sieves and covering families
is well-known (see, for instance,~\cite[C2.1]{ptj:elephant}), but it
is worth making explicit here to show how the arity class \ka enters.
Recall that any cocone $P\colon V\To u$ \emph{generates} a sieve
$\overline{P} = \Set{ f\colon w\to u | f \le P }$.  We have $P
\subseteq \overline{Q}$ if and only if $\overline{P} \subseteq
\overline{Q}$, if and only if $P\le Q$.

\begin{prop}\label{thm:sieves}
  For any category \bC, there is a bijection between
  \begin{itemize}[nolistsep]
  \item Weakly \ka-ary topologies on \bC, and
  \item Grothendieck topologies on \bC, in the usual sense, such that
    every covering sieve contains a \ka-small family which generates a
    covering sieve.
  \end{itemize}
\end{prop}
\begin{proof}
  First let \bC have a weakly \ka-ary topology, and define a sieve to
  be covering if it contains a covering family.  We show this is a
  Grothendieck topology in the usual sense.

  For~\ref{item:gtop1}, the maximal sieve on $u$ contains $1_u$, hence
  is covering.

  For~\ref{item:gtop2}, if $P$ is a sieve on $u$ containing a covering
  family $P'$ and $f\colon v\to u$, then by
  \ref{def:site}\ref{item:site-pb} there exists a covering family $Q$
  of $v$ such that $f Q\le P'$, and hence $Q \le f^{-1} P$; thus the
  sieve $f^{-1} P$ is covering.

  For~\ref{item:gtop3}, if $\Set{f\colon v\to u | f^{-1} P \text{ is
      covering}}$ is covering, then it contains a covering family
  $F\colon V\To u$.  Moreover, since for each $v\in V$, the sieve
  $f_v^{-1} P$ is covering, it contains a covering family $G_v\colon
  W_v \To v$.  But then $P$ contains $F \big(\bigsqcup G_v\big)$,
  hence is also covering.

  Finally, it is clear that in this Grothendieck topology, any
  covering sieve contains a \ka-small family which generates a
  covering sieve.
  
  Now let \bC be given a Grothendieck topology satisfying the
  condition above, and define a \ka-small cocone to be covering if it
  generates a covering sieve.  We show that this defines a weakly
  \ka-ary topology.

  For~\ref{item:site-sing}, we note that the identity morphism
  generates the maximal sieve.

  For~\ref{item:site-pb}, suppose that the \ka-small family $P\colon V
  \To u$ generates a covering sieve $\overline{P}$, and let $f\colon
  x\to u$.  Then the sieve $f^{-1} \overline{P}$ on $x$ is covering,
  hence contains a \ka-small family $Q\colon Y \To x$ such that
  $\overline{Q}$ is covering.  Since $Q \subseteq f^{-1}
  \overline{P}$, we have $f Q \le P$.

  For~\ref{item:site-comp}, let $R = \overline{P (\bigsqcup Q_v)}$,
  where $P\colon V \To u$ and each $Q_v\colon W_v \To v$ are covering
  families.  Then each sieve $p_v^{-1} R$ contains the sieve
  $\overline{Q_v}$, which is covering, so it is also covering.
  Therefore, the sieve $\Set{f\colon v\to u | f^{-1} R \text{ is
      covering}}$ contains the sieve $\overline{P}$, which is
  covering, so it is also covering.  Thus, by~\ref{item:gtop3}, $R$ is
  covering.

  For~\ref{item:site-sat}, if $P\le Q$, then $\overline{P}\subseteq
  \overline{Q}$, so if $\overline{P}$ is covering then so is
  $\overline{Q}$.

  Finally, we prove the two constructions are inverse.

  If we start with a weakly \ka-ary topology, then any covering family
  $P$ generates a covering sieve $\overline{P}$ since $P \subseteq
  \overline{P}$.  Conversely, if $Q$ is a \ka-small family such that
  $\overline{Q}$ is a covering sieve, then by definition there exists
  a covering family $P$ with $P\subseteq \overline{Q}$.  That means
  that $P\le Q$, so by \ref{def:site}\ref{item:site-sat}, $Q$ is
  covering.

  In the other direction, if we start with a Grothendieck topology in
  terms of sieves, then any sieve $R$ which contains a covering family
  $P$ contains the sieve $\overline{P}$, which is covering; hence $R$
  is itself covering in the original topology.  Conversely, if $R$ is
  covering, then by assumption it contains a \ka-small family $P$ that
  generates a covering sieve, so that $P$ is a covering family
  contained by $R$.
\end{proof}

Thus, we may unambiguously ask about a topology whether it ``is weakly
\ka-ary.''  When interpreted in this sense, a weakly \ka-ary topology
is also weakly $\ka'$-ary whenever $\ka\subseteq \ka'$.  (When
expressed with covering families, to pass from \ka to $\ka'$ we need
to ``saturate''.)

We have chosen to define \ka-ary topologies in terms of covering
families rather than sieves for several reasons.  Firstly, in
constructing the exact completion, there seems no way around working
with \ka-ary covering families to some extent, and constantly
rephrasing things in terms of sieves would become tiresome.  Secondly,
covering families tend to make the constructions somewhat more
explicit, especially for small values of \ka.  And thirdly, there may
be foundational issues: a sieve on a large category is a large object,
so that a collection of such sieves is an illegitimate object in ZFC.

We will consider some examples momentarily, but first we explain the
solution-set condition that eliminates the adjective ``weakly'' from
the notion of \ka-ary site.  For this we need a few more definitions.
First of all, it is convenient to generalize the notion of covering
family as follows.

\begin{defn}\label{def:covfam}
  If $U$ is a family of objects in a weakly \ka-ary site, a
  \textbf{covering family} of $U$ is a functional array $P\colon V\To
  U$ such that each $P|_u\colon V|_u \To u$ is a covering family.
\end{defn}

For instance, \autoref{def:site}\ref{item:site-comp} can then be
rephrased as ``the composite of two covering families is covering.''
We can also generalize \ref{def:site}\ref{item:site-pb} as follows.

\begin{lem}\label{thm:fam-pb}
  If $P\colon V\To U$ is a covering family and $F\colon X\To U$ is a
  functional array, then there exists a covering family $Q\colon Y\To
  X$ such that $F Q \le P$.
\end{lem}
\begin{proof}
  For each $x\in X$, there exists a covering family $Q_x\colon Y_x \To
  x$ such that $f_x Q_x \le P_{f(x)}$; take $Q = \bigsqcup Q_x$.
\end{proof}

Thus, the category of \ka-ary families and functional arrays in a
(weakly) \ka-ary site \bC inherits a weakly unary topology whose
covers are those of \autoref{def:covfam}.  We will see in
\autoref{thm:superext-fam} that this topology can be used to
``factor'' the \ka-ary exact completion into a coproduct completion
(recall \autoref{rmk:funarr-coprod}) followed by unary exact
completion.

\begin{defn}\label{def:lwmlim}
  Let \bC be a weakly \ka-ary site.
  \begin{enumerate}[nolistsep]
  \item If $F\colon X\To Z$ and $G\colon Y\To Z$ are arrays in \bC
    with the same target, we say that $F$ \textbf{factors locally
      through} $G$ or \textbf{locally refines} $G$, and write $F\lle
    G$, if there exists a covering family $P\colon V\To X$ such that
    $F P \le G$.\label{item:locref}
  \item If $F\lle G$ and $G\lle F$, we say $F$ and $G$ are
    \textbf{locally equivalent}.
  \item A \textbf{local \ka-prelimit} of $\bg\colon \bD\to\bC$ is a
    \ka-sourced array $T$ over $\bg$ such that every cone over $\bg$
    factors locally through $T$.
  \end{enumerate}
\end{defn}

\begin{remark}\label{rmk:wm-lwm}
  Since identities cover, any \ka-prelimit is also a local
  \ka-prelimit.  The converse holds if every covering family contains
  a split epic (see \autoref{eg:trivial}).
\end{remark}

\begin{remark}\label{thm:monic-lle}
  If covering families in \bC are strong-epic and $G$ is a monic cone,
  then $F\lle G$ implies $F\le G$.  In particular, in such a \bC,
  locally equivalent monic cones are actually isomorphic, and any
  monic cone that is a local \ka-prelimit is in fact a limit.
\end{remark}

Local \ka-prelimits are ``closed under passage to covers.''

\begin{prop}\label{thm:lwml-covers}
  In a weakly \ka-ary site, if $T\colon L\To X$ is a local
  \ka-prelimit of a functor $\bg$, and $P\colon M\To L$ is a covering
  family, then $T P$ is also a local \ka-prelimit of $\bg$.
\end{prop}
\begin{proof}
  If $F\colon U\To X$ is any array over $\bg$, then by assumption we
  have a covering family $Q\colon V\To U$ with $F Q \le T$.  Thus,
  there is a functional array $H\colon V\To L$ with $F Q = T H$.  But
  by \autoref{thm:fam-pb}, there is a covering family $R\colon W\To V$
  with $H R \le P$, whence $F Q R = T H R \le T P$, and $Q R$ is also
  a covering family.
\end{proof}

Conversely, assuming actual limits, any local \ka-prelimit can be
obtained in this way.

\begin{prop}\label{thm:lwfmlim}
  Suppose that $T\colon y\To \bg(\bD)$ is a limiting cone over $\bg$
  in a weakly \ka-ary site, and $S\colon Z \To \bg(\bD)$ is a
  \ka-sourced array over $\bg$.  Let $H\colon Z\To y$ be the unique
  cocone such that $S = T H$.  Then $S$ is a local \ka-prelimit of $D$
  if and only if $H$ is covering.
\end{prop}
\begin{proof}
  ``If'' is a special case of \autoref{thm:lwml-covers}.  Conversely,
  if $S$ is a local \ka-prelimit, then $T P = S K$ for some covering
  $P\colon V \To y$ and functional $K\colon V \To Z$.  Thus $T H K = T
  P$, whence $H K = P$ since $T$ is a limiting cone.  This means $P
  \le H$, so $H$ is covering.
\end{proof}

\begin{cor}\label{thm:lwml-obj}
  A local \ka-prelimit of a single object is the same as a covering
  family of that object.
\end{cor}

Finally, we note two ways to construct local \ka-prelimits from more
basic ones.

\begin{prop}\label{thm:prodeq}
  If a weakly \ka-ary site has local binary \ka-pre-products and local
  \ka-pre-equalizers, then it has all finite nonempty local
  \ka-prelimits.
\end{prop}
\begin{proof}
  This is basically like the same property for weak
  limits,~\cite[Prop.~1]{cv:reg-exact-cplt}.  By induction, we can
  construct nonempty finite local \ka-pre-products.  Now, given a
  finite nonempty diagram $\bg\colon \bD\to\bC$, let $P \colon X_0 \To
  \bg(\bD)$ be an array exhibiting $X_0$ as a local \ka-pre-product of
  the finite family $\bg(\bD)$.  Enumerate the arrows of \bD as
  $h_1,\dots,h_m$; we will define a sequence of \ka-ary families $X_i$
  in \bC and functional arrays
  \begin{equation}
    X_m \To \dots \To X_1 \To X_0. \label{eq:prodeq-seq}
  \end{equation}
  Suppose we have constructed the sequence $X_i \To \dots \To X_0$,
  and write $u$ and $v$ for the source and target of $h_{i+1}$
  respectively.  Then we have an induced array $X_i \To \bg(\bD)$, and
  therefore in particular we have cocones $X_i \To \bg(u)$ and $X_i
  \To \bg(v)$.  For each $x\in X_i$, let $E^x \To x$ be a local
  \ka-pre-equalizer of $x\to \bg(u)\xto{\bg(h_{i+1})} \bg(v)$ and
  $x\to \bg(v)$.  Finally, define $X_{i+1}= \bigsqcup_x E^x$.  This
  completes the inductive definition of~\eqref{eq:prodeq-seq}.  It is
  straightforward to verify that $X_m$ is then a local \ka-prelimit of
  \bg.
\end{proof}

Unlike the case of ordinary limits, but like that of weak limits, it
seems that local \ka-pre-pullbacks and a local \ka-pre-terminal-object
do not suffice to construct all finite local \ka-prelimits.  We do,
however, have the following.

\begin{prop}\label{thm:pbeq-connlim}
  If a weakly \ka-ary site has local \ka-pre-pullbacks and local
  \ka-pre-equalizers, then it has all finite connected local
  \ka-prelimits.
\end{prop}
\begin{proof}
  Suppose $\bg\colon \bD\to\bC$ is a finite connected diagram, and
  pick some object $u_0 \in \bD$.  For each $v\in\bD$, let $\ell(v)$
  denote the length of the shortest zigzag from $u_0$ to $v$.  Now
  order the objects of \bD as
  \[ u_0, u_1, \dots, u_n
  \]
  in such a way that $\ell(u_i) \le \ell(u_{i+1})$ for all $i$.  We
  will inductively define, for each $i$, a \ka-ary family $Y^i$ and an
  array
  \begin{equation}
    P^i\colon Y^i \To \Set{ \bg(u_j) | j \le i}. \label{eq:pbeq-seq}
  \end{equation}
  Let $Y^0 = \{\bg(u_0)\}$ and $P^0 = \{ 1_{\bg(u_0)} \}$.  For the
  inductive step, suppose given $Y^i$ and $P^i$, choose a zigzag from
  $u_0$ to $u_{i+1}$ of minimal length, and consider the final
  morphism in this zigzag, which connects some object $v$ to
  $u_{i+1}$.  By our choice of ordering, we have $v = u_j$ for some
  $j\le i$.  If this morphism is directed $k\colon v\to u_{i+1}$, we
  let $Y^{i+1} = Y^i$ and define $P^{i+1}$ by
  \[p^{i+1}_{y,u_j} =
  \begin{cases}
    p^i_{y,u_j} & \qquad j \le i\\
    \bg(k)\circ p^i_{y,v} & \qquad j = i+1.
  \end{cases}
  \]
  If, on the other hand, this morphism is directed $k\colon u_{i+1}\to
  v$, then for each $y\in Y^i$ consider a local \ka-pre-pullback
  \[\vcenter{\xymatrix{
      Z^y \ar@{=>}[r]^G \ar@{=>}[d]_H &
      y\ar[d]^{p^i_{y,v}}\\
      \bg(u_{i+1}) \ar[r]_{\bg(k)} &
      \bg(v).}}
  \]
  Let $Y^{i+1} = \bigsqcup_{y\in Y^i} Z^y$, with
  \[p^{i+1}_{z,u_j} =
  \begin{cases}
    p^i_{g(z),u_j} \circ g_z & \qquad j \le i\\
    h_z & \qquad j = i+1.
  \end{cases}
  \]
  This completes the inductive definition of~\eqref{eq:pbeq-seq}.  We
  can use local \ka-pre-equalizers as in the proof of
  \autoref{thm:pbeq-connlim}, starting from $P^n\colon Y^n \To
  \bg(\bD)$ instead of a local \ka-pre-product $X_0$, to construct a
  local \ka-prelimit of \bg.
\end{proof}


Finally, we can define (strongly) \ka-ary sites.

\begin{defn}\label{def:site2}
  A \textbf{\ka-ary topology} on a category \bC is a weakly \ka-ary
  topology for which \bC admits all finite local \ka-prelimits.  When
  \bC is equipped with a \ka-ary topology, we call it a
  \textbf{\ka-ary site}.
\end{defn}

The existence of finite local \ka-prelimits may seem like a somewhat
technical assumption.  Its importance will become clearer with use,
but we can say at this point that it is at least a generalization of
the existence of weak limits, which is known to be necessary for the
construction of the ordinary exact completion.

\begin{remark}\label{rmk:locsite}
  By a \textbf{locally \ka-ary site} we will mean a weakly \ka-ary
  site which admits finite \emph{connected} local \ka-prelimits.  By
  \autoref{thm:prodeq}, any slice category of a locally \ka-ary site
  is a \ka-ary site.  We are interested in these not because we have
  many examples of them, but because they will clarify the
  constructions in \S\ref{sec:fr-alleg-new}.
\end{remark}

If $\ka\subseteq \ka'$, then any local \ka-prelimit is also a local
$\ka'$-prelimit, and hence any \ka-ary site is also $\ka'$-ary.  In
particular, any \ka-ary site is also \KA-ary.  We now consider some
examples.

\begin{example}
  If \bC is a small category, then it automatically has \KA-prelimits,
  and every covering sieve is generated by a \KA-small family
  (itself).  Therefore, every Grothendieck topology on a small
  category is \KA-ary.  More generally, every topology on a \ka-small
  category is \ka-ary.
\end{example}

\begin{example}
  If \bC has finite limits, or even finite \ka-prelimits, then any
  weakly \ka-ary topology is \ka-ary.  Thus, for finitely complete
  sites, the only condition to be \ka-ary is that the topology be
  determined by \ka-small covering families.

  In particular, for a large category with finite limits, a topology
  is \KA-ary if and only if it is determined by small covering
  families.  This is the case for many large sites arising in
  practice, such as topological spaces with the open cover topology,
  or $\bRing\op$ with its Zariski or \'etale topologies.
\end{example}

\begin{example}\label{eg:trivial}
  Consider the \textbf{trivial topology} on a category \bC, in which a
  sieve is covering just when it contains a split epic.  In this
  topology every sieve contains a single covering morphism (the split
  epic), so it is \ka-ary just when \bC has local \ka-prelimits.  But
  as noted in \autoref{rmk:wm-lwm}, in this case local \ka-prelimits
  reduce to plain \ka-prelimits.

  In particular, \bC admits a trivial unary topology if and only if it
  has weak finite limits.  On the other hand, every small category
  admits a trivial \KA-ary topology, and any category with finite
  limits admits a trivial \ka-ary topology for any \ka.
\end{example}

\begin{example}\label{eg:gentop}
  The intersection of any collection of weakly \ka-ary topologies is
  again weakly \ka-ary, so any collection of \ka-ary cocones generates
  a smallest weakly \ka-ary topology for which they are covering.  If
  the category has finite \ka-prelimits, then such an intersection is
  of course \ka-ary.
\end{example}

\begin{example}
  A topology is called \textbf{subcanonical} if every covering family
  is effective-epic, in the sense of \autoref{def:eff-epi}.  By
  \autoref{def:site}\ref{item:site-pb}, every covering family in a
  subcanonical and weakly \ka-ary topology must in fact be
  \emph{\ka-universally effective-epic} in the sense of
  \autoref{def:univ}.  The collection of all \ka-universally
  effective-epic cocones forms a weakly \ka-ary topology on \bC, which
  we call the \textbf{\ka-canonical topology}.  It may not be \ka-ary,
  but it will be if (for instance) \bC has finite \ka-prelimits, as is
  usually the case in practice.  Note that unlike the situation for
  trivial topologies, the \ka-canonical and $\ka'$-canonical
  topologies rarely coincide for $\ka\neq \ka'$.

  More generally, if \sA is a class of \ka-ary cocones satisfying
  \ref{def:site}\ref{item:site-sing} and~\ref{item:site-comp}, and
  $\Astar$ is saturated, then $\Astar$ is a weakly \ka-ary topology.

  If \bC has pullbacks, its \un-canonical topology coincides with the
  ``canonical singleton pretopology'' of~\cite{roberts:ana}, and
  consists of the pullback-stable regular epimorphisms.
  If \bC is small, its \KA-canonical topology agrees with its
  canonical topology as usually defined (consisting of all universally
  effective-epic sieves).  This is not necessarily the case if \bC is
  large, but it is if the canonical topology is small-generated, as in
  the next example.
\end{example}

\begin{example}\label{eg:ess-small}
  Suppose a \bC is a \emph{small-generated
    site}\footnote{Traditionally called an \emph{essentially small
      site}, but this can be confusing since \bC itself need not be
    essentially small as a category (i.e.\ equivalent to a small
    category).}, meaning that it is locally small, is equipped with a
  Grothendieck topology in the usual sense, and has a small (full)
  subcategory \bD such that every object of \bC admits a covering
  sieve generated by morphisms out of objects in \bD.  For instance,
  \bC might be a Grothendieck topos with its canonical topology.  Then
  we claim that the topology of \bC is \KA-ary.

  Firstly, given a finite diagram $\bg\colon \bE\to \bC$ (or in fact
  any small diagram), consider the family $P$ of all cones over \bg
  with vertex in \bD.  Since \bD is small and \bC is locally small,
  there are only a small number of such cones.  Thus, we may put them
  together into a single array over \bg, whose domain is a small
  family of objects of \bD.  This array is a local \ka-prelimit of
  \bg, since for any cone $T$ over \bg, we can cover its vertex by
  objects of \bD, and each resulting cone will automatically factor
  through $P$.  Thus \bC has local \ka-prelimits.

  Now suppose $R$ is a covering sieve of an object $u\in \bC$; we must
  show it contains a small family generating a covering sieve.  Let
  $P$ be the family of all morphisms $v\to u$ in $R$ with $v\in \bD$.
  Since \bD is small and \bC is locally small, $P$ is small, and it is
  clearly contained in $R$.  Thus, it remains to show $\overline{P}$
  is a covering sieve.
  
  Consider any morphism $r\colon w\to u$ in $R$.  Then the sieve
  $r^{-1} \overline{P}$ contains all maps from objects of \bD to $w$,
  hence is covering.  Thus the sieve $\Set{r \colon w \to u | r^{-1}
    \overline{P} \text{ is covering}}$ contains $R$ and hence is
  covering; thus $\overline{P}$ itself is covering.
\end{example}

\begin{example}\label{eg:factsys}
  Suppose \bC has finite limits and a stable factorization system
  $(\cE,\cM)$, where \cM consists of monos; I claim \cE is then a
  unary topology on \bC.  It clearly satisfies
  \ref{def:site}\ref{item:site-sing}--\ref{item:site-comp} for
  $\ka=\un$.  For \ref{def:site}\ref{item:site-sat}, suppose $f g\in
  \cE$ and factor $f = m e$ with $m\in \cM$ and $e\in \cE$.  Then
  unique lifting gives an $h$ with $h f g = e g$ and $m h = 1$.  So
  $m$ is split epic (by $h$) and monic (since it is in \cM) and thus
  an isomorphism.  Hence $f$, like $e$, is in \cE.
\end{example}

\section{Morphisms of sites}
\label{sec:morphisms-sites}

While the search for a general construction of exact completion has
led us to \emph{restrict} the notion of \emph{site} by requiring
\ka-arity, it simultaneously leads us to \emph{generalize} the notion
of \emph{morphism of sites}.  Classically, a morphism of sites is
defined to be a functor $\f\colon \bC\to\bD$ which preserves covering
families and is \emph{representably flat}.  By the latter condition we
mean that each functor $\bD(d,\f-)$ is flat, which is to say that for
any finite diagram $\bg$ in \bC, any cone over $\f\bg$ in $\bD$
factors through the \f-image of some cone over \bg.

A representably flat functor preserves all finite limits, and indeed
all finite prelimits, existing in its domain.  Conversely, if \bC has
finite limits, or even finite prelimits, and \f preserves them, it is
representably flat.  For our purposes, it is clearly natural to seek a
notion analogously related to \emph{local} finite prelimits.  This
leads us to the following definition which was studied
in~\cite{kock:postulated} (using the internal logic)
and~\cite{karazeris:flatness} (who called it being \emph{flat relative
  to the topology of \bD}\!).

\begin{defn}
  Let \bC be any category and \bD any site.  A functor $\f\colon
  \bC\to \bD$ is \textbf{covering-flat} if for any finite diagram
  $\bg$ in \bC, every cone over $\f\bg$ in $\bD$ factors
  \emph{locally} through the \f-image of some array over \bg.
\end{defn}

\begin{lem}
  $\f\colon \bC\to \bD$ is covering-flat if and only if for any finite
  diagram $\bg\colon \bE\to\bC$ and any cone $T$ over $\f\bg$ with
  vertex $u$, the sieve
  \begin{equation}\label{eq:ms}
    \Set{ h\colon v\to u | \text{there exists a cone } S \text{ over }
      \bg \text{ such that } T h \le \f(S) }
  \end{equation}
  is a covering sieve of $u$ in \bD.\qed
\end{lem}

\begin{example}
  Any representably flat functor is covering-flat.
  The converse holds if \bD has a trivial topology.
\end{example}

\begin{lem}\label{thm:morsite-fact}
  If \bD has finite limits, then $\f\colon \bC\to \bD$ is
  covering-flat if and only if for any finite diagram $\bg\colon
  \bE\to\bC$, the family of factorizations through $\lim \f\bg$ of the
  \f-images of all cones over \bg generates a covering sieve.
\end{lem}
\begin{proof}
  When $u=\lim \f\bg$, the family in question generates the
  sieve~\eqref{eq:ms}, which is covering if \f is covering-flat.
  Conversely, for any $u$, the sieve~\eqref{eq:ms} is the pullback to
  $u$ of the corresponding one for $\lim \f\bg$, so if the latter is
  covering, so is the former.
\end{proof}

\begin{example}
  If \bD is a Grothendieck topos with its canonical topology, and \bC
  is small, then $\f\colon \bC\to \bD$ is covering-flat if and only if
  ``\f is representably flat'' is true in the internal logic of \bD.
  This is the sort of ``flat functor'' which Diaconescu's theorem says
  is classified by geometric morphisms $\bD\to [\bC\op,\bSet]$ (it is
  not the same as being representably flat).
\end{example}

If \bC has finite \ka-prelimits for some \ka, then in
\autoref{thm:morsite-fact} it suffices to consider the family of
factorizations through $\lim \f\bg$ of the cones in some \ka-prelimit
of \bg.

\begin{example}\label{eg:left-covering}
  If \bC has weak finite limits and \bD is a regular category with its
  regular topology, then $\f\colon \bC\to \bD$ is covering-flat if and
  only if for any weak limit $t$ of a finite diagram \bg in \bC, the
  induced map $\f(t)\to \lim \f\bg$ in \bD is regular epic.  As
  observed in~\cite{karazeris:flatness}, this is precisely the
  definition of \emph{left covering} functors used
  in~\cite{cv:reg-exact-cplt} (called \emph{$\aleph_0$-flat}
  in~\cite{ht:free-regex}) to describe the universal property of
  regular and exact completions.
\end{example}

\begin{prop}\label{thm:morsite}
  Let \bC be a \ka-ary site, \bD any site, and $\f\colon \bC\to\bD$ a
  functor; the following are equivalent.
  \begin{enumerate}[leftmargin=*,nolistsep]
  \item \f is covering-flat and preserves covering
    families.\label{item:lf1}
  \item For any finite diagram $\bg\colon \bE\to \bC$ and any local
    finite \ka-prelimit $T$ of \bg, the image $\f(T)$ is a local
    \ka-prelimit of $\f\bg$.\label{item:lf2}
  \item \f preserves covering families, and for any finite diagram
    $\bg\colon \bE\to \bC$ there exists a local \ka-prelimit $T$ of
    \bg such that $\f(T)$ is a local \ka-prelimit of
    $\f\bg$.\label{item:lf3}
  \end{enumerate}
\end{prop}
\begin{proof}
  Suppose~\ref{item:lf1} and let $T$ be a local \ka-prelimit of a
  finite diagram $\bg\colon \bE\to\bC$.  Then any cone $S$ over
  $\f\bg$ factors locally through $\f(R)$ for some array $R$ over \bg.
  But since $T$ is a local \ka-prelimit, $R\lle T$, and since \f
  preserves covering families, it preserves $\lle$.  Thus $S\lle \f(R)
  \lle \f(T)$, so $\f(T)$ is a local prelimit of $\f\bg$;
  hence~\ref{item:lf1}$\Rightarrow$\ref{item:lf2}.

  Now suppose~\ref{item:lf2}.  Since a local \ka-prelimit of a single
  object is just a covering family of that object, \f preserves
  covering families.  Since \bC has finite local
  \ka-prelimits,~\ref{item:lf3} follows.

  Finally, suppose~\ref{item:lf3}, and let $\bg\colon \bE\to\bC$ be a
  finite diagram and $S$ a cone over $\f\bg$.  Let $T$ be a local
  \ka-prelimit of \bg such that $\f(T)$ is a local \ka-prelimit of
  $\f\bg$.  Then $S\lle \f(T)$; hence \f is covering-flat.
\end{proof}

\begin{remark}
  If \bC and \bD have finite \ka-prelimits and trivial \ka-ary
  topologies, then \autoref{thm:morsite} reduces to the fact that a
  functor is representably flat if and only if it preserves these
  finite \ka-prelimits.
\end{remark}

\begin{defn}\label{def:morsite}
  For sites \bC and \bD, we say $\f\colon \bC\to\bD$ is a
  \textbf{morphism of sites} if it is covering-flat and preserves
  covering families.
\end{defn}

\begin{cor}\label{thm:prodeq-pres}
  A functor $\f\colon \bC\to\bD$ between \ka-ary sites is a morphism
  of sites if and only if it preserves covering families, local binary
  \ka-pre-products, local \ka-pre-equalizers, and local
  \ka-pre-terminal-objects.
\end{cor}
\begin{proof}
  ``Only if'' is clear, so suppose \f preserves the aforementioned
  things.  But then it preserves the construction of nonempty finite
  local \ka-prelimits in \autoref{thm:prodeq}, and hence satisfies
  \autoref{thm:morsite}\ref{item:lf3}.
\end{proof}

We define the very large 2-category \SITEk to consist of \ka-ary
sites, morphisms of sites, and arbitrary natural transformations.
Note that for $\ka\subseteq \ka'$, we have $\SITEk \subseteq
\nSITE_{\ka'}$ as a full sub-2-category.
Since representably flat implies covering-flat, any morphism of sites
in the classical sense is also one in our sense.  We now show,
following~\cite[1.829]{fs:catall}
and~\cite[Prop.~20]{cv:reg-exact-cplt}, that the converse is often
true.

\begin{lemma}\label{thm:morsite-monic}
  If \bD is a site in which all covering families are epic, then any
  covering-flat functor $\f\colon \bC\to\bD$ preserves finite monic
  cones.
\end{lemma}
\begin{proof}
  Suppose $T\colon x\To U$ is a finite monic cone in \bC, where $U$
  has cardinality $n$.  Let \bE be the finite category such that a
  diagram of shape \bE consists of a family of $n$ objects and two
  cones over it.  Let $\bg\colon \bE \to \bC$ be the diagram both of
  whose cones are $T$.  Monicity of $T$ says exactly that $T$ together
  with two copies of $1_x$ is a limit of $\bg$; call this cone $T'$.

  Now suppose $h,k\colon z\toto \f(x)$ satisfy $\f(T) \circ h =
  \f(T)\circ k$.  Then $h$ and $k$ induce a cone $S$ over $\f\bg$.
  Since \f is covering-flat, this cone factors locally through the
  \f-image of some array over \bg, and hence through $\f(T')$.  This
  just means that $z$ admits a covering family $P$ such that $h P = k
  P$; but $P$ is epic, so $h=k$.
\end{proof}

\begin{prop}\label{thm:morsite-flim}
  Suppose \bC has finite limits and all covering families in \bD are
  strong-epic.  Then any covering-flat functor $\f\colon \bC\to\bD$
  preserves finite limits.
\end{prop}
\begin{proof}
  Let $\bg\colon \bE\to\bC$ be a finite diagram and $T\colon x\To
  \bg(\bE)$ a limit cone.  Then $T$ is monic, so by
  \autoref{thm:morsite-monic}, $\f(T)$ is also monic.  But by
  \autoref{thm:morsite}, $\f(T)$ is a local \ka-prelimit of $\f\bg$;
  thus by \autoref{thm:monic-lle} it is a limit.
\end{proof}

Recall that if finite limits exist, every extremal-epic family is
strong-epic.

\begin{cor}\label{thm:fc-morsite}
  If \bC is a finitely complete site and \bD a site in which covering
  families are strong-epic, then $\f\colon \bC\to\bD$ is a morphism of
  sites if and only if it preserves finite limits and covering
  families.
\end{cor}

However, for morphisms between arbitrary sites, our notion is more
general than the usual one.  It is easy to give boring examples of
this.

\begin{example}
  Let \bC be the terminal category, let $\bD$ be the category $(0\to
  1)$ in which the morphism $0\to 1$ is a cover, and let $\f\colon
  \bC\to\bD$ pick out the object $0$.  Then \f is covering-flat, but
  not representably flat, since \bC has a terminal object but \f does
  not preserve it.  Note that \bC and \bD have finite limits and all
  covers in \bD are epic.
\end{example}

Our main reason for introducing the more general notion of morphism of
sites is to state the universal property of exact completion.  It is
further justified, however, by the following observation.

\begin{prop}\label{thm:morsite-sheaves}
  For a small category \bC and a small site \bD, a functor $\f\colon
  \bC\to\bD$ is covering-flat if and only if the composite
  \begin{equation}\label{eq:mslan}
    [\bC\op,\bSet] \xto{\lan_\f} [\bD\op,\bSet] \xto{\ba} \nSh(\bD)
  \end{equation}
  preserves finite limits, where \ba denotes sheafification.
  If \bC is moreover a site and \f a morphism of sites, then
  \[\f^*\colon [\bD\op,\bSet] \to [\bC\op,\bSet]\]
  takes $\nSh(\bD)$ into $\nSh(\bC)$, so \f induces
  a geometric morphism $\nSh(\bD) \to \nSh(\bC)$.
\end{prop}

Note that representable-flatness of \f is equivalent to $\lan_\f$
preserving finite limits.  (This certainly \emph{implies}
that~\eqref{eq:mslan} does so, since \ba always preserves finite
limits.)  This proposition can be proved explicitly, but we will
deduce it from general facts about exact completion in
\S\ref{sec:dense}.

\begin{remark}\label{rmk:premor}
  When \bC and \bD are \emph{locally} \ka-ary sites
  (\autoref{rmk:locsite}), we will consider the notion of a
  \textbf{pre-morphism of sites}, which we define to be a functor
  $\f\colon \bC\to\bD$ preserving finite \emph{connected} local
  \ka-prelimits.  (The name is chosen by analogy with ``pre-geometric
  morphisms'', which are adjunctions between toposes whose left
  adjoint preserves finite connected limits.  The prefix ``pre-'' here
  has unfortunately nothing to do with the ``pre-'' in ``prelimit''.)
  We write \LSITEk for the 2-category of locally \ka-ary sites,
  pre-morphisms of sites, and arbitrary natural transformations.

  Arguing as in \autoref{thm:prodeq-pres} but using
  \autoref{thm:pbeq-connlim}, we see that \f is a pre-morphism of
  sites just when it preserves covering families, local
  \ka-pre-pullbacks, and local \ka-pre-equalizers.  Similarly, we have
  versions of \autoref{thm:morsite-monic} for \emph{nonempty} finite
  monic cones, and of \autoref{thm:morsite-flim} and
  \autoref{thm:fc-morsite} for \emph{connected} finite limits.
\end{remark}

\section{Regularity and exactness}
\label{sec:regularity}

In this section we define the notions of \emph{\ka-ary regular} and
\emph{\ka-ary exact} categories, which are the outputs of our
completion operations.  These are essentially relativizations to \ka
of the notions of ``familially regular'' and ``familially exact''
from~\cite{street:family}.


\subsection{Some operations on arrays}
\label{arrayops}

We begin by defining some special arrays, and operations on
arrays. 
Firstly, for any object $x$, we write $\Delta_x = \{1_x,1_x\} \colon
x\To \{x,x\}$.  Secondly, for any $u,v$ we have a functional array
$\sigma\colon \{u,v\} \To \{v,u\}$ consisting of identities.

Thirdly, suppose $P\colon X\To U$ is a \ka-to-finite array in a
\ka-ary site, $V$ is another finite family of objects, and $F\colon
V\To U$ is a functional array.  For each $x\in X$, let $(F^*X)_x$ be a
local \ka-prelimit of the (finite) diagram consisting of $V$, $U$, all
the morphisms in $F$, and the cone $P|^x$.  Write $(F^*P)_x\colon
(F^*X)_x \To V$ for the cocone built from the projections to $V$.  We
define $F^* X= \bigsqcup_x (F^*X)_x$ and let $F^* P$ be the induced
array $F^* X \To V$.

Fourthly, suppose given a finite family of \ka-to-finite arrays $\{
P_i\colon X_i \To U\}_{1\le i\le n}$.  For each family $\{x_i\}_{1\le
  i\le n}$ with $x_i\in X_i$ for each $i$, consider a local
\ka-prelimit of the (finite) diagram consisting of $U$, all the
objects $x_i$, and all the morphisms $(p_i)_{x_i,u}$.  We write
$\bigwedge_i P_i\colon \bigwedge_i X_i \To U$ for the disjoint union
of these local \ka-prelimits over all families $\{x_i\}_{1\le i\le
  n}$, with its induced array to $U$.

Finally, suppose given arrays $P\colon X\To \{u,v\}$ and $Q\colon Y\To
\{v,w\}$.  For each $x\in X$ and $y\in Y$, let $R^{x y}\colon Z^{x y}
\To \{x,y\}$ be a local \ka-prelimit of the cospan $x \xto{p_{x v}} v
\xot{q_{y v}} y$.  Let $Z = \bigsqcup Z^{x y}$, with induced
functional arrays $R\colon Z \To X$ and $S\colon Z\To Y$, and define
$T\colon Z \To \{u,w\}$ by $t_{z u} = p_{r(z),u} r_z$ and $t_{z w} =
q_{s(z),w} s_z$.  We write $P \times_v Q$ for $T$.

Of course, all of these ``operations'' depend on a choice of local
\ka-prelimits, but the result is unique up to local equivalence in the
sense of \autoref{def:lwmlim}.

\begin{defn}\label{def:cong}
  Let \bC be a \ka-ary site, and $X$ a \ka-ary family of objects of
  \bC.  A \textbf{\ka-ary congruence on $X$} consists of
  \begin{enumerate}[nolistsep]
  \item For each $x_1,x_2\in X$, a \ka-sourced array
    $\Phi(x_1,x_2)\colon \Phi[x_1,x_2]\To
    \{x_1,x_2\}$.\label{item:cong2}
  \item For each $x\in X$ we have $\Delta_x \lle
    \Phi(x,x)$.\label{item:cong3}
  \item For each $x_1,x_2\in X$ we have $\sigma \circ \Phi(x_1,x_2)
    \lle \Phi(x_2,x_1)$.\label{item:cong4}
  \item For each $x_1,x_2,x_3\in X$ we have $\Phi(x_1,x_2)
    \times_{x_2} \Phi(x_2,x_3) \lle \Phi(x_1,x_3)$.\label{item:cong5}
  \end{enumerate}
  We say $\Phi$ is \textbf{strict} if each array $\Phi(x_1,x_2)$ is a
  monic cone.
\end{defn}

\begin{remark}\label{thm:eqv-cong}
  If $\Phi$ is a \ka-ary congruence and we replace each
  $\Phi(x_1,x_2)$ by a locally equivalent \ka-sourced array
  $\Psi(x_1,x_2)$, then $\Psi$ is again a \ka-ary congruence on the
  same underlying family $X$, since all the operations used in
  \autoref{def:cong} respect $\lle$.  In this case we say that $\Phi$
  and $\Psi$ are \textbf{equivalent} congruences.  By
  \autoref{thm:monic-lle}, if covering families are strong-epic, then
  any two equivalent \emph{strict} congruences are in fact isomorphic.
\end{remark}

\begin{lemma}\label{thm:basic-cong}
  Let $X$ and $Y$ be \ka-ary families of objects in a \ka-ary site.
  \begin{enumerate}[nolistsep]
  \item There is a \ka-ary congruence ${\Delta_X}$ on $X$ with
    \[
      {\Delta_X}(x_1,x_2) =
    \begin{cases}
      \Delta_x & x_1 = x_2\\
      \emptyset & x_1 \neq x_2.
    \end{cases}
    \]
  \item If $F\colon X\To Y$ is a functional array and $\Psi$ is a
    \ka-ary congruence on $Y$, we have a \ka-ary congruence $F^* \Psi$
    on $X$ defined by $(F^*\Psi)(x_1,x_2) =
    \{f_{x_1},f_{x_2}\}^*(\Psi(x_1,x_2))$.
  \item If $\{\Phi_i\}_{1\le i \le n}$ is a finite family of \ka-ary
    congruences on $X$, we have a \ka-ary congruence $\bigwedge_i
    \Phi_i$ on $X$ defined by $(\bigwedge_i \Phi_i)(x_1,x_2) =
    \bigwedge_i (\Phi_i(x_1,x_2))$.
  \end{enumerate}
\end{lemma}
\begin{proof}
  This is mostly straightforward verification.  The only possibly
  non-obvious fact is that ${\Delta_X}$ is \ka-ary; its indexing set
  is the subsingleton $\ssing{x_1=x_2}$ from \autoref{rmk:subsing}.
\end{proof}

\begin{defn}
  A \textbf{kernel} of a \ka-to-finite array $P\colon X \To U$ in a
  \ka-ary site is a \ka-ary congruence $\Phi$ on $X$ such that each
  $\Phi(x_1,x_2)$ is a local \ka-prelimit of the diagram
  \[ P|^{x_1} : x_1 \To U \Leftarrow x_2 : P|^{x_2}. \] A
  \textbf{strict kernel} is one constructed using actual limits rather
  than local prelimits.
\end{defn}

The condition to be a kernel is equivalent to saying that $\Phi$ can
be expressed as
\[ \bigwedge_{u\in U} (P|_u)^* {\Delta_{\{u\}}}.
\]
This implies, in particular, that a kernel \emph{is} a \ka-ary
congruence.

Any two kernels of the same array are equivalent in the sense of
\autoref{thm:eqv-cong}.  Conversely, two equivalent congruences are
kernels of exactly the same arrays.  If $\Phi$ is a kernel of $P\colon
X \To U$, then its defining property is that for any $a\colon v\to
x_1$ and $b\colon v\to x_2$ such that $p_{x_1,u} a = p_{x_2,u} b$ for
all $u\in U$, we have $\{a,b\} \lle \Phi(x_1,x_2)$.

\begin{lem}\label{thm:lwker-stker}
  Suppose covering families in \bC are strong-epic.  Then if a strict
  congruence is a kernel of $P$, it is also a strict kernel of $P$.
\end{lem}
\begin{proof}
  By \autoref{thm:monic-lle}.
\end{proof}

We may regard a congruence as a diagram consisting of all the objects
and morphisms occurring in $X$ and all the $\Phi(x_1,x_2)$.  In
particular, we can talk about \emph{colimits} of congruences, which we
also call \emph{quotients}.

\begin{lem}\label{thm:ee-lwkcolim}
  If covering families in \bC are epic, then an effective-epic \ka-ary
  cocone is the colimit of any of its kernels.
\end{lem}
\begin{proof}
  Let $P\colon X \To y$ be effective-epic and \Phi a kernel of it.  It
  suffices to show that if $F$ is a cocone under $\Phi$, then $f_{x_1}
  a = f_{x_2} b$ for any $x_1,x_2\in X$ and $\{a,b\}\colon u \To
  \{x_1,x_2\}$ such that $p_{x_1} a = p_{x_2} b$.  But by construction
  of \Phi, we have a covering family $Q\colon V\To u$ such that
  $\{a,b\} Q \le \Phi(x_1,x_2)$, hence $f_{x_1} a Q = f_{x_2} b Q$.
  Since $Q$ is epic, the claim follows.
\end{proof}


The following lemma and theorem are slight generalizations of results
in~\cite{street:family}
(see also~\cite[Lemma~2.1]{kelly:rel-factsys}).

\begin{lemma}\label{thm:street2}
  Let \bC be a \ka-ary site in which covering families are epic, and
  let $P\colon V\To u$ be a \ka-ary cocone and $Q\colon u\To W$ a cone
  in \bC.
  \begin{enumerate}[nolistsep,label=(\alph*)]
  \item If $Q$ is monic, any (strict) kernel of $P$ is also a (strict)
    kernel of $Q P$ and vice versa.\label{item:street2-1}
  \item Conversely, if $P$ is \ka-universally epic, and some kernel of
    $P$ is also a kernel of $Q P$, then $Q$ is
    monic.\label{item:street2-2}
  \end{enumerate}
\end{lemma}
\begin{proof}
  Part~\ref{item:street2-1} is easy: if $Q$ is monic, then an array
  $R\colon Z\To \{v_1,v_2\}$ satisfies $p_{v_1} R|_{v_1} = p_{v_2}
  R|_{v_2}$ if and only if it satisfies $Q p_{v_1} R|_{v_1} = Q
  p_{v_2} R|_{v_2}$.

  For part~\ref{item:street2-2}, let $a,b\colon x\to u$ be given with
  $Q a = Q b$.  Let $R\colon Y\To x$ be a local \ka-pre-pullback of
  $P$ along $a$, and for each $y\in Y$ let $S^y\colon Z^y \To y$ be a
  local \ka-pre-pullback of $P$ along $b r_y$.  Since $P$ is
  \ka-universally epic, $R$ and each $S^y$ are epic.

  Define $Z = \bigsqcup Z^y$, $S = \bigsqcup S^y$, and $T = R S \colon
  Z \To x$.  Then $T$ is epic, and $a T \le P$ and $b T \le P$.  Thus,
  we have $a T = P G$ and $b T = P H$ for some functional arrays
  $G,H\colon Z\To V$.  But since $Q a = Q b$, we have $Q P G = Q a R =
  Q b R = Q P H$.  Thus for each $z\in Z$, we have $\{g_z,h_z\} \lle
  \Phi(g(z),h(z))$, where \Phi is a kernel of $Q P$.  But since \Phi
  is also a kernel of $P$, we have $P g_z F = P h_z F$ for some
  covering family $F$ of $z$.  Since covering families are epic, $P
  g_z = P h_z$ for all $z$, hence $P G = P H$.  Thus $a T = b T$,
  whence $a=b$ as $T$ is epic.
\end{proof}

\begin{thm}\label{thm:regular}\label{thm:reglex}
  For any category \bC, the following are equivalent.
  \begin{enumerate}[leftmargin=*,noitemsep]
  \item \bC is a \ka-ary site whose covering families are strong-epic,
    and any \ka-to-finite array $R\colon V\To W$ factors as $Q P$,
    where $P\colon V\To u$ is covering and $Q\colon u\To W$ is
    monic.\label{item:reg1}
  \item \bC is a subcanonical \ka-ary site, and any kernel of a
    \ka-to-finite array is also a kernel of some covering
    cocone.\label{item:reg8}
  \item \bC has finite limits, and any \ka-ary cocone $R\colon V \To
    w$ factors as $q P$, where $P\colon V \To u$ is universally
    extremal-epic and $q\colon u \to w$ is monic.\label{item:reg3}
  \item \bC has finite limits, and the strict kernel of any \ka-ary
    cocone is also the strict kernel of some universally
    effective-epic cocone.\label{item:reg4}
  \item \bC is a regular category (in the ordinary sense) and has
    pullback-stable unions of \ka-small families of
    subobjects.\label{item:reg6}
  \end{enumerate}
  Moreover, they imply
  \begin{enumerate}[resume,leftmargin=*]
  \item Every \ka-ary extremal-epic cocone is universally
    effective-epic\label{item:reg5}
  \end{enumerate}
  and in~\ref{item:reg1} and~\ref{item:reg8} the topology is
  automatically the \ka-canonical one.
\end{thm}
\begin{proof}
  We will prove
  \[ \xymatrix@-1pc{
    \ref{item:reg1} \ar@{=>}[r] &
    \ref{item:reg3} \ar@{<=>}[d] \ar@{=>}[r]  \ar@{=>}[dr] &
    \ref{item:reg4} \ar@{=>}[r] &
    \ref{item:reg8} \ar@{=>}[r] &
    \ref{item:reg1} \\
    & \ref{item:reg6} & \ref{item:reg5}.
  }
  \]

  For \ref{item:reg1}$\Rightarrow$\ref{item:reg3}, it suffices to
  construct finite limits.  Given a finite diagram $\bg\colon \bD\to
  \bC$, let $T$ be a \ka-prelimit of $\bg$, and apply~\ref{item:reg1}
  to the array $T\colon U \To \bg(\bD)$.  The resulting array $Q$ is a
  cone over $\bg$ since $P$ is epic, and a local \ka-prelimit since
  $P$ is a covering family.  Hence, by \autoref{thm:monic-lle} it is a
  limiting cone.

  Now supposing~\ref{item:reg3}, since $1\in \ka$, \bC has
  pullback-stable (extremal-epi, mono) factorizations.  This is one
  definition of a regular category in the ordinary sense, so \bC is
  regular.  Moreover, applying the factorization of~\ref{item:reg3} to
  a \ka-small family of subobjects supplies a pullback-stable union;
  hence~\ref{item:reg3}$\Rightarrow$\ref{item:reg6}.

  Conversely, supposing~\ref{item:reg6} and given a \ka-small cocone
  $F\colon V\To w$, we can first factor each $f_v$ using the
  (extremal-epi, mono) factorization in a regular category, and then
  take the union of the resulting \ka-small family of subobjects of
  $w$.  This gives the desired factorization, so
  that~\ref{item:reg6}$\Rightarrow$\ref{item:reg3}.

  Now we assume~\ref{item:reg3} and prove~\ref{item:reg5}.  Firstly,
  if $R$ is extremal-epic, then it factors as $R = q P$ where $P$ is
  universally extremal-epic and $q$ is monic; hence $q$ is an
  isomorphism and so $R$, like $P$, is universally extremal-epic.
  Thus it suffices to show that every \ka-small extremal-epic family
  is effective-epic.

  Let $P\colon V \To u$ be a \ka-ary extremal-epic cocone with a
  kernel $\Phi$, and let $R\colon V \To w$ be a cocone under $\Phi$.
  Then the induced cocone $(P,R)\colon V \To u\times w$ factors as
  $(a,b) P'$, where $P'\colon V \To t$ is (universally) extremal-epic
  and $(a,b)\colon t \to u\times w$ is monic.  Since $P$ factors
  through $a$, $a$ must be extremal-epic.

  Now since $\Phi$ is a kernel of $P$ and $R$ is a cocone under
  $\Phi$, it follows that $\Phi$ is also a kernel of $(P,R)$.  And
  since $(a,b)$ is monic, by \autoref{thm:street2}\ref{item:street2-1}
  $\Phi$ is also a kernel of $P'$.  But $a P' = P$ and $P'$ is
  universally extremal-epic, hence universally epic; so by
  \autoref{thm:street2}\ref{item:street2-2}, $a$ is monic.  Since it
  is both extremal-epic and monic, it must be an isomorphism, and then
  the composite $b a^{-1}$ provides a factorization of $R$ through
  $P$.  This factorization is unique, since $P$ is extremal-epic and
  hence epic.  Thus~\ref{item:reg5} holds.

  Now assuming~\ref{item:reg3}, hence also~\ref{item:reg5}, we
  prove~\ref{item:reg4}.  If $\Phi$ is a strict kernel of $R$, then
  write $R=q P$ with $q$ monic and $P$ universally extremal-epic.  By
  \autoref{thm:street2}\ref{item:street2-1} $\Phi$ is also a strict
  kernel of $P$, and by~\ref{item:reg5} $P$ is universally
  effective-epic, so~\ref{item:reg4} holds.

  Now suppose \ref{item:reg4} and give \bC its \ka-canonical topology.
  Let \Phi be a kernel of a \ka-to-finite array $R\colon V\To W$, and
  let \Psi be the strict kernel of $R$.  Then \Psi is also the strict
  kernel of the induced cocone $V \To \prod_{w\in W} w$, so
  by~\ref{item:reg4}, \Psi is the strict kernel (hence a kernel) of
  some universally effective-epic cocone, which is covering in the
  \ka-canonical topology.  Since \Phi and \Psi are equivalent
  congruences, \Phi is also a kernel of this cocone; thus
  \ref{item:reg4}$\Rightarrow$\ref{item:reg8}.

  To complete the circle, assume~\ref{item:reg8}, and suppose given a
  \ka-to-finite array $R$.  Let $\Phi$ be a kernel of $R$, and let $P$
  be universally effective-epic and have $\Phi$ also as its kernel.
  Then by \autoref{thm:ee-lwkcolim}, $P$ is the colimit of \Phi, so
  $R$ factors through it as $Q P$, and by
  \autoref{thm:street2}\ref{item:street2-2} $Q$ is monic;
  thus~\ref{item:reg1} holds.

  Finally, we show that the topology in~\ref{item:reg1}
  and~\ref{item:reg8} must be \ka-canonical.  It is subcanonical by
  definition in~\ref{item:reg8} and by~\ref{item:reg5}
  in~\ref{item:reg1}, so it suffices to show that any universally
  effective-epic cocone is covering.  But by~\ref{item:reg1} such a
  cocone factors as a covering family followed by a monomorphism, and
  since effective-epic families are extremal-epic, the monic must be
  an isomorphism.
\end{proof}

\begin{defn}
  We say that a category \bC is \textbf{\ka-ary regular}, or
  \textbf{\ka-ary coherent}, if it satisfies the equivalent conditions
  of \autoref{thm:regular}.
\end{defn}

By \autoref{thm:regular}\ref{item:reg6}, \ka-ary regularity
generalizes a number of more common definitions.
\begin{itemize}[noitemsep]
\item \bC is \un-ary regular iff it is regular in the usual sense.
\item \bC is $\{0,1\}$-ary regular iff it is regular and has a strict
  initial object.
\item \bC is \om-ary regular iff it is coherent.
\item \bC is $\om_1$-ary regular iff it is countably-coherent, a.k.a.\
  \si-coherent.
\item \bC is \KA-ary regular iff it is infinitary-coherent (a.k.a.\
  geometric, although sometimes geometric categories are also required
  to be well-powered).
\end{itemize}

By \autoref{thm:regular}\ref{item:reg5}, in a \ka-ary regular
category, a \ka-ary cocone is extremal-epic if and only if it is
strong-epic, if and only if it is effective-epic, and in all cases it
is automatically universally so.  These \ka-ary cocones form the
\ka-canonical topology on a \ka-ary regular category, which we may
also call the \textbf{\ka-regular} or \textbf{\ka-coherent topology}.

Similarly, a functor between \ka-ary regular categories is a morphism
of sites (relative to the \ka-regular topologies) just when it
preserves finite limits and \ka-small extremal-epic families; we call
such a functor \textbf{\ka-ary regular}.  We write $\nREG_\ka$ for the
full sub-2-category of $\SITEk$ whose objects are the \ka-ary regular
categories with their \ka-canonical topologies.

\begin{remark}
  Let \nLEX denote the 2-category of finitely complete categories and
  finitely continuous functors, and \nWLEX that of categories with
  weak finite limits and weak-finite-limit--preserving functors (i.e.\
  representably flat functors).  Equipping such categories with their
  trivial \ka-ary topologies, we can regard \nLEX and \nWLEX as full
  sub-2-categories of \SITEk for any \ka.  However, the embedding
  $\nREG_\ka \into \nSITE_{\ka}$ does \emph{not} factor through the
  embeddings of $\nLEX$ or $\nWLEX$.  This is why our universal
  property for regular and exact completions will look simpler than
  that of~\cite{cv:reg-exact-cplt}.
\end{remark}

\begin{remark}
  Our notion of \ka-ary regular category is a different ``infinitary
  generalization'' of regularity than that considered
  in~\cite{ht:free-regex,lack:exreg-inf}; the latter instead adds to
  (unary) regularity the existence and compatibility of \ka-ary
  \emph{products}.
\end{remark}

\begin{remark}\label{rmk:locreg}
  Following~\cite[A3.2.7]{ptj:elephant}, by a \textbf{locally \ka-ary
    regular category} we will mean a category \bC with finite
  \emph{connected} limits and stable image factorizations for arrays
  with \ka-ary domain and \emph{nonempty} finite codomain.  Note that
  the construction of kernels of such arrays requires only connected
  finite (local \ka-pre-)limits.  An analogue of \autoref{thm:regular}
  is true in this case, although in the absence of finite products,
  in~\ref{item:reg3} and~\ref{item:reg4} we need to assert at least
  factorization for \ka-to-binary arrays as well as for cocones.
\end{remark}


\begin{cor}\label{thm:eqv-monic}\label{thm:reg-strictcong}
  In a \ka-ary regular category (with its \ka-regular topology):
  \begin{enumerate}[nolistsep]
  \item Every \ka-to-finite array is locally equivalent to a monic cone.
  \item Every \ka-ary congruence is equivalent to a strict one.\qed
  \end{enumerate}
\end{cor}

Furthermore, in a \ka-ary regular category with its \ka-regular
topology, for monic cones $Q\colon x \To U$ and $R\colon y\To U$, we
have $Q\lle R$ if and only if $Q \le R$.  Thus, in this case a strict
\ka-ary congruence on $X$ consists of
\begin{enumerate}[noitemsep]
\item Monic spans $x_1 \ot \Phi(x_1,x_2)\to x_2$.
\item Each diagonal span $x \ot x \to x$ factors through $x \ot
  \Phi(x,x) \to x$.
\item Each $x_1 \ot \Phi(x_1,x_2) \to x_2$ is (isomorphic to) the
  reversal of $x_2 \ot \Phi(x_2,x_1) \to x_1$.
\item Each $x_{1} \ot \Phi(x_1,x_2) \times_{x_2} \Phi(x_2,x_3) \to
  x_3$ factors through $x_1 \ot \Phi(x_1,x_3) \to x_3$.
\end{enumerate}
When $X$ is a singleton, this is just an \emph{equivalence relation}
in \bC in the usual sense.

\begin{thm}\label{thm:exact}
  For a category \bC, the following are equivalent.
  \begin{enumerate}[leftmargin=*,nolistsep]
  \item \bC is a subcanonical \ka-ary site in which any \ka-ary
    congruence is a kernel of some covering cocone.\label{item:ex1}
  \item \bC has finite limits, and every strict \ka-ary congruence is
    a strict kernel of some universally effective-epic
    cocone.\label{item:ex1a}
  \item \bC is \ka-ary regular, and every strict \ka-ary congruence is
    a strict kernel of some \ka-ary cocone.\label{item:ex2a}
  \item \bC is (Barr-)exact in the ordinary sense (a.k.a.\ effective
    regular), and has disjoint and universal coproducts of \ka-small
    families of objects.\label{item:ex3}
  \end{enumerate}
  Moreover, in~\ref{item:ex1} the topology is automatically the
  \ka-canonical one.
\end{thm}
\begin{proof}
  By \autoref{thm:regular}\ref{item:reg8}, condition~\ref{item:ex1}
  implies \bC is \ka-ary regular, hence has finite limits.
  \autoref{thm:lwker-stker} then implies the rest of~\ref{item:ex1a}.
  \autoref{thm:regular}\ref{item:reg4} immediately implies
  \ref{item:ex1a}$\Leftrightarrow$\ref{item:ex2a}.  In this case, by
  \autoref{thm:reg-strictcong} any \ka-ary congruence is equivalent to
  a strict one, so since equivalent congruences are kernels of the
  same cocones,~\ref{item:ex1} holds.

  Now we assume~\ref{item:ex2a} and prove~\ref{item:ex3}.  Since a
  singleton strict congruence is an equivalence
  relation,~\ref{item:ex2a} implies that \bC is exact in the ordinary
  sense.  Moreover, for any \ka-ary family of objects $X$, the
  congruence ${\Delta_X}$ from \autoref{thm:basic-cong} is equivalent
  to a strict one, and this strict congruence is a kernel of a
  universally effective-epic cocone just when $X$ has a disjoint and
  universal coproduct.  This
  shows~\ref{item:ex2a}$\Rightarrow$\ref{item:ex3}.

  Conversely, it is well-known that~\ref{item:ex3} implies
  \autoref{thm:regular}\ref{item:reg6}, so that \bC is \ka-ary
  regular.  Moreover, if \bC satisfies~\ref{item:ex3}, then any strict
  \ka-ary congruence \Phi on $X$ gives rise to an internal equivalence
  relation on $\sum_{x\in X} x$ by taking coproducts of the
  $\Phi(x_1,x_2)$.  The quotient of this equivalence relation then
  admits a universally effective-epic cocone from $X$ of which \Phi is
  the strict kernel.
  Thus,~\ref{item:ex3}$\Rightarrow$\ref{item:ex2a}.
\end{proof}

\begin{defn}
  We say that a category \bC is \textbf{\ka-ary exact}, or a
  \textbf{\ka-ary pretopos}, if it satisfies the equivalent conditions
  of \autoref{thm:exact}.
\end{defn}

As before, \autoref{thm:exact}\ref{item:ex3} implies that \ka-ary
exactness generalizes a number of more common definitions.
\begin{itemize}[noitemsep]
\item \bC is \un-ary exact iff it is exact in the usual sense.
\item \bC is $\{0,1\}$-ary exact iff it is exact and has a strict
  initial object.
\item \bC is \om-ary exact iff it is a pretopos.
\item \bC is $\om_1$-ary exact iff it is a \si-pretopos.
\item \bC is \KA-ary exact iff it is an
  infinitary-pretopos\footnote{Also called an $\infty$-pretopos, but
    we avoid that term due to potential confusion with the very
    different ``$\infty$-toposes'' of~\cite{lurie:higher-topoi}.}
  (again, minus the occasional requirement of well-poweredness).  This
  means it satisfies all the conditions of Giraud's theorem except
  possibly the existence of a generating set.
\end{itemize}

\noindent
Let $\nEX_{\ka} \subset \nREG_{\ka}$ be the full sub-2-category
consisting of the \ka-ary exact categories.

\begin{remark}
  If $\om \in \ka$, then any \ka-ary exact category admits all
  \ka-small colimits.  The proof is well-known: it has \ka-small
  coproducts, and $\om \in \ka$ implies that any coequalizer generates
  an equivalence relation with the same quotient; hence it also has
  all coequalizers.

  Similarly, in this case a functor between \ka-ary exact categories
  is \ka-ary regular if and only if it preserves finite limits and
  \ka-small colimits.  In particular, a functor between Grothendieck
  toposes is \KA-ary regular if and only if it preserves finite limits
  and small colimits, i.e.\ if and only if it is the left-adjoint part
  of a geometric morphism.  Hence $\nEX_\KA$ contains, as a full
  sub-2-category, the opposite of the usual 2-category of Grothendieck
  toposes.
\end{remark}

\begin{remark}\label{rmk:locex}
  We call a category \textbf{locally \ka-ary exact} if it is locally
  \ka-ary regular and every \ka-ary strict congruence is a kernel.  An
  analogue of \autoref{thm:exact} is then true.
\end{remark}

\section{Framed allegories}
\label{sec:fr-alleg-new}


We now recall the basic structure used in the ``relational''
construction of the exact completion of a regular category
from~\cite{fs:catall} (see also~\cite[\S A3]{ptj:elephant}), which
will be the basis for our construction of the \ka-ary exact
completion.

\begin{defn}\label{def:allegory}
  A \textbf{\ka-ary allegory} is a 2-category \cA such that
  \begin{enumerate}[nolistsep]
  \item Each hom-category of \cA is a poset with binary meets and
    \ka-ary joins;\label{item:a1}
  \item Composition and binary meets preserve \ka-ary joins in each
    variable;\label{item:a2}
  \item \cA has a contravariant identity-on-objects involution
    $(-)\o\colon \cA\op\to\cA$; and\label{item:a3}
  \item The modular law $\psi\phi \wedge \chi \le (\psi \wedge
    \chi\phi\o)\phi$ holds.\label{item:a4}
  \end{enumerate}
\end{defn}

As with categories, all allegories will be be moderate, but not
necessarily locally small.  The basic example consists of binary
relations in a \ka-ary regular category.  This will be a special case
of our more general construction for \ka-ary sites in
\S\ref{sec:relations-ka-ary}.

\begin{remark}
  Our \un-ary allegories are the \emph{allegories}
  of~\cite{fs:catall,ptj:elephant}.  Similarly, our \om-ary allegories
  and \KA-ary allegories are called \emph{distributive} and
  \emph{locally complete} allegories, respectively,
  in~\cite{fs:catall}.  The \emph{union allegories} and
  \emph{geometric allegories} of~\cite{ptj:elephant} are almost the
  same, but do not require binary meets to distribute over \ka-ary
  joins (this is automatic for tabular allegories).
\end{remark}

Note that $(-)\o$ reverses the direction of morphisms, but preserves
the ordering on hom-posets.  Thus the modular law is equivalent to its
dual $\psi\phi \wedge \chi \le \psi(\phi \wedge \psi\o\chi)$.  We
follow~\cite{ptj:elephant} by writing morphisms in an allegory as
$\phi\colon x\rto y$.

A morphism in an allegory is called a \textbf{map} if it has a right
adjoint.  In an allegory, we write a map as $f\colon x\lto y$.  For
completeness, we reproduce the following basic proofs.

\begin{lem}\label{thm:maps}\ 
  \begin{enumerate}[nolistsep]
  \item The right adjoint of a map $f\colon x\lto y$ is necessarily
    $f\o$.\label{item:m1}
  \item The maps in an allegory \cA are discretely ordered (i.e.\
    $f\le g$ implies $f=g$).\label{item:m2}
  \end{enumerate}
\end{lem}
\begin{proof}
  For~\ref{item:m1}, since $f\dashv \phi$ implies $\phi\o \dashv f\o$,
  $\phi\o$ is a map, say $g$.  Then $1_x = 1_x \meet g\o f \le (f\o
  \meet g\o)f$ by the modular law, while $f(f\o\meet g\o) \le f g\o
  \le 1_y$.  Thus $f\dashv (f\o\meet g\o)$ also.  So $f\o\meet g\o =
  g\o$, i.e.\ $g\o \le f\o$, and by symmetry $f\o\le g\o$ and so
  $f\o=g\o$.

  For~\ref{item:m2}, if $f\le g$ then $f\o\le g\o$, so $g \le g f\o f
  \le g g\o f \le f$ and hence $f=g$.
\end{proof}

\noindent
We denote by $\nMap(\cA)$ the category whose morphisms are the maps in
\cA.

\begin{remark}\label{rmk:frobenius}
  Undoubtedly, the most obscure part of \autoref{def:allegory} is the
  modular law.  There are similar structures, such as the
  \emph{bicategories of relations} of~\cite{cw:cart-bicats-i}, which
  replace this with a more familiar-looking ``Frobenius'' condition.
  In fact, a bicategory of relations is the same as a ``unital and
  pretabular'' allegory~\cite{walters:relations,nlab:bicat-rel}.
  Unfortunately, the allegories we need are not unital or pretabular,
  but they do satisfy a weaker condition which also allows us to
  rephrase the modular law as a Frobenius condition.

  On the one hand, notice that if $f$ is a map, then for any
  $\phi,\chi$ we have
  \begin{equation}
    \psi f \wedge \chi \le \psi f \wedge \chi f\o f
    = (\psi \wedge \chi f\o)f\label{eq:modular-map}
  \end{equation}
  since precomposition with $f$, being a right adjoint, preserves
  meets.  Thus, when $\phi=f$ is a map, the modular law is automatic.

  On the other hand, if $\phi = f\o$ is the right adjoint of a map
  $f$, we have
  \begin{equation}\label{eq:frobenius}
    (\psi \wedge \chi f) f\o \le (\psi f\o f \wedge \chi f) f\o
    = (\psi f\o \meet \chi)f  f\o \le \psi f\o \meet \chi
  \end{equation}
  which is the \emph{reverse} of the modular law.  Moreover, asking
  that~\eqref{eq:frobenius} be an equality is a familiar form of
  ``Frobenius law'' for the adjunction $(-\circ f\o) \dashv (-\circ
  f)$.
\end{remark}

\begin{lemma}\label{thm:frob-mod}
  Suppose \cA is a 2-category satisfying
  \ref{def:allegory}\ref{item:a1}--\ref{item:a3} and also the
  following.
  \begin{enumerate}[nolistsep,label=(\alph*)]
  \item If a morphism $f$ in \cA has a right adjoint, then that
    adjoint is $f\o$.\label{item:fm1}
  \item For any map $f$ in \cA, the inequality~\eqref{eq:frobenius} is
    an equality.
  \item Every morphism $\phi\colon x\rto y$ in \cA can be written as
    $\phi = \bigvee_u g_u f_u\o$, for some \ka-ary cocones $F\colon
    U\To x$ and $G\colon U\To y$ consisting of maps.\label{item:fm4}
\end{enumerate}
Then \cA satisfies the modular law, hence is a \ka-ary allegory.
\end{lemma}
\begin{proof}
  Given $\phi\colon x\rto y$, write $\phi = \bigvee_u g_u f_u\o$ as
  in~\ref{item:fm4}.  Then for any $\psi,\chi$ we have
  \begin{multline*}
    \psi \phi \wedge \chi =
    \psi \left(\bigvee_u g_u f_u\o\right) \wedge \chi =
    \bigvee_u \Big(\psi g_u f_u\o \wedge \chi\Big) \\ =
    \bigvee_u \Big(\psi g_u \wedge \chi f_u\Big) f_u \o \le
    \bigvee_u \Big(\psi \wedge \chi f_u g_u\o\Big) g_u f_u\o \le
    \bigvee_{u,u'} \Big(\psi \wedge \chi f_u g_u\o\Big) g_{u'} f_{u'}\o \\ =
    \left(\psi \wedge \chi\left(\bigvee_u f_u g_u\o\right)\right)
    \left(\bigvee_{u'} g_{u'} f_{u'}\o \right) =
    (\psi \wedge \chi \phi\o)\phi
  \end{multline*}
  which is the modular law for $\phi$.
\end{proof}

As observed in~\cite{ptj:elephant}, it is technically even possible to
omit $(-)\o$ from the structure, since it is determined by
\autoref{thm:frob-mod}\ref{item:fm1} and~\ref{item:fm4}.  But it seems
difficult to ensure that an operation defined in this way is
well-defined and functorial, and in all naturally-occurring examples
the involution is easy to define directly.  However, this observation
does imply that a 2-functor between \ka-ary allegories of this sort
which preserves local \ka-ary joins must automatically preserve the
involution.

\begin{defn}
  A \textbf{\ka-ary allegory functor} is a 2-functor preserving the
  involution and the binary meets and \ka-ary joins in the hom-posets.
  An \textbf{allegory transformation} is an \emph{oplax} natural
  transformation (i.e.\ we have $\alpha_y \circ F(\phi) \le G(\phi)
  \circ \alpha_x$ for $\phi\colon x\rto y$) whose components
  $\alpha_x$ are all maps.
\end{defn}

Since maps in an allegory are discretely ordered, an allegory
transformation is strictly natural on maps ($\alpha_y \circ F(f) =
G(f) \circ \alpha_x$).  Similarly, there are no ``modifications''
between allegory transformations.  We write \ALLk for the 2-category
of \ka-ary allegories.

The central classical theorem about allegories is that the ``binary
relations'' construction induces an equivalence from $\REGk$ to the
2-category of ``unital and tabular'' \ka-ary allegories.  One can then
identify those allegories that correspond to exact categories and
construct the exact completion in the world of allegories.  We aim to
proceed analogously, but for this we need a class of allegories which
correspond to \ka-ary sites in the same way that unital tabular ones
correspond to regular categories.  This requires generalizing the
allegory concept slightly (essentially to allow for non-subcanonical
topologies; see \autoref{thm:chordate1}).

\begin{defn}
  A \textbf{framed \ka-ary allegory} \lA consists of a \ka-ary
  allegory \cA, a category $\nTMap(\lA)$, and a bijective-on-objects
  functor $J\colon \nTMap(\lA)\to \nMap(\cA)$.
\end{defn}

\noindent A framed allegory can be thought of as:
\begin{itemize}[noitemsep]
\item a proarrow equipment~\cite{wood:proarrows-i} whose proarrows
  form an allegory, or
\item a framed bicategory~\cite{shulman:frbi} whose horizontal
  bicategory is an allegory, or
\item an \sF-category~\cite{ls:limlax} whose loose morphisms form an
  allegory, and where every tight morphism has a loose right adjoint.
\end{itemize}

Following the terminology of~\cite{ls:limlax}, we call the morphisms
of $\nTMap(\lA)$ \textbf{tight maps}, and write them as $f\colon x\to
y$.  We refer to maps in the underlying allegory \cA as \textbf{loose
  maps}.  The functor $J$ takes each tight map $f\colon x\to y$ to a
loose map which we write as $f\sb\colon x \lto y$.  We also write
$f\pb = (f\sb)\o\colon y\rto x$ for the right adjoint of the
underlying loose map of a tight map $f$.  Note that a given loose map
may have more than one ``tightening''.

We call \lA \textbf{chordate} if $J\colon \nTMap(\lA)\to \nMap(\cA)$
is an isomorphism (i.e.\ every loose map has a unique tightening), and
\textbf{subchordate} if $J$ is faithful (i.e.\ a loose map has at most
one tightening).  This does not quite match the terminology
of~\cite{ls:limlax}, where ``chordate'' means that every
\emph{morphism} is tight, but it is ``as chordate as a framed allegory
can get'' since we require all tight morphisms to be maps.  Evidently,
chordate framed allegories can be identified with ordinary allegories.

A \textbf{framed \ka-ary allegory functor} $\lA \to \lB$ consists of a
\ka-ary allegory functor $\cA\to\cB$, together with a functor
$\nTMap(\lA) \to \nTMap(\lB)$ making the following square commute:
\[\vcenter{\xymatrix{
    \nTMap(\lA)\ar[r]\ar[d]_J &
    \nTMap(\lB)\ar[d]^J\\
    \nMap(\cA)\ar[r] &
    \nMap(\cB).
  }}\]
Similarly, a \textbf{framed \ka-ary allegory transformation} consists
of a \ka-ary allegory transformation and a compatible natural
transformation on tight maps.  We obtain a 2-category $\FALLk$ of
framed \ka-ary allegories, with a full inclusion $\ALLk\into \FALLk$
onto the chordate ones.  This has a left 2-adjoint which forgets about
the tight maps; after composing it with the inclusion we call this the
\textbf{chordate reflection}.  Similarly, there is a
\textbf{subchordate reflection} which declares two tight maps to be
equal whenever their underlying loose maps are.

\subsection{Relations in \ka-ary sites}
\label{sec:relations-ka-ary}

To simplify matters, we consider first the framed allegories that
correspond to \emph{locally} \ka-ary sites, then add conditions to
characterize the \ka-ary ones.  Thus, let \bC be a locally \ka-ary
site.  Recall this means that it is weakly \ka-ary and has finite
\emph{connected} local \ka-prelimits.  We define a framed allegory
$\lrelk(\bC)$ as follows.

The objects of $\lrelk(\bC)$ are those of \bC.  The hom-poset
$\lrelk(\bC)(x,y)$ of its underlying allegory is the poset reflection
of the preorder of \ka-sourced arrays over $\{x,y\}$, with the
relation $\lle$ of local refinement from \autoref{def:lwmlim}.  This
has binary meets by the construction $\bigwedge_i$ from
\S\ref{arrayops}, and \ka-ary joins by taking disjoint unions of
domains of arrays.  Composition is the operation $P\times_v Q$ from
\S\ref{arrayops}, and the involution is obvious.

Finally, we define $\nTMap(\lrelk(\bC)) = \bC$.  For each $f\colon
x\to y$ in \bC, we have a \ka-ary array $\{1_x,f\}\colon \{x\} \to
\{x,y\}$, which has a right adjoint $\{f,1_x\}\colon \{x\} \to
\{y,x\}$, giving a functor $J\colon \bC \to \nMap(\lrelk(\bC))$.  It
is easy to prove that composition and binary meets preserve \ka-ary
joins and that the modular law holds (perhaps by using
\autoref{thm:frob-mod}).

\begin{examples}\label{rmk:relk-egs}
  Suppose \bC is a locally \ka-ary regular category with its \ka-ary
  regular topology.  By \autoref{thm:eqv-monic}, the underlying
  allegory of $\lrelk(\bC)$ is isomorphic to the usual allegory of
  relations in \bC.  It is well-known that \bC can be identified with
  the category of maps in this allegory, so $\lrelk(\bC)$ is chordate
  (see also \autoref{thm:chordate1}).

  On the other hand, if \bC is a small \KA-ary site, the underlying
  allegory of $\lRel_{\KA}(\bC)$ is the bicategory used
  in~\cite{walters:sheaves-cauchy-2} to construct $\nSh(\bC)$.
  Finally, if \bC has pullbacks and a trivial unary topology, the
  underlying allegory of $\lRel_{\un}(\bC)$ is the allegory used
  in~\cite[\S7]{carboni:free-constr} and~\cite[A3.3.8]{ptj:elephant}
  to construct the regular completion of \bC.
\end{examples}

We now aim to isolate the essential properties of $\lrelk(\bC)$.
Recall the notions of matrix, array, and sparse array from
\S\ref{sec:preliminaries}.  Of course, arrays in an allegory inherit a
pointwise ordering.  Moreover, if $\Phi\colon X\Rto Y$ is a matrix in
a \ka-ary allegory such that each set $\phi_{x y}$ is \ka-small, then
it has a join $\bigvee \Phi$ which is an array $X\Rto Y$.  In
particular, if $\Phi\colon X\Rto Y$ and $\Psi\colon Y\Rto Z$ are
arrays in a \ka-ary allegory and $Y$ is \ka-ary, then the composite
matrix $\Psi\Phi$ satisfies this hypothesis, hence has a join
$\bigvee(\Psi\Phi)$.  Similarly, $1_X \colon X\Rto X$ also always
satisfies this condition when $X$ is \ka-ary.

We first observe that we can recover the topology of \bC from $\lrelk(\bC)$.

\begin{lem}\label{thm:cov-detect}
  Let $P\colon V\To u$ be a \ka-ary cocone in a locally \ka-ary site
  \bC.  Then $P$ is a covering family if and only if $\bigvee(P\sb
  P\pb) = 1_u$ in $\lrelk(\bC)$.
\end{lem}
\begin{proof}
  The inequality $\bigvee(P\sb P\pb) \le 1_u$ is always true since
  $P\sb$ consists of maps.  By definition of $\lrelk(\bC)$, the other
  inequality $1_u\le \bigvee(P\sb P\pb)$ asserts that $u$ admits a
  covering family factoring through $P$, which is equivalent to $P$
  being covering.
\end{proof}

\begin{defn}\label{thm:wkktab}
  A \ka-ary framed allegory \lA is \textbf{weakly \ka-tabular} if
  \begin{enumerate}[leftmargin=*,nolistsep]
  \item Every morphism $\phi$ can be written as $\phi = \bigvee(F\sb
    G\pb)$, where $F$ and $G$ are \ka-ary cocones of tight
    maps.\label{item:relk1}
  \item If $f,g\colon x\toto y$ are parallel tight maps such that
    $f\sb=g\sb$, then there exists a \ka-ary cocone $P\colon U\To x$
    of tight maps such that $f P = g P$ and $\bigvee(P\sb P\pb) =
    1_x$.\label{item:relk3}
  \end{enumerate}
\end{defn}

\begin{prop}
  $\lrelk(\bC)$ is weakly \ka-tabular, for any locally \ka-ary site \bC.
\end{prop}
\begin{proof}
  By definition, a morphism $\phi\colon x\rto y$ in $\lrelk(\bC)$ is
  an array $H\colon Z\To \{x,y\}$.  It is easy to see that then $\phi
  = \bigvee ((H|_y)\sb (H|_x)\pb)$, showing~\ref{item:relk1}.

  For~\ref{item:relk3}, if we have $f,g\colon x\toto y$ with $f\sb =
  g\sb$, then $f\sb\le g\sb$, and so by definition $\{1_x,f\} \lle
  \{1_x,g\}$.  Thus, we have a covering family $P\colon V\To x$ and a
  cocone $H\colon V\To x$ such that $1_x P = 1_x H$ and $f P = g H$.
  Hence $H=P$ and thus $f P = g P$.
\end{proof}

We refer to $F$ and $G$ as in~\ref{item:relk1} as a \textbf{weak
  \ka-tabulation} of $\phi$.  Note that~\ref{item:relk1} implies that
any array of morphisms $\Phi\colon U \Rto V$ can be written as $\Phi =
\bigvee (F\sb G\pb)$, where $F$ and $G$ are functional arrays of tight
maps.  If $U$ and $V$ are \ka-ary, then the common source of $F$ and
$G$ will also be \ka-ary.

We now aim to prove that any weakly \ka-tabular \lA is of the form
$\lrelk(\bC)$.  \autoref{thm:cov-detect} suggests the following
definition.

\begin{defn}
  A \ka-ary cocone $P\colon V\To u$ of tight maps in a \ka-ary framed
  allegory is \textbf{covering} if $\bigvee(P\sb P\pb) = 1_u$.
\end{defn}

In the categorified context of~\cite{street:cauchy-enr}, such families
were called \emph{cauchy dense}.  The following fact is useful for
recognizing covering families.

\begin{lem}\label{thm:entire-detect}
  Suppose $\Phi \colon x\Rto V$ is a \ka-ary cone in a \ka-ary
  allegory such that $\Phi = \bigvee (F G\o)$, where $G\colon W\To x$
  is a \ka-ary cocone of maps and $F\colon W \To V$ is a functional
  array of maps.  Then $1_x \le \bigvee(\Phi\o\Phi)$ if and only if
  $1_x = \bigvee(G G\o)$.
\end{lem}
\proof
  If $1_x = \bigvee(G G\o)$, then
  \begin{multline*}
    1_x = \bigvee(G G\o) = \bigvee_{w} \big(g_w g_w\o\big)
    \le \bigvee_{w} \big(g_w f_w\o f_w g_w\o\big)
    \le \bigvee_{w,w'\atop f(w)=f(w')} \big(g_{w'} f_{w'}\o f_{w} g_{w}\o\big)\\
    = \bigvee_v \left(\bigvee_{f(w')=v} (g_{w'} f_{w'}\o)\right)
    \left(\bigvee_{f(w)=v} (f_{w} g_{w}\o)\right)
    = \bigvee_v (\phi_v\o \phi_v)
    = \bigvee (\Phi\o\Phi).
  \end{multline*}
  Conversely, suppose $1_x \le \bigvee(\Phi\o\Phi)$.  Then using the
  modular law, we have
  \begin{align*}
    1_x &= 1_x \wedge \bigvee(\Phi\o\Phi)
    = \bigvee_v \Big( 1_x \wedge \phi_v\o \phi_v \Big)\\
    &= \bigvee_v \left( 1_x \wedge \left(\bigvee_{f(w')=v}
        (g_{w'} f_{w'}\o)\right)
      \left(\bigvee_{f(w)=v} (f_{w} g_{w}\o)\right) \right)\\
    &= \bigvee_{w,w'\atop f(w)=f(w')}
    \Big( 1_x \wedge g_{w'} f_{w'}\o f_{w} g_{w}\o\Big)\\
    &\le \bigvee_{w,w'\atop f(w)=f(w')}
    \Big( g_{w} \wedge g_{w'} f_{w'}\o f_{w}\Big) g_{w}\o
    \le \bigvee_{w} \big( g_{w} g_{w}\o \big)
    = \bigvee (G G\o).\qedhere
  \end{align*}

\begin{cor}\label{thm:entire-detect2}
  If $\Phi\colon X\Rto V$ is a \ka-targeted array in a \ka-ary framed
  allegory and $\Phi = \bigvee (F\sb G\pb)$, with $F$ and $G$ \ka-ary
  functional arrays of tight maps, then $\bigvee 1_X \le \bigvee
  (\Phi\o\Phi)$ if and only if $G$ is a covering family.\qed
\end{cor}

\begin{thm}\label{thm:ktab-top}
  For any weakly \ka-tabular \lA, the covering families as defined
  above form a weakly \ka-ary topology on $\nTMap(\lA)$.
\end{thm}
\begin{proof}
  Firstly, it is clear that $\{1_u\colon u\to u\}$ is covering.

  Secondly, if $P\colon V\To u$ is covering, and for each $v\in V$ we
  have a covering family $Q_v\colon W_v \To v$, let $Q\colon W \To V$
  denote $\bigsqcup Q_v$.  Then $\bigvee (Q\sb Q\pb) = \bigvee 1_V$,
  so we have
  \begin{equation*}
    1_u = \bigvee(P\sb P\pb)
    = \bigvee\left (P\sb\circ \bigvee \left(Q\sb Q\pb \right)
      \circ  P\pb\right )
    = \bigvee\left ((P Q)\sb \circ (P Q)\pb \right )
  \end{equation*}
  so $P Q$ is also covering.

  Thirdly, if $P\colon V\To u$ is covering and $P\le Q$, say $P = Q F$
  for some functional array $F$, then
  \[ 1_u = \bigvee(P\sb P\pb) \le \bigvee(P\sb F\sb F\pb P\pb)
  = \bigvee(Q\sb Q\pb) \]
  since $F\sb$ consists of maps.
  The reverse inclusion is always true, so $Q$ is also covering.

  Fourthly, let $P\colon V\To u$ be a covering family and $f\colon x
  \to u$ a morphism in \bC; thus $f$ is a tight map and $P$ a cocone
  of tight maps with $\bigvee(P\sb P\pb) = 1_u$.  Then by weak
  \ka-tabularity, we can write $P\pb f\sb = \bigvee (F\sb G\pb)$ for
  $G\colon Y\To x$ a \ka-ary cocone of tight maps and $F\colon Y \To
  V$ a functional array of tight maps.  Then we have
  \[ 1_x \le f\pb f\sb = \bigvee (f\pb P\sb P\pb f\sb)
  = \bigvee \Big( (P\pb f\sb)\o (P\pb f\sb) \Big)
  \]
  so by applying \autoref{thm:entire-detect2} to $P\pb f\sb$, we
  conclude that $G$ is covering.

  Now the assumption $P\pb f\sb = \bigvee (F\sb G\pb)$ implies, in
  particular, that for any $y\in Y$ we have $(f_y)\sb (g_y)\pb \le
  (p_{f(y)})\pb f\sb$.  Since $(p_{f(y)})\sb \dashv (p_{f(y)})\pb$ and
  $(g_y)\sb\dashv (g_y)\pb$, by the mates correspondence for
  adjunctions this is equivalent to $f\sb (g_y)\sb \le (p_{f(y)})\sb
  (f_y)\sb$.  But since maps are discretely ordered, this is
  equivalent to $f\sb (g_y)\sb = (p_{f(y)})\sb (f_y)\sb$.  As this
  holds for all $y\in Y$, we have $f\sb G\sb = P\sb F\sb$.

  Therefore, by \autoref{thm:wkktab}\ref{item:relk3} there is a
  covering family $Q\colon Z\To Y$ with $f G Q = P F Q$, hence $f (G
  Q) \le P$.  We have already proven that covering families compose,
  so $G Q$ is covering.
\end{proof}

Henceforth, we will always consider $\nTMap(\lA)$ as a weakly \ka-ary
site with this topology.  Next, we characterize local \ka-prelimits in
$\nTMap(\lA)$ in terms of \lA, following~\cite[\S2.14]{fs:catall}.

\begin{lem}\label{thm:pb-detect}
  In a weakly \ka-tabular \lA, a commuting diagram of arrays of tight
  maps
  \[\vcenter{\xymatrix{
      U\ar@{=>}[r]^A\ar@{=>}[d]_B &
      x\ar[d]^f\\
      y\ar[r]_g &
      z}}
  \]
  such that $g\pb f\sb = \bigvee (B\sb A\pb)$ is a local
  \ka-pre-pullback in $\nTMap(\lA)$.  Therefore, $\nTMap(\lA)$ has
  local \ka-pre-pullbacks.
\end{lem}

\noindent
Note that the given condition holds in $\lrelk (\bC)$ for any local
\ka-pre-pullback in \bC.

\begin{proof}
  Suppose $h\colon v\to x$ and $k\colon v\to y$ with $f h = g k$; then
  \[ \bigvee (k\pb B\sb A\pb h\sb ) = k\pb g\pb f\sb h\sb
  = h\pb f\pb f\sb h\sb \ge 1_v.
  \]
  Define $\Theta = A\pb h\sb \meet B\pb k\sb \colon v \Rto U$.  Since
  \lA is weakly \ka-tabular, we can find a \ka-ary cocone $P\colon W
  \To v$ and a functional array $S\colon W \To U$ such that $\bigvee
  (S\sb P\pb) = \Theta$.  Using twice the modular law and the fact
  that meets distribute over joins, we have
  \begin{align*}
    \bigvee (\Theta\o\Theta)
    &= \bigvee \big(\Theta\o(A\pb h\sb \meet \Theta)\big)\\
    &\ge \left(\bigvee (\Theta\o A\pb h\sb)\right) \meet 1_v\\
    &= \left(\bigvee ((h\pb A\sb \meet k\pb B\sb) A\pb h\sb)\right)
    \meet 1_v\\
    &\ge 1_v \meet \left(\bigvee (k\pb B\sb A\pb h\sb )\right) \meet 1_v\\
    &= 1_v
  \end{align*}
  (using the calculation above in the last step).  Thus, by
  \autoref{thm:entire-detect2}, $P$ is covering.  Now by definition of
  $S$ and $P$, we have
  \[ \bigvee (A\sb S\sb P\pb)
  = \bigvee (A\sb (A\pb h\sb \meet B\pb k\sb))
  \le \bigvee (A\sb A\pb h\sb) \le h\sb
  \]
  and therefore (since maps are discretely ordered) $A\sb S\sb = h\sb
  P\sb$.  Similarly, $B\sb S\sb = k\sb P\sb$.  Thus, by the second
  half of weak \ka-tabularity, we can find a further covering family
  $Q$ such that $A S Q = h P Q$ and $B S Q = k P Q$.  Since $P Q$ is
  covering and $SQ$ is functional, this gives the desired
  factorization.

  Finally, to show that $\nTMap(\lA)$ has local \ka-pre-pullbacks, it
  suffices to show that for any $f,g$ there exist $A,B$ as above.  By
  the first half of weak \ka-tabularity, we can find $C,D$ with domain
  $V$, say, such that $g\pb f\sb = \bigvee (D\sb C\pb)$.  Then
  $(d_v)\sb (c_v)\pb \le g\pb f\sb$ for any $v$, whence $g\sb (d_v)\sb
  \le f\sb (c_v)\sb$, i.e.\ $g\sb D\sb \le f\sb C\sb$.  Since maps are
  discretely ordered, $g\sb D\sb = f\sb C\sb$, so by the second half
  of weak \ka-tabularity we can find a covering family $R\colon W \To
  V$ such that $g D R = f C R$.  And since $\bigvee (R\sb R\pb) = 1$,
  we have $g\pb f\sb = \bigvee (D\sb C\pb) = \bigvee (D\sb R\sb R\pb
  C\pb)$, so defining $A = C R$ and $B = D R$ suffices.
\end{proof}

\begin{lem}\label{thm:ktab-monic}
  If $\phi \colon x\rto x$ in a \ka-ary allegory satisfies $\phi \le
  1_x$, and $\phi = \bigvee (F G\o)$ is a weak \ka-tabulation of \phi,
  then $F=G$.
\end{lem}
\begin{proof}
  Denote by $U$ the domain of $F$, $G$.
  Then for any $u\in U$ we have
  \[ f_u \le (f_u g_u\o g_u) \le \bigvee (F G\o g_u) = \phi g_u \le g_u \]
  and consequently $f_u=g_u$ since maps are discretely ordered.
\end{proof}

\begin{lem}\label{thm:eql-detect}
  If $f,g\colon x\toto y$ are tight maps in a weakly \ka-tabular \lA,
  and $E\colon U\To x$ is a \ka-ary cocone of tight maps such that $f
  E = g E$ and
  \[\bigvee \big(E\sb E\pb\big)
  = 1_x \meet (f\pb \meet g\pb)(f\sb \meet g\sb),
  \]
  then $E$ is a local \ka-pre-equalizer of $f$ and $g$ in
  $\nTMap(\lA)$.  Therefore, $\nTMap(\lA)$ has local
  \ka-pre-equalizers.
\end{lem}
\begin{proof}
  Suppose $f h = g h$ for a tight map $h\colon v \to x$.  By weak
  \ka-tabularity, we have a \ka-ary cocone $P\colon W\To v$ and a
  functional array $S\colon W \To U$ such that $\bigvee(S\sb P\pb) =
  E\pb h\sb$.  Then using the modular law frequently, we have
  \begin{align*}
    \bigvee (h\pb E\sb E\pb h\sb)
    &= h\pb \left(\bigvee (E\sb E\pb) \right) h\sb\\
    &= h\pb \Big(1_x \meet \big(f\pb \meet g\pb\big)
    \big(f\sb \meet g\sb\big)\Big) h\sb\\
    &\ge h\pb \Big(h\sb h\pb \meet \big(f\pb \meet g\pb\big)
    \big(f\sb \meet g\sb\big)\Big) h\sb\\
    &\ge h\pb \Big(h\sb \meet \big(f\pb \meet g\pb\big)
    \big(f\sb \meet g\sb\big)h\sb\Big)\\
    &\ge 1_v \meet h\pb\big(f\pb \meet g\pb\big)
    \big(f\sb \meet g\sb\big)h\sb\\
    &\ge 1_v \meet h\pb\big(h\sb h\pb f\pb \meet g\pb\big)
    \big(f\sb h\sb h\pb \meet g\sb\big)h\sb\\
    &\ge 1_v \meet \big(h\pb f\pb \meet h\pb g\pb\big)
    \big(f\sb h\sb \meet g\sb h\sb\big)\\
    &= 1_v \meet \big(h\pb f\pb\big)\big(f\sb h\sb\big)\\
    &= 1_v.
  \end{align*}
  Therefore, by \autoref{thm:entire-detect2}, $P$ is a covering family.

  Now, $\bigvee(S\sb P\pb) = E\pb h\sb$ implies $(s_{w})\sb (p_w)\pb
  \le (e_{s(w)})\pb h\sb$ for all $w$, hence $E\sb S\sb \le h\sb
  P\sb$.  Since maps are discretely ordered, we have $E\sb S\sb = h\sb
  P\sb$, and so by the second half of weak \ka-tabularity we have a
  further covering family $Q$ such that $E S Q = h P Q$.  Since $P Q$
  is covering and $S Q$ is functional, this gives the desired
  factorization.

  Finally, to show that $\nTMap(\lA)$ has local \ka-pre-equalizers, it
  suffices to show that for any $f,g$ there exists an $E$ as above.
  By the first half of weak \ka-tabularity, we can write
  \[\bigvee \big(A\sb B\pb\big)
  = 1_x \meet (f\pb \meet g\pb)(f\sb \meet g\sb)
  \]
  for some $A$ and $B$, but then by \autoref{thm:ktab-monic}, we must
  have $A=B$.  Then for any $u$, we have $(a_u)\sb (a_u)\pb \le f\pb
  g\sb$, hence $f\sb (a_u)\sb \le g\sb (a_u)\sb$, whence $f\sb A\sb =
  g\sb A\sb$ since maps are discretely ordered.  As usual, we can then
  find a covering family $R$ such that $f A R = g A R$, and also
  $\bigvee (A\sb A\pb) = \bigvee (A\sb R\sb R\pb A\pb)$ since $R$ is
  covering; thus defining $E = A R$ suffices.
\end{proof}

Now we can prove the main theorem relating framed allegories to sites.

\begin{thm}\label{thm:relk-im}
  A \ka-ary framed allegory is of the form $\lrelk(\bC)$ for a locally
  \ka-ary site \bC if and only if it is weakly \ka-tabular.
\end{thm}
\begin{proof}
  Only ``if'' remains to be proven.  Thus, suppose \lA is weakly
  \ka-tabular, and define $\bC = \nTMap (\lA)$.  By
  Lemmas~\ref{thm:ktab-top}, \ref{thm:pb-detect},
  and~\ref{thm:eql-detect} and \autoref{thm:pbeq-connlim}, \bC is a
  locally \ka-ary site, so it remains to show $\lA \cong \lrelk
  (\bC)$.
  
  Define $\f\colon \lrelk (\bC)\to \lA$ to be the identity on objects
  and tight maps, and to take a \ka-sourced array $P\colon U \To
  \{x,y\}$ (regarded as a morphism $x\rto y$ in $\lrelk (\bC)$) to the
  morphism $\bigvee\left((P|_y)\sb (P|_x)\pb\right)\colon x\rto y$ in
  \lA.  Our construction of local \ka-pre-pullbacks in $\nTMap(\lA)$
  implies easily that this operation preserves composition, and it
  certainly preserves identities.  Moreover, if $Q\colon V \To U$ is a
  covering family, then by definition $\bigvee (Q\sb Q\pb)$ is the
  identity, so $\f(P) = \f(P Q)$.  This implies that $\f$ preserves
  the ordering on hom-posets, hence defines a 2-functor.

  Now we claim that $\f(P) \le \f(Q)$ implies $P\lle Q$ (i.e.\ $P\le
  Q$ in $\lrelk (\bC)$).  Write $U$ for the domain of $P$ and $V$ for
  the domain of $Q$, and define $F = P|_x$, $G = P|_y$, $H = Q|_x$,
  and $K = Q|_y$; thus $\f(P) \le \f(Q)$ means $\bigvee (G\sb F\pb)
  \le \bigvee (K\sb H\pb)$.  Fix some $u\in U$; then by assumption
  $(g_u)\sb (f_u)\pb \le \bigvee (K\sb H\pb)$, hence
  \[ 1_u \le \bigvee \Big( (g_u)\pb K\sb H\pb (f_u)\sb \Big).
  \]
  Let $\Phi\colon u\rto V$ denote the cone $H\pb (f_u)\sb \meet K\pb
  (g_u)\sb$.  By weak \ka-tabularity, we can write $\Phi = \bigvee
  (S\sb R\pb)$ for a \ka-ary cocone $R\colon W \To u$ and a functional
  array $S\colon W \To V$.  Moreover, using the modular law twice, we
  have
  \begin{align*}
    \bigvee (\Phi\o\Phi)
    &= \bigvee \Big(\Phi\o \big( H\pb (f_u)\sb \meet \Phi\big)\Big)\\
    &\ge \left( \bigvee (\Phi\o H\pb (f_u)\sb) \right) \meet 1_u\\
    &= \left( \bigvee \big((f_u)\pb H\sb \meet (g_u)\pb K\sb\big)
      H\pb (f_u)\sb \right) \meet 1_u\\
    &\ge 1_u \meet
    \left(\bigvee \big((g_u)\pb K\sb H\pb (f_u)\sb\big) \right)
    \meet 1_u\\
    &= 1_u.
  \end{align*}
  Thus by \autoref{thm:entire-detect2}, $R$ is covering.  By
  definition, we also have for any $w\in W$
  \[ (s_w)\sb (r_w)\pb \le (h_{s(w)})\pb (f_{r(w)})\sb.
  \]
  Hence $H\sb S\sb \le F\sb R\sb$, whence $H\sb S\sb = F\sb R\sb$ as
  maps are discretely ordered.  As usual, by the second half of weak
  \ka-tabularity we can then find a further covering family $R'$ such
  that $H S R' = F R R'$.  Repeating this process for $K$ and $G$, we
  obtain a covering family $R''$ such that $K S R'' = G R R''$.  By
  passing to a common refinement of $R'$ and $R''$, we can conclude
  that $P\lle Q$.

  Thus, \f is an embedding on hom-posets.  But by the first half of
  weak \ka-tabularity, it is also surjective on hom-posets; hence
  (since it is bijective on objects) it is an isomorphism of
  2-categories.  Since it evidently preserves the involution and the
  finite meets and \ka-ary joins in hom-posets, this makes it an
  isomorphism of \ka-ary framed allegories.
\end{proof}

We can also make this functorial.  Recall the notion of
\emph{pre-morphism of sites} from \autoref{rmk:premor}.  Since a
pre-morphism of sites preserves $\lle$ and all the prelimit
constructions used in defining $\lrelk$, it induces a \ka-ary framed
allegory functor $\lrelk(\f) \colon \lrelk(\bC) \to \lrelk(\bD)$.
This operation easily extends to a 2-functor $\lrelk\colon \LSITEk \to
\FALLk^{\mathrm{wt}}$, where $\FALLk^{\mathrm{wt}}$ denotes the full
sub-2-category of $\FALLk$ determined by the weakly \ka-tabular framed
\ka-ary allegories.

\begin{thm}\label{thm:relk-2ff}
  The 2-functor $\lrelk\colon \LSITEk\to \FALLk^{\mathrm{wt}}$ is a
  2-equivalence.
\end{thm}
\begin{proof}
  Let \bC and \bD be locally \ka-ary sites; we must show that
  \begin{equation}\label{eq:relk-2ff}
    \lrelk\colon \LSITEk(\bC,\bD) \to \FALLk(\lrelk(\bC),\lrelk(\bD))
  \end{equation}
  is an isomorphism of categories.  Firstly, since a natural
  transformation in \LSITEk is literally still present as the tight
  part of its induced framed allegory
  transformation,~\eqref{eq:relk-2ff} is faithful.  Secondly, since a
  framed allegory transformation is determined uniquely by its tight
  part, which is a natural transformation and hence automatically a
  2-cell in \LSITEk (since we have imposed no restrictions on
  these),~\eqref{eq:relk-2ff} is full.  Thirdly, since a pre-morphism
  of sites is literally still present as the tight part of its induced
  framed allegory functor,~\eqref{eq:relk-2ff} is injective on
  objects.  Thus, it remains to show it is surjective on objects.

  Let $\f\colon \lrelk(\bC) \to \lrelk(\bD)$ be a \ka-ary framed
  allegory functor.  Its underlying tight part is a functor
  \[ \bC = \nTMap(\lrelk(\bC)) \to \nTMap(\lrelk(\bD)) = \bD.
  \]
  By \autoref{rmk:premor} and Lemmas~\ref{thm:cov-detect},
  \ref{thm:pb-detect} and~\ref{thm:eql-detect}, this functor is a
  pre-morphism of sites; thus it remains only to show that the framed
  allegory functor it induces coincides with \f.  But every morphism
  of $\lrelk(\bC)$ can be written as a \ka-ary join $\bigvee(F\sb
  G\pb)$ for some \ka-ary cocones $F$ and $G$ in \bC.  Thus, since \f
  preserves $J$, $(-)\o$, and $\bigvee$, its action on all morphisms
  is determined by its action on tight maps.
\end{proof}

It now remains only to characterize those weakly \ka-tabular framed
allegories \lA for which $\nTMap(\lA)$ is a \ka-ary site.

\begin{defn}
  An allegory has \textbf{local maxima} if each hom-poset has a top
  element.
\end{defn}

\begin{lem}\label{thm:prod-detect}
  A \ka-sourced array $P\colon W \To \{u,v\}$ in a locally \ka-ary
  site \bC is a local \ka-pre-product of $u$ and $v$ if and only if
  $P|_u$ and $P|_v$ form a weak \ka-tabulation of a top element $u\rto
  v$ in $\lrelk(\bC)$.
\end{lem}
\begin{proof}
  Essentially by definition.
\end{proof}

\begin{defn}
  A \textbf{weak \ka-unit} in a \ka-ary allegory \cA is a \ka-ary
  family of objects $U$ such that for every object $x$, there is a
  cone $\Phi\colon x\Rto U$ such that $1_x \le \bigvee (\Phi\o\Phi)$.
\end{defn}

\begin{lem}\label{thm:term-detect}
  A \ka-ary family $U$ of objects in a locally \ka-ary site is a local
  \ka-pre-terminal-object if and only if it is a weak \ka-unit in
  $\lrelk(\bC)$.
\end{lem}
\begin{proof}
  By \autoref{thm:entire-detect2} and weak \ka-tabularity, $U$ is a
  weak \ka-unit in $\lrelk (\bC)$ if and only if for every $x$, there
  exists a \ka-ary cocone $G\colon W\To x$ and a functional array
  $F\colon W\To U$ in \bC such that $\bigvee (G\sb G\pb) = 1_x$.  But
  by \autoref{thm:cov-detect}, $\bigvee (G\sb G\pb) = 1_x$ just says
  that $G$ is covering, so this is precisely the definition of when
  $U$ is a local \ka-pre-terminal-object.
\end{proof}

\noindent
Let $\FALLk^{\mathrm{wt}\times}$ denote the locally full
sub-2-category of \FALLk determined by
\begin{itemize}[nolistsep]
\item weakly \ka-tabular \ka-ary framed allegories with local maxima
  and weak \ka-units, and
\item \ka-ary framed allegory functors that preserve local maxima and
  weak \ka-units.
\end{itemize}

\begin{thm}\label{thm:relk-eqv}
  The 2-functor $\lrelk$ is a 2-equivalence from \SITEk to
  $\FALLk^{\mathrm{wt}\times}$.
\end{thm}
\begin{proof}
  Lemmas~\ref{thm:prod-detect} and~\ref{thm:term-detect} show that a
  locally \ka-ary site \bC is \ka-ary exactly when $\lrelk(\bC)$ lies
  in $\FALLk^{\mathrm{wt}\times}$.  Similarly, the morphisms in
  $\FALLk^{\mathrm{wt}\times}$ are chosen precisely to be those
  corresponding to morphisms of sites.
\end{proof}

\section{The exact completion}
\label{sec:exact-completion}

We now identify those framed allegories that correspond to \ka-ary
exact categories, using the notion of \emph{collage of a congruence},
an allegorical analogue of the \emph{quotients} of congruences we used
in \S\ref{sec:regularity}.  In this section we consider only
``unframed'' collages.

\begin{defn}\label{def:all-cong}
  A \textbf{\ka-ary congruence} in a \ka-ary allegory \cA is a \ka-ary
  family of objects $X$ in \cA with an array $\Phi\colon X \Rto X$
  such that $\bigvee 1_X \le \Phi$, $\bigvee(\Phi\Phi) \le \Phi$, and
  $\Phi\o \le \Phi$.
\end{defn}

It follows that the last two inequalities are actually equalities:
$\bigvee(\Phi\Phi) = \Phi$ and $\Phi\o = \Phi$.

\begin{example}\label{thm:cong-cong}
  A \ka-ary congruence in $\lrelk(\bC)$ is precisely an equivalence
  class of \ka-ary congruences in \bC as in \autoref{def:cong}.
\end{example}

\begin{lemma}
  Let $X$ and $Y$ be \ka-ary families of objects in a \ka-ary site.
  \begin{enumerate}[nolistsep]
  \item $\bigvee 1_X \colon X\Rto X$ is a \ka-ary congruence.
  \item For any \ka-ary congruence $\Phi\colon Y\Rto Y$ and functional
    array $F\colon X\To Y$ of maps, $F\o\Phi F\colon X\Rto X$ is a
    \ka-ary congruence.
  \item For any finite family $\{\Phi_i \colon X\Rto X\}_{1\le i\le
      n}$ of \ka-ary congruences, $\bigwedge_i \Phi_i$ is a \ka-ary
    congruence.\qed
  \end{enumerate}
\end{lemma}

This should be compared with \autoref{thm:basic-cong}.  In particular,
the kernel of a \ka-to-finite array $P\colon X \To U$ in a \ka-ary
site can be constructed as $\bigwedge_{u\in U} (P|_u)\pb (P|_u)\sb$ in
$\lrelk(\bC)$.

\begin{defn}
  Let $\Phi\colon X\Rto X$ be a \ka-ary congruence in a \ka-ary
  allegory \cA.  A \textbf{collage} of $\Phi$ is a lax colimit of
  $\Phi$ regarded as a diagram in the 2-category \cA.
\end{defn}

\noindent
Explicitly, a collage of $\Phi$ is an object $w$ with a cocone
$\Psi\colon X\Rto w$ such that
\begin{enumerate}[label=(\alph*),leftmargin=*,noitemsep]
\item $\bigvee(\Psi\Phi) \le \Psi$;\label{item:coll1}
\item For any other object $z$ and cocone $\Theta\colon X\Rto z$ such
  that $\bigvee (\Theta\Phi) \le \Theta$, there is a morphism
  $\chi\colon w\rto z$ such that $\chi \Psi = \Theta$;
  and\label{item:coll2}
\item Given morphisms $\chi,\chi'\colon w\rto z$ such that $\chi \Psi
  \le \chi' \Psi$, we have $\chi \le \chi'$.\label{item:coll3}
\end{enumerate}

\noindent
Note that~\ref{item:coll3} implies uniqueness of $\chi$
in~\ref{item:coll2}.  Also, since $\bigvee 1_X \le \Phi$, we have
$\Theta\le \bigvee (\Theta\Phi)$ for any $\Theta\colon X\Rto z$; hence
$\bigvee (\Theta\Phi) \le \Theta$ is equivalent to $\bigvee
(\Theta\Phi) = \Theta$.

\begin{remark}
  A unary congruence is precisely a \emph{symmetric monad} in the
  2-category \cA, while its collage is a \emph{Kleisli object}.
\end{remark}

The following lemma is essentially a special case of parts
of~\cite[Prop.~2.2]{street:cauchy-enr}.

\begin{lem}\label{thm:cauchydense}
  If $\Psi\colon X\Rto w$ is a collage of $\Phi$, then:
  \begin{enumerate}[leftmargin=*,nolistsep]
  \item Each component $\psi_x\colon x \rto w$ is a
    map.\label{item:cd1}
  \item We have $\Phi = \Psi\o\Psi$ and $1_w = \bigvee
    (\Psi\Psi\o)$.\label{item:cd4}
  \item The cone $\Psi\o\colon w \Rto X$ exhibits $w$ as the lax limit
    of $\Phi$.\label{item:cd2}
  \item A morphism $\xi\colon w\rto z$ is a map if and only if
    $\xi\Psi$ is composed of maps.\label{item:cd3}
\end{enumerate}
\end{lem}
\begin{proof}
  For any $x\in X$, consider the cocone $\Phi|_x \colon X \Rto x$.
  Since $\Phi$ is a congruence, we have $\bigvee (\Phi|_x\circ \Phi)
  \le \Phi|_x$, hence a unique induced morphism $\chi_x\colon w\to x$
  such that $\chi_x \psi_{x'} = \phi_{x' x}$ for all $x'$.  In
  particular, $1_x \le \phi_{xx} = \chi_x \psi_x$.  On the other hand,
  since
  \[ \psi_x \chi_x \psi_{x'} = \psi_x \phi_{x' x} \le \psi_{x'} = 1_x \psi_{x'}
  \]
  we have $\psi_x \chi_x \le 1_w$.  Hence $\psi_x \dashv \chi_x$, so
  $\psi_x$ is a map and $\chi_x = \psi_x\o$; thus~\ref{item:cd1}
  holds.

  Now the equality $\psi_x\o\psi_{x'} = \chi_x \psi_{x'} = \phi_{x'
    x}$ is exactly the first part of~\ref{item:cd4}.  For the second
  part, since $\Psi$ is composed of maps, we have $\bigvee(\Psi\Psi\o)
  \le 1_w$.  But furthermore, for any $x$ we have $ \psi_x \le \psi_x
  \psi_x\o \psi_x$,
  hence $\psi_x \le \bigvee(\Psi\Psi\o)\psi_x$.  Since $w$ is a
  collage, this implies $1_w \le \bigvee(\Psi\Psi\o)$;
  thus~\ref{item:cd4} holds.

  \ref{item:cd2} is clear since $(-)\o$ is a contravariant
  equivalence, so it remains to prove~\ref{item:cd3}.  ``Only if''
  follows from~\ref{item:cd1}, so let $\xi\colon w\rto z$ be such that
  each $\xi\psi_x$ is a map.  Then
  \[ \psi_x \le \psi_x \psi_x\o \xi\o\xi \psi_x  \le \xi\o\xi \psi_x  \]
  for all $x$, hence since $\Psi$ is a collage we have $1_w \le \xi\o\xi$.
  On the other hand, using~\ref{item:cd4} we have
  $\xi\xi\o = \bigvee \xi \Psi \Psi\o \xi \le 1_z$,
  so $\xi$ is a map.
\end{proof}

\begin{lem}\label{thm:coll-char}
  A cocone $\Psi\colon X \Rto w$ is a collage of $\Phi\colon X \Rto X$
  if and only if $\Phi = \Psi\o\Psi$ and $1_w = \bigvee(\Psi\Psi\o)$.
\end{lem}
\begin{proof}
  ``Only if'' is just \autoref{thm:cauchydense}\ref{item:cd4}, so
  suppose $\Phi = \Psi\o\Psi$ and $1_w = \bigvee(\Psi\Psi\o)$.  Then
  $1_x \le \phi_{xx} \le \psi_x\o \psi_x$, while $\psi_x \psi_x\o \le
  1_w$, so $\Psi$ is composed of maps.  Thus, $\phi_{x'x}\le
  \psi_{x'}\o \psi_x$ implies $\psi_{x'} \phi_{x'x} \le \psi_x$, i.e.\
  $\bigvee(\Psi\Phi) \le \Psi$.
  Now for any $\Theta\colon X\Rto z$ with $\bigvee (\Theta\Phi) =
  \Theta$, let $\chi=\bigvee(\Theta\Psi\o)\colon w\rto z$; then
  \[\chi\Psi = \bigvee(\Theta\Psi\o \Psi)
  = \bigvee(\Theta\Phi) = \Theta.
  \]
  Finally, given $\chi,\chi'\colon w\rto z$ with $\chi\Psi \le
  \chi'\Psi$, we have
  \[ \chi = \bigvee (\chi \Psi\Psi\o)
  \le \bigvee(\chi' \Psi\Psi\o) = \chi'.
  \]
  Thus, $\Psi$ is a collage of $\Phi$.
\end{proof}

In particular, any \ka-ary allegory functor preserves collages of
\ka-ary congruences.  We are now ready to characterize properties of
\bC in terms of $\lrelk (\bC)$.

\begin{prop}\label{thm:subchordate1}
  For a locally \ka-ary site \bC, the following are equivalent.
  \begin{enumerate}[nolistsep]
  \item All covering families in \bC are epic.
  \item $\lrelk (\bC)$ is subchordate.
  \end{enumerate}
\end{prop}
\begin{proof}
  Since $\lrelk (\bC)$ is weakly \ka-tabular, if $f\sb = g\sb$ then $f
  P = g P$ for a covering family $P$.  Thus, if covering families are
  epic, $f=g$.  Conversely, if $f P = g P$,
  then 
  \[ f\sb = \bigvee f\sb P\sb P\pb = \bigvee g\sb P\sb P\pb = g\sb.
  \]
  Thus, if $\lrelk (\bC)$ is subchordate, then $f=g$; hence covering
  families are epic.
\end{proof}

\begin{prop}\label{thm:chordate1}
  For a locally \ka-ary site \bC, the following are equivalent.
  \begin{enumerate}[nolistsep]
  \item \bC is subcanonical.
  \item $\lrelk (\bC)$ is chordate.
  \end{enumerate}
\end{prop}
\begin{proof}
  Suppose \bC is subcanonical.  By \autoref{thm:subchordate1}, $\lrelk
  (\bC)$ is subchordate, so it will suffice to show that every loose
  map in $\lrelk (\bC)$ can be tightened.  Let $f\colon x\lto y$ be a
  loose map, and let $f = \bigvee (G\sb P\pb)$ be a weak
  \ka-tabulation of it, for cocones $G\colon U \To y$ and $P\colon
  U\To x$ in \bC.  Then $P$ is covering by
  \autoref{thm:entire-detect2}.  Define $\Phi = P\pb P\sb$; then
  $\Phi$ is a congruence, and $P\sb$ is a collage of $\Phi$ by
  \autoref{thm:coll-char}.

  Now $f = \bigvee (G\sb P\pb)$ implies $(g_u)\sb p_u\pb \le f$ for
  all $u\in U$, hence $(g_u)\sb \le f (p_u)\sb$ and so $G\sb = f P\sb$
  since maps are discretely ordered.  Thus, we have
  \[ \bigvee (G\sb \Phi) = \bigvee (G\sb P\pb P\sb) = f P\sb = G\sb
  \]
  so $G\sb$ is the cocone which induces $f$ by the universal property
  of the collage $P\sb$.

  Choose weak \ka-tabulations to obtain functional arrays $A$ and $B$
  such that $\Phi = \bigvee (A\sb B\pb)$.  Then by construction of
  $\lrelk (\bC)$, $A$ and $B$ form a kernel of $P$ in \bC, and by the
  above we have $\bigvee (G\sb A\sb B\pb) \le G\sb$, hence $G\sb A\sb
  = G\sb B\sb$ since maps are discretely ordered.  Since covering
  families in \bC are epic, by \autoref{thm:subchordate1} we have $G A
  = G B$.  And since \bC is subcanonical, by \autoref{thm:ee-lwkcolim}
  $P$ is a colimit of $A$ and $B$.  Thus, there is a unique map
  $h\colon x \to y$ in \bC with $h P = G$.  Since $f$ is unique such
  that $f P\sb = G\sb$, we have $f = h\sb$.  Thus, $\lrelk (\bC)$ is
  chordate.
  
  Conversely, suppose $\lrelk (\bC)$ is chordate, and let $P\colon
  U\To x$ be a covering family in \bC.  As before, define $\Phi = P\pb
  P\sb$ and $\Phi = \bigvee (A\sb B\pb)$; then $A$ and $B$ are a
  kernel of $P$, and $P\sb$ is a collage of $\Phi$.  It will suffice
  to show that $P$ is a colimit of $A$ and $B$, so let $G\colon U\To
  y$ satisfy $G A = G B$.  Then $G\sb \ge \bigvee (G\sb B\sb B\pb) =
  \bigvee (G\sb A\sb B\pb) = \bigvee (G\sb \Phi)$, so $G$ induces a
  unique (loose) map $f\colon x\lto y$ such that $f P\sb = G\sb$.
  Since $\lrelk (\bC)$ is chordate, $f$ admits a unique tightening,
  which says precisely that $G$ factors uniquely through $P$.  Thus,
  $P$ is effective-epic, so \bC is subcanonical.
\end{proof}

\begin{thm}\label{thm:alleg-exact}
  For a \ka-ary site \bC, the following are equivalent.
  \begin{enumerate}[nolistsep]
  \item \bC is \ka-ary exact with its \ka-canonical
    topology.\label{item:ae1}
  \item $\lrelk(\bC)$ is chordate and has collages of \ka-ary
    congruences.\label{item:ae2}
  \end{enumerate}
\end{thm}
\begin{proof}
  By \autoref{thm:exact}\ref{item:ex1}, condition~\ref{item:ae1} is
  equivalent to \bC being subcanonical and any \ka-ary congruence
  being a kernel of some covering cocone.  By \autoref{thm:chordate1},
  the former is equivalent to $\lrelk(\bC)$ being chordate.  And by
  \autoref{thm:coll-char} and the identification of congruences and
  kernels in \bC and $\lrelk(\bC)$, the latter is equivalent to
  $\lrelk(\bC)$ having collages of \ka-ary congruences.
\end{proof}

Thus, to show that \EXk is reflective in \SITEk, by
\autoref{thm:relk-eqv} it suffices to show that framed allegories
satisfying \ref{thm:alleg-exact}\ref{item:ae2} are reflective in
$\FALLk^{\mathrm{wt}\times}$.  We know that chordate framed allegories
are reflective in \FALLk, and the chordate reflection clearly takes
$\FALLk^{\mathrm{wt}\times}$ into itself, so it remains to freely add
collages.  This can be done very explicitly, but we prefer a more
abstract approach, for reasons that will become clear in
\S\ref{sec:frac}.

Let \SUPk denote the category of moderate \textbf{\ka-cocomplete
  posets} (those with \ka-ary joins).  There is a tensor product on
\SUPk which represents ``bilinear'' maps: functions that preserve
\ka-ary joins in each variable.  A \SUPk-enriched category is
precisely a 2-category whose hom-categories are posets with \ka-ary
joins that are preserved by composition.  In particular, every \ka-ary
allegory is such.  We use the same notation and terminology for
morphisms and maps in \SUPk-categories as we do in allegories.

\begin{remark}
  Of course, \SUPk is very large, so this is an exception to our
  general rule that all categories are moderate.  We will only
  consider \SUPk-enriched categories that are moderate (but not
  necessarily, of course, locally small).
\end{remark}

If we omit the condition $\Phi\o \le \Phi$ from the definition of a
\ka-ary congruence, we obtain a notion which makes sense in any
\SUPk-category.

\begin{defn}\label{def:dircong}
  A \textbf{\ka-ary directed congruence} in a \SUPk-category is a
  \ka-ary family of objects $X$ with an array $\Phi\colon X \Rto X$
  such that $\bigvee 1_X \le \Phi$ and $\bigvee(\Phi\Phi) \le \Phi$.
\end{defn}

\begin{remark}
  In~\cite{walters:sheaves-cauchy-1,walters:sheaves-cauchy-2,
    bcsw:variation-enr,ckw:axiom-mod}
  directed congruences in \cA are called \emph{\cA-categories}, but we
  prefer a different terminology to minimize confusion with (for
  instance) \SUPk-categories.
\end{remark}

A \textbf{collage} of a directed congruence is, as before, its lax
colimit; this can be expressed as a certain \SUPk-weighted colimit.

\begin{lem}\label{thm:supk-coll-char}
  If $\Phi\colon X\Rto X$ is a \ka-ary directed congruence in a
  \SUPk-category, then a cocone $\Psi\colon X \Rto w$ is a collage of
  $\Phi$ if and only if $\Psi$ is composed of maps and for
  $\Psi^*\colon w \Rto X$ its cone of right adjoints, we have $\Phi =
  \Psi^* \Psi$ and $1_w = \bigvee(\Psi\Psi^*)$.
\end{lem}
\begin{proof}
  The proofs of Lemmas~\ref{thm:cauchydense}\ref{item:cd1} and
  \ref{thm:coll-char} apply with trivial modifications.
\end{proof}

\begin{cor}
  Any \SUPk-functor preserves collages of directed congruences.\qed
\end{cor}

\noindent
That is, collages of directed congruences are \textbf{absolute
  colimits} for \SUPk (see~\cite{street:absolute}).

\begin{cor}
  If $\Psi\colon X \Rto w$ is a collage of a directed congruence
  $\Phi$ in a \SUPk-category, then $\Psi^*\colon w \Rto X$ is a lax
  limit of $\Phi$.
\end{cor}
\begin{proof}
  Apply one direction of \autoref{thm:supk-coll-char} in \cA and the
  other direction in $\cA\op$.
\end{proof}

\begin{cor}\label{thm:coll-mor}
  If $u$ and $v$ are collages of directed congruences $\Phi\colon
  X\Rto X$ to $\Theta\colon Y\Rto Y$, respectively, then $\cA(u,v)$ is
  isomorphic to the poset of arrays $\Psi\colon X\Rto Y$ such that
  $\bigvee (\Theta\Psi) \le \Psi$ and $\bigvee (\Psi\Phi) \le \Psi$,
  with pointwise ordering.  If \cA is additionally a \ka-ary allegory,
  then the involution $\cA(u,v) \to \cA(v,u)$ is identified with the
  induced involution taking arrays $X\Rto Y$ to arrays $Y\Rto X$.  And
  if $w$ is the collage of a third directed congruence $\Xi\colon
  Z\Rto Z$, then the composite of $\Psi\colon X\Rto Y$ representing a
  morphism $u\rto v$ with $\Omega\colon Y\Rto Z$ representing a
  morphism $v\rto w$ is represented by $\bigvee (\Omega\Psi)\colon
  X\Rto Z$.
\end{cor}
\begin{proof}
  The first sentence follows because collages are both lax colimits
  and lax limits, and the second because the involution is functorial.
  The third follows from the construction of $\chi$ in the proof of
  \autoref{thm:coll-char}.
\end{proof}

\noindent
Note that
the inequalities $\bigvee (\Theta\Psi) \le \Psi$ and $\bigvee
(\Psi\Phi) \le \Psi$ are automatically equalities.

We would now like to construct the free cocompletion of a \ka-ary
allegory with respect to collages of congruences, and general enriched
category theory suggests that this should be the closure of \cA in
$[\cA\op,\SUPk]$ under collages of congruences.  In general, such a
closure must be constructed by (transfinite) iteration, but in the
case of collages it stops after one step, due to the following fact.

\begin{lem}\label{thm:coll-coll}
  Let $X$ and $Y$ be \ka-ary families of objects in a \SUPk-category
  \cA, let $F\colon X\To Y$ be a functional array of maps, and let
  $G\colon Y\To z$ be a cocone of maps.  Suppose $G$ is a collage of
  some \ka-ary directed congruence $\Psi\colon Y\Rto Y$, and that for
  each $y\in Y$, the cocone $F|_y \colon X|_y \To y$ is a collage of
  some \ka-ary directed congruence $\Phi_y \colon X|_y \Rto X|_y$.
  Then $G F\colon X \To z$ is a collage of $F^* \Psi F$.
\end{lem}
\begin{proof}
  On the one hand, we have
  \[(G F)^* (G F) = F^* G^* G F = F^* \Psi F,\]
  since $G$ is a collage of \Psi.
  On the other hand, we have
  \[ \bigvee (G F) (G F)^*
  = \bigvee G F F^* G^* \le \bigvee G G^* \le 1_z
  \]
  since $F$ and $G$ are both covering.  Since $G F$ is clearly
  composed of maps, by \autoref{thm:supk-coll-char} it is a collage of
  $F^* \Psi F$.
\end{proof}

Therefore, let us define $\Modk(\cA)$ to be the full
sub-\SUPk-category of $[\cA\op,\SUPk]$ determined by the collages of
(the images of) congruences in \cA, and write $\fy\colon \cA\to
\Modk(\cA)$ for the restricted Yoneda embedding.  We can describe
$\Modk(\cA)$ more explicitly as follows.

\begin{lem}\label{thm:cocompletion}
  Let \cA be a \ka-ary allegory; then $\Modk(\cA)$ is equivalent to
  the following \SUPk-category.
  \begin{itemize}[nolistsep]
  \item Its objects are \ka-ary congruences in \cA.
  \item Its morphisms from $\Phi\colon X\Rto X$ to $\Theta\colon Y\Rto
    Y$ are arrays $\Psi\colon X\Rto Y$ such that $\bigvee(\Theta\Psi)
    \le \Psi$ and $\bigvee(\Psi\Phi)\le \Psi$.  (This is equivalent to
    $\bigvee(\Theta\Psi) = \Psi$ and $\bigvee(\Psi\Phi) = \Psi$, and
    also to $\bigvee(\Theta\Psi\Phi) = \Psi$.)
  \item The ordering on such arrays is pointwise.
  \item The composition of $\Psi$ and $\Psi'$ is $\bigvee(\Psi'\Psi)$,
    and the identity of $\Phi$ is $\Phi$ itself.
  \end{itemize}
  Under this equivalence, the functor \fy takes an object $x$ to the
  corresponding singleton congruence $\Delta_{\{x\}}$.
\end{lem}
\begin{proof}
  Up to equivalence, we may certainly take the objects to be the
  \ka-ary congruences themselves rather than their collages in
  $[\cA\op,\SUPk]$.  The identification of the morphisms between these
  collages in $[\cA\op,\SUPk]$, along with their composition, follows
  from \autoref{thm:coll-mor} and fully-faithfulness of the Yoneda
  embedding.
\end{proof}

Unfortunately, since colimits of congruences (as opposed to directed
congruences) are not determined by a class of \SUPk-weights, we cannot
directly apply a general theorem such as~\cite[5.35]{kelly:enriched}
to deduce the universal property of $\Modk (\cA)$.  However,
essentially the same proofs apply.

\begin{lem}\label{thm:alleg-cocplt}
  If \cA is a \ka-ary allegory, then:
  \begin{enumerate}[leftmargin=*,nolistsep]
  \item $\Modk(\cA)$ is a \ka-ary allegory.\label{item:mod1}
  \item $\fy\colon \cA\to\Modk(\cA)$ is a \ka-ary allegory
    functor.\label{item:mod3}
  \item All \ka-ary congruences in $\Modk(\cA)$ have
    collages.\label{item:mod2}
  \item If \cB is a \ka-ary allegory with collages of all \ka-ary
    congruences, then
    \[ \ALLk(\Modk(\cA),\cB) \xto{(-\circ \fy)} \ALLk (\cA,\cB) \]
    is an equivalence of categories.\label{item:mod4}
  \end{enumerate}
\end{lem}
\begin{proof}
  Binary meets and an involution for $\Modk(\cA)$ are defined
  pointwise.  The modular law follows from that in \cA, since
  composition and meets in \cA distribute over joins.  Thus
  property~\ref{item:mod1} holds, and~\ref{item:mod3} is clear by
  definition.

  Now by \autoref{thm:coll-coll}, the collage in $[\cA\op,\SUPk]$ of
  any \ka-ary congruence in $\Modk(\cA)$ is also the collage of some
  \ka-ary directed congruence in \cA.  This directed congruence is
  easily verified to be a congruence, so its collage also lies in
  $\Modk(\cA)$; thus~\ref{item:mod2} holds.

  Finally, for~\ref{item:mod4}, suppose \cB has collages of \ka-ary
  congruences.  By \autoref{thm:coll-mor}, the meets and involutions
  between collages in \cB of congruences in \cA are determined by
  meets and involutions in the image of \cA.  Thus, a \SUPk-functor
  $\Modk (\cA) \to \cB$ is a \ka-ary allegory functor if and only if
  its composition with \fy is.

  Now by~\cite[4.99]{kelly:enriched}, since \fy is fully faithful,
  $(-\circ \fy)$ induces an equivalence
  \[ \ALLk(\Modk(\cA),\cB)^\ell \xto{(-\circ \fy)}\ALLk(\cA,\cB)' \]
  whose domain is the full subcategory of $\ALLk(\Modk(\cA),\cB)$
  determined by the functors that are (pointwise) left Kan extensions
  of their restrictions to \cA, and whose codomain is the full
  subcategory of $\ALLk(\cA,\cB)$ determined by the functors which
  admit (pointwise) left Kan extensions along \fy.

  Now since $\Modk(\cA)$ is a full sub-\SUPk-category of
  $[\cA\op,\SUPk]$, the inclusion \fy is \SUPk-dense.  Thus, since
  collages of (directed) congruences are absolute \SUPk-colimits,
  by~\cite[5.29]{kelly:enriched} we have $\ALLk(\Modk(\cA),\cB)^\ell =
  \ALLk(\Modk(\cA),\cB)$.  Finally, since \cB has collages of \ka-ary
  congruences and every object of $\Modk (\cA)$ is an absolute colimit
  (a collage) of a \ka-ary congruence in \cA, the same proof as
  for~\cite[4.98]{kelly:enriched} proves that $\ALLk(\cA,\cB)' =
  \ALLk(\cA,\cB)$.
\end{proof}

\begin{thm}\label{thm:cong-refl}
  The 2-category of \ka-ary allegories that have collages of all
  \ka-ary congruences is reflective in \ALLk.\qed
\end{thm}

By ``reflective'' here we mean the inclusion has a left biadjoint, not
a strict 2-adjoint.  The same is true in the following, which is the
central theorem of the paper.

\begin{thm}\label{thm:adjoint}
  The inclusion $\EXk \into \SITEk$ is reflective.
\end{thm}
\begin{proof}
  Let $\ALLk^{\mathrm{wt}\times}$ denote the full sub-2-category of
  $\FALLk^{\mathrm{wt}\times}$ on the chordate objects.  By the
  remarks after \autoref{thm:alleg-exact}, it suffices to show that
  $\Modk$ takes $\ALLk^{\mathrm{wt}\times}$ into itself.  Thus,
  suppose $\cA\in\ALLk^{\mathrm{wt}\times}$.

  In the chordate case, the second part of weak \ka-tabularity is
  automatic.  For the first part, let $\Phi\colon X\Rto X$ and
  $\Theta\colon Y\Rto Y$ be congruences in \cA, and $\Psi\colon \Phi
  \rto \Theta$ a morphism in $\Modk(\cA)$.  Thus $\Psi$ is an array
  $X\Rto Y$ in \cA with $\bigvee(\Theta\Psi) = \Psi$ and
  $\bigvee(\Psi\Phi)= \Psi$.  By assumption, for each $x\in X$ and
  $y\in Y$ we have a weak \ka-tabulation $\psi_{x y} = \bigvee (G^{xy}
  (F^{xy})\o)$ in \cA, for \ka-ary cocones of maps $F^{xy}\colon
  U^{xy}\To x$ and $G^{xy}\colon U^{xy}\To y$.  Define $U =
  \bigsqcup_{x,y} U^{xy}$, with induced functional arrays $F\colon U
  \To X$ and $G\colon U\To Y$.  Then $\bigvee (\Phi F)\colon \fy(U)
  \To \Phi$ and $\bigvee (\Theta G)\colon \fy(U) \To \Theta$ are
  cocones of maps in $\Modk(\cA)$, and we have
  \[ \bigvee((\Theta G) (\Phi F)\o)
  = \bigvee \left(\Theta \circ \bigvee (G F\o) \circ \Phi\right)
  = \bigvee (\Theta \Psi \Phi) = \Psi.
  \]
  Hence $\bigvee (\Phi F)$ and $\bigvee (\Theta G)$ are a weak
  \ka-tabulation of $\Psi$.

  Next, suppose $U$ is a weak \ka-unit in \cA.  Given any \ka-ary
  congruence $\Phi\colon X\Rto X$, for each $x\in X$ we have
  $\Psi^x\colon x\Rto U$ with $1_x \le \bigvee ((\Psi^x)\o\Psi^x)$.
  Putting these together, we obtain an array $\Psi\colon X \Rto U$
  with $\bigvee 1_X \le \bigvee (\Psi\o \Psi)$.  Therefore,
  \[ \Phi = \bigvee(\Phi\o\Phi) \le \bigvee (\Phi\o \Psi\o\Psi \Phi)
  = \bigvee \left(\bigvee(\Psi\Phi)\o \circ \bigvee(\Psi\Phi)\right),
  \]
  so $\bigvee(\Psi\Phi)\colon \Phi \To \fy(U)$ is a cone in
  $\Modk(\cA)$ exhibiting $\fy(U)$ as a weak \ka-unit.

  Finally, if \cA has local maxima, then so does $\Modk(\cA)$; they
  are arrays consisting entirely of local maxima in \cA.  Moreover, if
  $\f\colon \cA \to \cB$ preserves these structures, then clearly so
  does $\Modk(\f)$.
\end{proof}

We denote this left biadjoint by $\exk\colon \SITEk \to \EXk$.

\begin{remark}
  If we ignore products and terminal objects, the same proof shows
  that locally \ka-ary exact categories are reflective in \LSITEk.  We
  also write $\exk(\bC)$ for this reflection; if \bC is \ka-ary the
  notation is unambiguous.
\end{remark}

By definition, the objects of $\exk(\bC)$ are the \ka-ary congruences
in \bC.  Its morphisms are less easy to describe explicitly at the
moment; see \S\ref{sec:frac} for an alternative description.  However,
our current definition is sufficient to identify our construction of
the exact completion with many of those existing in the literature.

\begin{example}
  As in \autoref{rmk:relk-egs}, if \bC is \ka-ary regular with its
  \ka-regular topology, then $\lrelk(\bC)$ is chordate and can be
  identified with the usual 2-category of relations $\nRel(\bC)$ in a
  regular category.  In the case $\ka=\un$, we can then identify
  $\nMod_{\un}(\nRel(\bC))$ with the splitting of equivalence
  relations (unary congruences), which is precisely the construction
  of the \emph{exact completion} of a regular category \bC
  from~\cite[2.169]{fs:catall}.

  For general \ka, we can factor $\Modk$ by first passing to an
  allegory of \ka-ary families, then splitting equivalence relations
  (see \autoref{thm:superext-fam}).  In this way we obtain (the
  \ka-ary version of) the \emph{pretopos completion} of a coherent
  category, as described in~\cite[2.217]{fs:catall}.
\end{example}

\begin{example}
  If \bC is a small \KA-ary site, we have remarked that
  $\lRel_{\KA}(\bC)$ is the bicategory
  of~\cite{walters:sheaves-cauchy-2}.  In that paper, the topos
  $\nSh(\bC)$ of sheaves on \bC was identified with the category of
  ``small symmetric Cauchy-complete $\lRel_{\KA}(\bC)$-categories'',
  and isomorphism classes of functors between them.
  Now small symmetric $\lRel_{\KA}(\bC)$-categories are precisely
  \KA-ary congruences in $\lRel_{\KA}(\bC)$, and the
  \emph{profunctors} between them can be identified with the loose
  morphisms in $\lmodk(\lRel_{\KA}(\bC))$.  Thus, our $\exK(\bC)$ is
  equivalent to the category of small symmetric
  $\lRel_{\KA}(\bC)$-categories and isomorphism classes of
  left-adjoint profunctors between them.  Not every left-adjoint
  profunctor is induced by a functor, but this is so when the codomain
  is Cauchy-complete.  Moreover, every $\lRel_{\KA}(\bC)$-category is
  Morita equivalent (that is, equivalent by profunctors) to a
  Cauchy-complete one as in~\cite{street:enr-cohom}.  Thus, our
  $\exK(\bC)$ is also equivalent to $\nSh(\bC)$.  We will reprove and
  generalize this by a different method in
  \S\ref{sec:exact-compl-sheav}.
\end{example}

We can also identify the universal property of $\exk(\bC)$, as
expressed in \autoref{thm:adjoint}, with those of other exact
completions.

\begin{example}
  Since \EXk is contained in \REGk as subcategories of \SITEk, the
  construction \exk is also left biadjoint to the inclusion $\EXk\into
  \REGk$.  In particular, this is the case for $\nEX_{\un} \into
  \nREG_{\un}$, which is the usual universal property of the exact
  completion of a regular category.  The case of the (infinitary)
  pretopos completion of an (infinitary) coherent category is also
  well-known.
\end{example}

\begin{example}\label{eg:exlex}
  On the other hand, we have remarked that \EXk is \emph{not}
  contained in \nLEX as a subcategory of \SITEk, since $\nLEX \into
  \SITEk$ equips a lex category with its trivial topology, while
  $\EXk\into \SITEk$ equips a \ka-ary exact category with its
  \ka-regular topology.  Nevertheless, the functor $\exk\colon
  \nLEX\to \EXk$ is still left biadjoint to the forgetful functor.
  This is because if \bC has finite limits and a trivial \ka-ary
  topology, while \bD is \ka-ary exact with its \ka-regular topology,
  then a functor $\bC\to \bD$ is a morphism of sites precisely when it
  preserves finite limits (by \autoref{thm:morsite-flim}).  Thus, we
  also reproduce the known universal property of the exact or pretopos
  completion of a lex category.
\end{example}

\begin{example}\label{eg:exwlex}
  On the third hand, the situation is different for \exk regarded as a
  functor $\nWLEX\to \EXk$, since the converse part of
  \autoref{thm:morsite-flim} fails if the domain does not have true
  finite limits.  Instead, we recover (in the case $\ka=\un$) the
  result of~\cite{cv:reg-exact-cplt} that when \bC has weak finite
  limits and \bD is exact, regular functors $\exk(\bC)\to \bD$ are
  naturally identified with \emph{left covering functors} $\bC\to \bD$
  (see \autoref{eg:left-covering}).
\end{example}

We end this section with the following observation.

\begin{thm}\label{thm:ex-prop}
  Let \bC be a \ka-ary site.
  \begin{enumerate}[nolistsep]
  \item $\fy\colon \bC\to \exk(\bC)$ is faithful if and only if all
    covering families in \bC are epic.\label{item:ep1}
  \item $\fy\colon \bC\to \exk(\bC)$ is fully faithful if and only if
    \bC is subcanonical.\label{item:ep2}
  \item $\fy\colon \bC\to \exk(\bC)$ is an equivalence in $\SITEk$ if
    and only if \bC is \ka-ary exact and equipped with its
    \ka-canonical topology, and if and only if \bC is subcanonical and
    every \ka-ary congruence is a kernel of some covering
    cocone.\label{item:ep3}
  \end{enumerate}
\end{thm}
\begin{proof}
  Recall from \autoref{thm:subchordate1} that all covering families in
  \bC are epic just when $\lrelk(\bC)$ is subchordate.  But this is
  equivalent to faithfulness (on tight maps) of the map from
  $\lrelk(\bC)$ to its chordate reflection.  Since $\cA \to \Modk
  (\cA)$ is fully faithful for any \ka-ary allegory \cA, this
  proves~\ref{item:ep1}.  \ref{item:ep2} is analogous using
  \autoref{thm:chordate1}, and~\ref{item:ep3} is immediate from the
  universal property of \exk and \autoref{thm:exact}.
\end{proof}

\section{Exact completion with anafunctors}
\label{sec:frac}

As constructed in \S\ref{sec:exact-completion}, the objects of
$\exk(\bC)$ are \ka-ary congruences in \bC, and its morphisms are a
sort of ``entire functional relations'' or ``representable
profunctors''.  In this section, we give an equivalent description of
these morphisms using a calculus of fractions, i.e.\ as
``anafunctors'' in the style of~\cite{roberts:ana}.

From the perspective of the rest of the paper, the goal of this
section is to prove \autoref{thm:ex-concrete}, which will be used to
prove \autoref{thm:ex-sheaves}.  If this were the only point, then
\autoref{thm:ex-concrete} could no doubt be obtained more directly, or
the need for it avoided entirely.  But I believe that the ideas of
this section clarify certain aspects of the theory, and will be useful
when we come to categorify it.  However, the reader is free to skip
ahead to the statement of \autoref{thm:ex-concrete} on
page~\pageref{thm:ex-concrete} and then go on to
\S\ref{sec:exact-compl-sheav}.

The central construction of this section is a \emph{framed} version of
$\Modk$.  Here is where our description of $\Modk$ as an enriched
cocompletion is most helpful; all we need do is change the enrichment.
Let $\Fk$ denote the (very large) category whose objects are functions
\[ j\colon A_\tau \to A_\lambda
\]
where $A_\tau$ is a moderate set and $A_\lambda$ is a moderate
\ka-cocomplete poset.  Its morphisms are commutative squares
consisting of a set-function and a \ka-cocontinuous poset map.  Then
\Fk is complete, cocomplete, and closed symmetric monoidal under the
tensor product
\[ A_\tau\times B_\tau \to A_\lambda\otimes B_\lambda \]
where $\otimes$ denotes the tensor product in \SUPk.

A \Fk-enriched category \lA consists of a \SUPk-enriched category \cA,
together with a category $\bA$ and a bijective-on-objects functor
$J\colon \bA\to\cA$.  In particular, any framed allegory is an
\Fk-category.  As in a framed allegory, we refer to morphisms in \bA
as \textbf{tight}, writing them as $f\colon x\to y$, and denoting by
$f\sb\colon x\rto y$ the image of $f$ under $J$.  In this generality,
$f\sb$ is not necessarily a map, but if it is, we denote its right
adjoint by $f\pb$.

If \lA is an \Fk-enriched category, by a \textbf{directed congruence}
in \lA we mean a directed congruence, as in \autoref{def:dircong}, in
the underlying \SUPk-category of \lA.

\begin{defn}
  A \textbf{tight collage} of a directed congruence $\Phi\colon X \Rto
  X$ in an \Fk-category is a cocone of tight morphisms $F\colon X\To
  w$ such that
  \begin{enumerate}[leftmargin=*,nolistsep]
  \item $F\sb$ is a collage of $\Phi$ in the underlying
    \SUPk-category, and
  \item For any $\chi\colon w\rto z$, composition with $F$ determines
    a bijection between
    \begin{enumerate}[nolistsep]
    \item tight morphisms $h$ with $h\sb = \chi$ and
    \item tight cocones $G$ with $G\sb = \chi F\sb$.
    \end{enumerate}
  \end{enumerate}
\end{defn}

Recall (\autoref{thm:cong-cong}) that \ka-ary congruences in $\lrelk
(\bC)$ are equivalence classes of \ka-ary congruences in \bC.

\begin{lem}\label{thm:tcoll-colim}
  Let \bC be a \ka-ary site in which covering families are epic, let
  $\Phi\colon X\Rto X$ be a \ka-ary congruence in \bC, and let
  $F\colon X\To w$ be a cocone in \bC such that $F\sb$ is a collage of
  $\Phi$ in the underlying allegory of $\lrelk (\bC)$.  Then $F$ is a
  tight collage of $\Phi$ in $\lrelk (\bC)$ if and only if it is a
  colimit of \Phi in \bC.
\end{lem}
\begin{proof}
  First, let $G\colon X\To z$ be any cocone in \bC; we claim that $G$
  is a cocone under \Phi in \bC if and only if $\bigvee G\sb \Phi \le
  G\sb$ in $\lrelk (\bC)$.  If we write $A,B \colon \bigsqcup
  \Phi[x_1,x_2] \To X$ for the two functional arrays that make up
  \Phi, then we have $\Phi = A\sb B\pb$ in $\lrelk (\bC)$.  Thus
  $\bigvee G\sb \Phi \le G\sb$ is equivalent to $G\sb A\sb \le G\sb
  B\sb$, which is equivalent to $G\sb A\sb = G\sb B\sb$ since maps are
  discretely ordered, and thence to $G A = G B$ since $\lrelk (\bC)$
  is subchordate.  But this is precisely to say that $G$ is a cocone
  under \Phi.

  In particular, since $F\sb$ is a collage of \Phi, we have $\bigvee
  F\sb \Phi \le F\sb$, so $F$ is a cocone under \Phi.  Moreover, a
  cocone $G\colon X\To z$ is under \Phi if and only if $G\sb$ factors
  uniquely through $F\sb$ by a (loose) map.  Thus, any such $G$
  factors uniquely through $F$ just when these loose maps are
  necessarily tight; hence $F$ is a colimit of \Phi just when it is a
  tight collage.
\end{proof}

Note that any collage in a chordate framed allegory is automatically
tight.

Tight collages can be expressed as \Fk-enriched colimits, along the
lines of~\cite[Corollary 6.6]{ls:limlax}.  We can therefore hope to
construct the free cocompletion of an \Fk-enriched category with
respect to tight collages of any class of directed congruences.

In order to describe this cocompletion explicitly, we first describe
collages in \Fk itself.  For clarity, we write $\lF_\ka$ for \Fk
regarded as an \Fk-category.  Its tight morphisms are the morphisms in
the ordinary category \Fk, while its loose morphisms $A\to B$ are
morphisms $A_\lambda \to B_\lambda$ in \SUPk.  Thus, a directed
congruence $\Phi\colon X\Rto X$ in $\lF_\ka$ consists of a \ka-ary
family $X$ of objects of \Fk, together with morphisms $\phi_{x
  x'}\colon x_\lambda \to (x')_\lambda$ in \SUPk for all $x,x'\in X$,
satisfying $\xi \le \phi_{xx}(\xi)$ and $\phi_{x'x''}(\phi_{xx'}(\xi))
\le \phi_{x x''}(\xi)$ for all $\xi\in x_{\lambda}$.

\begin{lem}\label{thm:fk-collage}
  Let $X$ be a \ka-ary family of objects of \Fk, and let $\Phi\colon
  X\Rto X$ be a directed congruence in $\lF_\ka$.  The tight collage
  of $\Phi$ is the object $w$ of \Fk described as follows.
  \begin{itemize}[nolistsep]
  \item The \ka-cocomplete poset $w_\lambda$ consists of $X$-tuples
    $(\xi_x)_{x\in X}$, where $\xi_x\in x_\lambda$ and for each
    $x,x'\in X$ we have $\phi_{x x'}(\xi_{x}) \le \xi_{x'}$.
  \item The set $w_\tau$ consists of pairs $(z,\ze)$ where $z\in X$
    and $\ze\in z_\tau$.
  \item The function $w_\tau \to w_\lambda$ sends $(z,\ze)$ to the
    tuple $(\phi_{z x}(j(\ze)))_{x\in X}$.
  \end{itemize}
  The tight cocone $F\colon X\Rto w$ is defined by
  \begin{itemize}[nolistsep]
  \item $(f_z)_\tau(\ze) = (z,\ze)$ for $z\in X$, $\ze \in z_\tau$.
  \item $(f_z)_\lambda(\xi) = (\phi_{z x}(\xi))_{x\in X}$ for
    $z\in X$, $\xi \in z_\lambda$.
  \end{itemize}
\end{lem}
\begin{proof}
  Straightforward verification.
\end{proof}

Now we need a framed version of \autoref{thm:coll-coll}.

\begin{lem}\label{thm:frcoll-coll}
  Let $X$ and $Y$ be \ka-ary families of objects in an \Fk-category
  \lA, let $F\colon X\To Y$ be a functional array of tight morphisms,
  and let $G\colon Y\To z$ be a cocone of tight morphisms.  Suppose
  $G$ is a tight collage of some \ka-ary directed congruence
  $\Psi\colon Y\Rto Y$, and that for each $y\in Y$, the cocone $F|_y
  \colon X|_y \To y$ is a tight collage of some \ka-ary directed
  congruence $\Phi_y \colon X|_y \Rto X|_y$.  Then $G F\colon X \To z$
  is a tight collage of $F\pb \Psi F\sb$.
\end{lem}
\begin{proof}
  Note that since tight collages have underlying loose collages, by
  \autoref{thm:supk-coll-char} if $H$ is a tight collage then $H\sb$
  is composed of maps.  In particular, $F\pb$ exists, so that $F\pb
  \Psi F\sb$ makes sense.  Moreover, by \autoref{thm:coll-coll}, $(G
  F)\sb$ is a loose collage of $F\pb \Psi F\sb$, so it remains to
  check the tight part of the universal property.
  
  Now given $\chi\colon z\rto w$, since $G$ is a tight collage, we
  have a bijection between tight morphisms $h\colon z\to w$ with $h\sb
  = \chi$ and tight cocones $K\colon Y\To w$ with $K\sb = \chi G\sb$.
  But for any $y\in Y$, since $F|_y$ is a tight collage, we have a
  bijection between tight morphisms $k_y\colon y\to w$ with $(k_y)\sb
  = \chi (g_y)\sb$ and tight cocones $L_y \colon X|_y \To w$ with
  $(L_y)\sb = \chi (g_y)\sb (F|_y)\sb$.  Putting these together, we
  see that $G F$ is a tight collage as well.
\end{proof}

Thus, given a \ka-ary framed allegory \lA, we let $\lmodk(\lA)$ denote
the full sub-\Fk-category of $[\lA\op,\lF_\ka]$ determined by the
collages of \ka-ary congruences in \lA.  We write $\fy\colon
\lA\to\lmodk(\lA)$ for the restricted Yoneda embedding.

\begin{prop}\label{thm:tcoll-cplt}
  The \Fk-category $\lmodk(\lA)$ can be described directly as follows.
  \begin{itemize}[nolistsep]
  \item Its underlying \ka-ary allegory is $\Modk(\cA)$, as described
    in \autoref{thm:cocompletion}.
  \item A tight morphism from $\Phi\colon X\Rto X$ to $\Theta\colon
    Y\Rto Y$ is a functional array $G\colon X\To Y$ of tight maps in
    \lA such that $\bigvee(G\sb \Phi) \le \Theta G\sb$.
  \item For such a $G$, the underlying loose morphism $\Phi \rto
    \Theta$ in $\lmodk(\lA)$ is the array $\Theta G\sb\colon X \Rto Y$
    in \lA.
  \end{itemize}
\end{prop}
\begin{proof}
  Write $\widehat{(-)}\colon \lA \to [\lA\op,\lF_{\ka}]$ for the
  \Fk-enriched Yoneda embedding of \lA, and write $\ncoll(-)$ for the
  collage of a directed congruence.  First of all, since the loose
  parts of colimits in \Fk coincide with colimits in \SUPk, and the
  analogous cocompletion of \SUPk-categories could have been
  constructed in the same way using the \SUPk-enriched Yoneda
  embedding, the loose morphisms must be the same as those described
  in \autoref{thm:cocompletion} for the \SUPk-case.

  Now, let $\Phi\colon X\Rto X$ to $\Theta\colon Y\Rto Y$ and let
  $\Psi\colon \ncoll(\widehat{\Phi}) \rto \ncoll(\widehat{\Theta})$ be
  a loose morphism in $[\lA\op,\lF_{\ka}]$.  By
  \autoref{thm:cocompletion}, $\Psi$ is determined by an array $\Psi
  \colon X\Rto Y$ such that $\bigvee(\Theta \Psi) \le \Psi$ and
  $\bigvee (\Psi\Phi)\le \Psi$.  By the universal property of a tight
  collage, a tightening of $\Psi$ is determined by a tight cocone
  $G\colon \widehat{X} \To \ncoll (\widehat{\Theta})$ in
  $[\lA\op,\lF_{\ka}]$ such that $(g_x)\sb = \Psi (f_x)\sb$ for every
  $x$, where $F\colon \widehat{X}\To \ncoll (\widehat{\Phi})$ is the
  colimiting cocone.

  Now each $g_x$ is a tight morphism from $\widehat{x}$ to $\ncoll
  (\widehat{\Theta})$.  By the enriched Yoneda lemma, this is
  equivalently an element of $\ncoll (\widehat{\Theta})(x)_\tau$.  And
  since colimits in $[\lA\op,\lF_{\ka}]$ are pointwise, this is the
  same as a tight element of the collage of $\widehat{\Theta}(x)$.
  Finally, by \autoref{thm:fk-collage}, this is just a choice of some
  $g(x)\in Y$ and a tight morphism $g_x\colon x\to g(x)$ in \lA.

  Similarly, $f_x$ can be identified with the identity
  $1_x\colon x\to x$.

  By \autoref{thm:fk-collage}, the loose morphism underlying $g_x$ in
  $[\lA\op,\lF_{\ka}]$ is determined by the family of composites
  $(\theta_{g(x),y} (g_x)\sb)_{y\in Y}$.  Similarly, the loose
  morphism underlying $f_x$ is determined by the family
  $(\phi_{x,x'})_{x'\in X}$.  Thus, the condition $(g_x)\sb = \Psi
  (f_x)\sb$ asks that
  \[ \theta_{g(x),y} (g_x)\sb = \bigvee_{x'} (\psi_{x',y} \phi_{x,x'}), \]
  i.e.\ that $\Theta G\sb = \bigvee (\Psi\Phi) = \Psi$.
  So it suffices to show that $\Theta G\sb$ determines a loose
  morphism $\ncoll(\widehat{\Phi}) \rto \ncoll(\widehat{\Theta})$ if
  and only if $\bigvee(G\sb \Phi) \le \Theta G\sb$.  But the former is
  equivalent to
  \[ \bigvee(\Theta \Theta G\sb) \le \Theta G\sb \quad\text{and}\quad
  \bigvee (\Theta G\sb \Phi)\le \Theta G\sb.
  \]
  Since $\bigvee(\Theta\Theta) = \Theta$, the first inequality is
  automatic, while since additionally $\bigvee 1_Y \le \Theta$, the
  second is equivalent to $\bigvee(G\sb \Phi) \le \Theta G\sb$.
\end{proof}

\begin{remark}
  To avoid confusion, if $G\colon X\To Y$ is a functional array of
  tight maps in \lA representing a tight morphism $\Phi\to\Theta$ in
  $\lmodk(\lA)$, we reserve the notation $G\sb$ for the underlying
  array of loose maps in \lA, and write $G\sbb$ for the underlying
  loose morphism $\Phi\lto\Theta$ in $\lmodk(\lA)$ (which is
  represented by $\Theta G\sb$ in \lA).
\end{remark}

Composition of tight maps in $\lmodk(\lA)$ is just composition of
functional arrays in \lA.  \autoref{thm:tcoll-cplt} implies that for
$G\colon \Phi\to\Theta$ and $H\colon \Theta \to \Psi$, we have $H\sbb
G\sbb = (H G)\sbb$.  But once we show in \autoref{thm:fall-cocplt}
that $\lmodk(\lA)$ is a framed allegory, we can deduce this directly.
Namely, we have
\[ H\sbb G\sbb = \bigvee (\Psi H\sb \Theta G\sb)
\le \bigvee (\Psi\Psi H\sb G\sb) = \Psi (HG)\sb = (HG)\sbb;
\]
hence $H\sbb G\sbb = (H G)\sbb$, since both are maps in $\lmodk(\lA)$
and maps are discretely ordered.

\begin{remark}\label{rmk:tight-frs}
  If we regard congruences as a decategorification of the
  bicategory-enriched categories
  of~\cite{street:cauchy-enr,ckw:axiom-mod}, then loose morphisms
  between them decategorify \emph{modules} (a.k.a.\ profunctors).
  \autoref{thm:tcoll-cplt} says that the \emph{tight} maps between
  them decategorify \emph{functors}, just as framed allegories
  decategorify proarrow equipments.

  For instance, if \cA is a monoidal \ka-cocomplete poset, regarded as
  a \Fk-enriched category \lA with one object and only the identity
  being tight, then a \ka-ary congruence in \lA is precisely a
  symmetric \cA-enriched category with a \ka-small set of objects.  In
  this case, the tight morphisms in $\lmodk(\lA)$ are precisely
  \cA-enriched functors, while its loose morphisms are \cA-enriched
  profunctors.

  Similarly, if \lA is $\lRel_{\un}(\bC)$ for regular \bC, then a
  unary congruence has a unique representation as an internal
  equivalence relation.  If we view these as a particular sort of
  internal category, then the tight morphisms in $\lmodk(\lA)$ are
  precisely internal functors, while its loose morphisms are internal
  profunctors whose defining span is jointly monic.

  The general case is analogous, for suitably generalized ``functors''
  and ``profunctors''.
\end{remark}

Unlike the case of \SUPk-enriched categories, tight collages are
\emph{not} absolute \Fk-colimits.  Thus the universal property of
$\lmodk (\lA)$ differs from that of $\Modk (\cA)$.

\begin{lem}\label{thm:fall-cocplt}
  If \lA is a \ka-ary framed allegory, then:
  \begin{enumerate}[leftmargin=*,nolistsep]
  \item $\lmodk(\lA)$ is a \ka-ary framed allegory.\label{item:fmod1}
  \item $\fy\colon \lA\to\lmodk(\lA)$ is a \ka-ary framed allegory
    functor.\label{item:fmod3}
  \item All \ka-ary congruences in $\lmodk(\lA)$ have tight
    collages.\label{item:fmod2}
  \item If \lB is a \ka-ary framed allegory with tight collages of
    \ka-ary congruences, then
    \[ \FALLk(\lmodk(\lA),\lB\big)^\ell \xto{(-\circ \fy)}
    \FALLk\big(\lA,\lB\big)
    \]
    is an equivalence of categories, where the $(-)^\ell$ on the
    domain denotes the full subcategory determined by the functors
    which preserve tight collages of \ka-ary congruences coming from
    \lA.\label{item:fmod4}
  \end{enumerate}
\end{lem}
\begin{proof}
  Since the underlying \SUPk-category of $\lmodk(\lA)$ is
  $\Modk(\cA)$, condition~\ref{item:fmod1} will follow from
  \autoref{thm:alleg-cocplt}\ref{item:mod1} as soon as we show that
  every tight morphism in $\lmodk(\lA)$ is a map.  But if $G\colon
  \Phi\to \Theta$ is a tight morphism as in \autoref{thm:tcoll-cplt},
  with underlying loose morphism $G\sbb = \Theta G\sb$, then $G\pb
  \Theta$ defines a loose morphism $\Theta \rto \Phi$, and we have
  \begin{gather*}
    \bigvee(\Theta G\sb G\pb \Theta) \le \bigvee (\Theta\Theta)
    = \Theta \mathrlap{\qquad\text{and}}\\
    \bigvee(G\pb \Theta \Theta G\sb) = G\pb \Theta G\sb
    \ge \bigvee (G\pb G\sb \Phi) \ge \Phi.
  \end{gather*}
  Thus, $G\pb \Theta$ is right adjoint to $\Theta G\sb$,
  so~\ref{item:fmod1} holds.  And since an \Fk-functor is a \ka-ary
  framed allegory functor just when its underlying \SUPk-functor is a
  \ka-ary allegory functor,~\ref{item:fmod3} is immediate from
  \autoref{thm:alleg-cocplt}\ref{item:mod3}.

  Now by \autoref{thm:frcoll-coll}, the tight collage in
  $[\lA\op,\lF_{\ka}]$ of any \ka-ary congruence in $\lmodk (\lA)$ is
  also the tight collage of a \ka-ary directed congruence in \lA.  As
  in \autoref{thm:alleg-cocplt}, this directed congruence is actually
  a congruence, so its collage also lies in $\lmodk (\lA)$,
  yielding~\ref{item:fmod2}.

  For~\ref{item:fmod4}, we also argue essentially as in
  \autoref{thm:alleg-cocplt}.  Note that an \Fk-functor $F\colon
  \lmodk(\lA)\to\lB$ is a \ka-ary framed allegory functor if and only
  if $F\circ \fy$ is so.  Thus by~\cite[4.99]{kelly:enriched}, since
  \fy is fully faithful, $(-\circ \fy)$ induces an equivalence
  \[ \FALLk(\lmodk(\lA),\lB)^\ell \xto{(-\circ \fy)} \FALLk(\lA,\lB)'
  \]
  for any \lB.  Here the domain is the full subcategory of functors
  that are (pointwise) left Kan extensions of their restrictions to
  \lA, and the codomain is the full subcategory of functors which
  admit (pointwise) left Kan extensions along \fy.  Since
  $\lmodk(\lA)$ is a full sub-\Fk-category of $[\lA\op,\lF_{\ka}]$,
  the inclusion \fy is \Fk-dense.  Thus,
  by~\cite[5.29]{kelly:enriched}, $\FALLk(\lmodk(\lA),\lB)^\ell$
  consists of the functors that preserve tight collages of \ka-ary
  congruences coming from \lA.  And since \lB has collages of \ka-ary
  congruences, every object of $\lmodk (\lA)$ is a colimit (a tight
  collage) of a \ka-ary congruence in \lA, and these colimits are
  preserved by maps out of objects of \lA (since they are colimits in
  a presheaf category), the same proof as
  for~\cite[4.98]{kelly:enriched} proves that $\FALLk(\lA,\lB)' =
  \FALLk(\lA,\lB)$.
\end{proof}

We write $G\pbb\colon \Theta \rto \Phi$ for the right adjoint of
$G\sbb$ (represented by $G\pb \Theta$ as above).  Now we make the
following simple observation.

\begin{prop}\label{thm:mod-chordate}
  For a \ka-ary framed allegory \lA, the following are equivalent.
  \begin{itemize}[nolistsep]
  \item The chordate reflection of $\lmodk(\lA)$.
  \item $\Modk$ of the chordate reflection of \lA.
  \end{itemize}
\end{prop}
\begin{proof}
  This follows immediately from the explicit description in
  \autoref{thm:tcoll-cplt}, but we can also show directly that they
  have the same universal property: they both reflect \lA into
  chordate \ka-ary allegories with collages of \ka-ary congruences
  (since every collage in a chordate framed allegory is automatically
  tight, by \autoref{thm:cauchydense}\ref{item:cd3}).
\end{proof}

Therefore, although in \S\ref{sec:exact-completion} we constructed the
\ka-ary exact completion of a \ka-ary site \bC by applying \Modk to
the chordate reflection of $\lrelk (\bC)$, it could equally well be
defined as the chordate reflection of $\lmodk(\lrelk(\bC))$.  This is
useful because the latter admits an alternative description as a
category of fractions.

\begin{defn}\label{def:weq}
  We say that a tight map $f\colon x\to y$ in a framed allegory is a
  \textbf{weak equivalence} if $f\sb$ is an isomorphism.
\end{defn}

Since $f\sb$ has a right adjoint $f\pb$, this is equivalent to
requiring $1_x = f\pb f\sb$ and $1_y = f\sb f\pb$.

\begin{prop}\label{thm:mod-weq}
  A tight map $G\colon \Phi \to \Theta$ in $\lmodk(\lA)$, as in
  \autoref{thm:tcoll-cplt}, is a weak equivalence if and only if $\Phi
  = G\pb \Theta G\sb$ and $\bigvee 1_Y \le \bigvee (\Theta G\sb G\pb
  \Theta)$ in \lA.
\end{prop}
\begin{proof}
  Since $G\sbb = \Theta G\sb$ in \lA, and the identity morphism of
  $\Phi$ in $\lmodk(\lA)$ is \Phi itself in \lA, asking that $G\pbb
  G\sbb = 1_\Phi$ in $\lmodk(\lA)$ is to ask that $\Phi = \bigvee(G\pb
  \Theta\o \Theta G\sb)$ in \lA, which is equivalent to $\Phi = G\pb
  \Theta G\sb$.

  Similarly, asking that $1_\Theta = G\sbb G\pbb$ is to ask that
  $\Theta = \bigvee(\Theta G\sb G\pb \Theta)$ in \lA.  Since $G$
  consists of maps, we always have $\bigvee(\Theta G\sb G\pb \Theta)
  \le \bigvee(\Theta\Theta) = \Theta$, so the content is in $\Theta
  \le \bigvee(\Theta G\sb G\pb \Theta)$.  But since $\bigvee 1_Y\le
  \Theta$, this implies $\bigvee 1_Y \le \bigvee(\Theta G\sb G\pb
  \Theta)$, while the converse holds since $\bigvee(\Theta\Theta) =
  \Theta$.
\end{proof}

\begin{remark}
  If we regard tight maps of congruences as ``functors'' as in
  \autoref{rmk:tight-frs}, \autoref{thm:mod-weq} says that that the
  weak equivalences are those which are ``fully faithful'' and
  ``essentially surjective''.
\end{remark}

Recall that a subcategory $\bW$ of a category \bA which contains all
the objects is said to admit a \emph{calculus of right fractions} if
\begin{itemize}[noitemsep]
\item Given $f\colon y\to z$ in \bA and $p\colon x\to z$ in \bW, there
  exist $g\colon w\to x$ and $q\colon w\to y$ such that $q\in\bW$ and
  $p g = f q$ (the \emph{right Ore condition}), and
\item Given $p\colon y\to z$ in \bW and $f,g\colon x\toto y$ in \bA
  such that $p f = p g$, there exists $q\colon w\to x$ in \bW such
  that $f q = g q$ (the \emph{right cancellability condition}).
\end{itemize}
In this case, the localization $\bA[\bW^{-1}]$ can be constructed as
follows: its objects are those of \bA, and its morphisms from $x$ to
$y$ are equivalence classes of spans
\[ x \xleftarrow{p} w \xto{f} y
\]
such that $p\in \bW$, under the equivalence relation that identifies
$(p,f)$ with $(p',f')$ if we have a commutative diagram
\[ \xymatrix@-.5pc{
  & w \ar[dl]_p \ar[dr]^f \\
  x & z  \ar[u] \ar[d] & y\\
  & w' \ar[ul]^{p'} \ar[ur]_{f'}
}\]
with the common composite $z\to x$ in \bW.  The composite of $(p,f)$
with $(q,g)$ is defined to be $(p r, g h)$, where $q h = f r$ and
$r\in\bW$ (such $h$ and $r$ exist by the right Ore condition).

\begin{lemma}\label{thm:ktabun}
  In a weakly \ka-tabular framed allegory with tight collages of
  \ka-ary congruences, if $f,g\colon x\toto y$ are tight maps such
  that $f\sb = g\sb$, then there is a weak equivalence $k\colon u\to
  x$ such that $f k = g k$.
\end{lemma}
\begin{proof}
  By weak \ka-tabularity, if $f\sb = g\sb$ we have $f P = g P$ for a
  covering family $P\colon V \To x$.  Let $\Phi = P\pb P\sb$ be the
  congruence ``kernel'' of $P$, and let $Q\colon V\To u$ be a tight
  collage of $\Phi$.  In particular, $Q\sb$ is a loose collage of
  \Phi.  Since $P\sb$ is also a loose collage of $\Phi$ (by
  \autoref{thm:coll-char}), there is a unique loose isomorphism
  $h\colon u \lto x$ with $P\sb = h Q\sb$.  But because $P\sb$ admits
  the tightening $P$, and $Q$ is a tight collage, $h$ admits a unique
  tightening $k\colon u\to x$ with $P = k Q$.  Since $k\sb = h$ is an
  isomorphism, $k$ is a weak equivalence, and the universal property
  of $Q$ implies $f k = g k$.
\end{proof}

\begin{thm}\label{thm:rfrac}
  If \lA is weakly \ka-tabular with tight collages of \ka-ary
  congruences, then the weak equivalences in $\nTMap(\lA)$ admit a
  calculus of right fractions, and
  \[\nTMap (\lA)[\cW^{-1}] \cong \nMap (\cA).\]
\end{thm}
\begin{proof}
  The weak equivalences are clearly a subcategory and contain all
  objects.  For the right Ore condition, suppose $f\colon y\to z$ is a
  tight map and $p\colon x\to z$ a weak equivalence.  As in
  \autoref{thm:pb-detect}, we can find tight cocones $G\colon U \To x$
  and $Q\colon U \To y$ such that $p G = f Q$ and $p\pb f\sb =
  \bigvee(G\sb Q\pb)$.  Since $1_y \le f\pb f\sb = f\pb p\sb p\pb
  f\sb$, by \autoref{thm:entire-detect2} $Q$ is covering.

  Let $\Phi = Q\pb Q\sb$, and let $F\colon U \To w$ be a tight collage
  of $\Phi$.  By \autoref{thm:coll-char}, $Q\sb$ is a (loose) collage
  of $\Phi$, so there is a unique tight map $r\colon w\to y$ such that
  $r F = Q$, and (by uniqueness of loose collages) $r$ is a weak
  equivalence.  Similarly, we have
  \[G\sb \Phi = p\pb f\sb Q\sb \Phi \le p\pb f\sb Q\sb = Q\sb
  \]
  so there is a unique tight map $h\colon w\to x$ with $h F = G$.
  Finally, we have $p h F = p G = f Q = f r F$, so since $F$ is a
  tight collage, $p h = f r$.  This shows the right Ore condition.

  For right cancellability, suppose $p\colon y\to z$ is a weak
  equivalence and $f,g\colon x\toto y$ are tight maps with $p f = p
  g$.  Since $p\sb$ is an isomorphism, $f\sb = g\sb$, so we can apply
  \autoref{thm:ktabun}.  Thus the weak equivalences admit a calculus
  of right fractions.

  Now since $J\colon \nTMap(\lA) \to \nMap(\cA)$ inverts weak
  equivalences, it extends to a functor $\nTMap(\lA)[\cW^{-1}] \too
  \nMap(\cA)$,
  which, like $J$, takes a span $x \xleftarrow{p} w \xto{f} y$ (with
  $p$ a weak equivalence) to $f\sb p\pb$.  Since it is bijective on
  objects, it suffices to show that it is full and faithful.

  For fullness, suppose $\phi\colon x\lto y$ is a loose map in \lA,
  and let $\phi = \bigvee (F\sb G\pb)$ for tight cocones $F\colon Z\To
  y$, $G\colon Z\To x$.  By \autoref{thm:entire-detect2}, $G$ is
  covering.  Let $\Psi = G\pb G\sb$ and $H$ be a tight collage of
  $\Psi$.  Since $G$ is a loose collage of $\Psi$, we have a weak
  equivalence $r$ with $G = r H$.  Now for any $z\in Z$, we have
  $(f_z)\sb (g_z)\pb \le \phi$, hence $(g_z)\sb (f_z)\pb \le \phi\o$
  and so
  \[ \phi (g_z)\sb \le \phi (g_z)\sb (f_z)\pb (f_z)\sb
  \le \phi\phi\o (f_z)\sb \le (f_z)\sb.
  \]
  Thus, $\bigvee (F\sb \Psi) = \bigvee(F\sb G\pb G\sb) = \phi G\sb \le
  F\sb$, so since $H$ is a tight collage of $\Psi$, we have a unique
  $k\colon w\to y$ with $k H = F$.  Since tight collages are also
  loose collages, we have $\phi r\sb = k\sb$, hence $\phi = k\sb
  r\pb$; thus $\phi$ is in the image of $\nTMap(\lA)[\cW^{-1}]$.

  For faithfulness, suppose that $x \xleftarrow{p} w \xto{f} y$ and $x
  \xleftarrow{q} v \xto{g} y$ are spans with $p,q$ weak equivalences
  such that $f\sb p\pb = g\sb q\pb$.  As in \autoref{thm:pb-detect},
  we can find $R\colon U \To w$ and $S\colon U\To v$ such that $p R =
  q S$ and $p\pb q\sb = \bigvee(R\sb S\pb)$.  Since $p\pb q\sb$ is an
  isomorphism, by \autoref{thm:entire-detect2} $R$ and $S$ are both
  covering, and $R\pb R\sb = R\pb p\pb p\sb R\sb = S\pb q\pb q\sb S\sb
  = S\pb S\sb $.

  Let $H\colon U\To z$ be a tight collage of the congruence $\Phi =
  R\pb R\sb = S\pb S\sb$.  Then since $R$ and $S$ are loose collages
  of $\Phi$, we have weak equivalences $m\colon z\to w$ and $n\colon
  z\to v$ with $m H = R$ and $n H = S$.  Hence
  $ p m H = p R = q S = q n H$,
  so since $H$ is a tight collage, $p m = q n$ and is a weak
  equivalence.  Now
  \[ f\sb m\sb H\sb = f\sb R\sb = f\sb p\pb p\sb R\sb
  = g \sb q\pb q\sb S\sb  = g\sb S\sb = g\sb n\sb H\sb,
  \]
  so since $H\sb$ is a loose collage, $f\sb m\sb = g\sb n\sb$.  By
  \autoref{thm:ktabun}, we can find a weak equivalence $t$ with $f m t
  = g n t$ and (of course) $p m t = q n t$.  Thus $(p,f)=(q,g)$ in
  $\nTMap(\lA)[\cW^{-1}]$.
\end{proof}

Unfortunately, $\lmodk(\lrelk (\bC))$ does not quite satisfy the
hypotheses of \autoref{thm:rfrac}; it may not inherit the second half
of weak \ka-tabularity from $\lrelk (\bC)$.  But we can remedy this by
considering its subchordate reflection, which is easily seen to
inherit all the other relevant properties of $\lrelk (\bC)$.  This
yields a fairly explicit description of $\exk(\bC)$ as a category of
fractions, but we can improve it even further as follows.

\begin{defn}\label{def:surjeq}
  Let $\Phi\colon X\Rto X$ and $\Theta\colon Y\Rto Y$ be \ka-ary
  congruences in a framed \ka-ary allegory \lA.  A tight map $G\colon
  \Phi \to \Theta$ in $\lmodk (\lA)$ is a \textbf{surjective
    equivalence} if
  \begin{enumerate}[nolistsep]
  \item $\Phi = G\pb \Theta G\sb$, and\label{item:se1}
  \item $G\colon X\To Y$ is a covering family in \lA.\label{item:se2}
  \end{enumerate}
\end{defn}

Note that the condition on a functional array $G\colon X\To Y$ to be a
tight map $\Phi\to\Theta$ in $\lmodk (\lA)$, namely $\bigvee (G\sb
\Phi) \le \Theta G\sb$, is equivalent by adjunction to $\Phi \le G\pb
\Theta G\sb$.  Thus, in \autoref{def:surjeq} we do not need to assert
that $G$ is a tight map; it follows from~\ref{item:se1}.

\begin{lemma}
  A surjective equivalence is a weak equivalence in $\lmodk (\lA)$.
\end{lemma}
\begin{proof}
  If $G$ is a surjective equivalence, we have
  \[ \bigvee 1_Y \le \Theta = \bigvee(\Theta\Theta)
  = \bigvee (\Theta G\sb G\pb \Theta). \]
  So by \autoref{thm:mod-weq}, $G$ is a weak equivalence.
\end{proof}

\begin{lemma}\label{thm:surj-eqv}
  If \lA is weakly \ka-tabular and $G\colon \Phi \to \Theta$ is a weak
  equivalence in $\lmodk(\lA)$, then there exist surjective
  equivalences $F\colon \Psi\to \Theta$ and $H\colon \Psi\to \Phi$
  with $G\sbb H\sbb = F\sbb$.
\end{lemma}
\begin{proof}
  Let $\Theta G\sb = \bigvee(F\sb H\pb)$ for functional arrays of
  tight maps $F$ and $H$ in \lA.  Since
  \[ \bigvee 1 \le \bigvee (G\pb G\sb)
  \le \bigvee(G\pb \Theta\o \Theta G\sb),
  \]
  by \autoref{thm:entire-detect2} $H$ is a covering family.  But since
  $G$ is a weak equivalence, we also have $\bigvee 1 \le
  \bigvee(\Theta G\sb G\pb \Theta\o)$, so $F$ is also a covering
  family.  Let $\Psi = H\pb \Phi H\sb$; then $H$ becomes by definition
  a tight map $\Psi\to\Phi$ that is a surjective equivalence.

  Now $\bigvee(F\sb H\pb) = \Theta G\sb$ implies, by the mates
  correspondence, that $H\pb G\pb \le F\pb \Theta$, and hence also
  $G\sb H\sb \le \Theta F\sb$.  Therefore, we have
  \begin{equation*}
    \Psi = H\pb \Phi H\sb 
     = H\pb G\pb \Theta G\sb H\sb
     \le F\pb \Theta\Theta\Theta F\sb
    = F\pb \Theta F\sb.
  \end{equation*}
  Hence $F$ is a tight map $\Psi \to\Theta$.
  Now we also have
  \[ \bigvee(\Theta G\sb \Phi H\sb)
  = \bigvee(F\sb H\pb \Phi H\sb)
  = \bigvee(F\sb \Psi)
  \le \Theta F\sb,
  \]
  so that in $\lmodk(\lA)$, we have $G\sbb H\sbb \le F\sbb$, hence
  $G\sbb H\sbb = F\sbb$.  It follows that since $G$ and $H$ are weak
  equivalences in $\lmodk (\lA)$, so is $F$.  Thus, as it is a
  covering family in \lA, it is also a surjective equivalence $\Psi
  \to\Theta$.
\end{proof}

\begin{lemma}\label{thm:sbbeq}
  If $F,G\colon \Phi\toto\Theta$ are two tight maps in $\lmodk (\lA)$,
  then $F\sbb = G\sbb$ if and only if $\bigvee(F\sb G\pb) \le \Theta$.
\end{lemma}
\begin{proof}
  If $F\sbb = G\sbb$, then
  \[ 1 \le F\pbb G\sbb = \bigvee (F\pb \Theta\o \Theta G\sb)
  = F\pb \Theta G\sb,
  \]
  and hence
  \[ \bigvee(F\sb G\pb) \le \bigvee(F\sb F\pb \Theta G\sb G\pb)
  \le \Theta.\]
  Conversely, if $F\sb G\pb \le \Theta$, then
  \[ F\sbb = \Theta F\sb \le \bigvee(\Theta F\sb G\pb G\sb)
  \le \bigvee(\Theta\Theta G\sb) = \Theta G\sb = G\sbb, \]
  and dually.
\end{proof}

\begin{thm}\label{thm:frac-concrete}
  If \lA is weakly \ka-tabular, then the category of loose maps in
  $\lmodk(\lA)$ can be described as follows.
  \begin{enumerate}[nolistsep]
  \item Its objects are \ka-ary congruences in \lA.\label{item:fc1}
  \item For congruences $\Phi\colon X\Rto X$ and $\Theta\colon Y\Rto
    Y$, a morphism $\Phi \to \Theta$ is represented by a span $X
    \xLeftarrow{P} W \xRightarrow{F} Y$ of functional arrays of tight
    maps in \lA, such that $P$ is a covering family and $P\pb \Phi
    P\sb \le F\pb \Theta F\sb$.\label{item:fc2}
  \item Two such spans represent the same morphism $\Phi \to \Theta$
    if there is a diagram of functional arrays of tight maps in
    \lA:\label{item:fc3}
    \begin{equation}
      \vcenter{\xymatrix@R=1.5pc@C=2pc{
        & W \ar@{=>}[dl]_P \ar@{=>}[dr]^F \\
        X & U  \ar@{=>}[u]_(.4){S} \ar@{=>}[d]^(.4){S'} & Y\\
        & W' \ar@{=>}[ul]^{P'} \ar@{=>}[ur]_{F'}
      }}\label{eq:bfeqr}
    \end{equation}
    in which the left-hand quadrilateral commutes (but not necessarily
    the other one), $PS=P'S'$ is a covering family, and $\bigvee
    \big((F S)\sb (F'S')\pb\big) \le \Theta$.
  \end{enumerate}
\end{thm}
\begin{proof}
  By \autoref{thm:rfrac}, it suffices to consider the category of
  fractions of the tight maps in the subchordate reflection of
  $\lmodk(\lA)$, so a morphism $\Phi \to \Theta$ can be represented by
  a span $\Phi \xleftarrow{P} \Psi \xrightarrow{F} \Theta$ where $P$
  and $F$ are tight maps in $\lmodk(\lA)$ and $P$ is a weak
  equivalence.  By \autoref{thm:surj-eqv}, we may assume $P$ is a
  surjective equivalence.  Hence, $P\colon W\To X$ is a covering
  family and $\Psi = P\pb \Phi P\sb$, so $F$ being a tight map in
  $\lmodk(\lA)$ is equivalent to $P\pb \Phi P\sb \le F\pb \Theta
  F\sb$.  This gives~\ref{item:fc2}.

  Now, by the construction of a category of fractions, two such spans
  represent the same morphism just when we have a diagram of tight
  functional arrays
  \begin{equation}
    \vcenter{\xymatrix@R=1.5pc@C=2pc{
      & W \ar@{=>}[dl]_P \ar@{=>}[dr]^F \\
      X & Z  \ar@{=>}[u]_(.4){H} \ar@{=>}[d]^(.4){H'} & Y\\
      & W' \ar@{=>}[ul]^{P'} \ar@{=>}[ur]_{F'}
    }}\label{eq:bfeqr-cfrac}
  \end{equation}
  (not necessarily commutative)
  and a congruence $\Psi\colon Z\Rto Z$ such that
  \begin{enumerate}[noitemsep,label=(\arabic*)]
  \item $H$ and $H'$ are tight maps $\Psi\to P\pb\Phi P\sb$ and
    $\Psi\to (P')\pb\Phi (P')\sb$, respectively;\label{item:fca1}
  \item $(P H)\sbb = (P' H')\sbb$;\label{item:fca2}
  \item $P H$ is a weak equivalence $\Psi\to \Phi$ (hence so is
    $P'H'$); and\label{item:fca3}
  \item $(F H)\sbb = (F' H')\sbb$.\label{item:fca4}
  \end{enumerate}
  If we are given~\eqref{eq:bfeqr}, then we define $Z=U$, $H=S$,
  $H'=S'$, and
  \[\Psi = (PS)\pb \Phi (PS)\sb = (P'S')\pb \Phi (P'S')\sb.\]
  Then~\ref{item:fca1}--\ref{item:fca3} are immediate,
  while~\ref{item:fca4} follows from \autoref{thm:sbbeq}.
  
  Conversely, suppose given~\eqref{eq:bfeqr-cfrac}
  satisfying~\ref{item:fca1}--\ref{item:fca4}.  Then by
  \autoref{thm:surj-eqv}, we can find a surjective equivalence
  $Q\colon \Psi' \to \Psi$ such that $P H Q$ is a surjective
  equivalence.  Thus we have three covering families $P$, $P'$, and
  $PHQ$ of $X$.

  Since $\nTMap(\lA)$ is a locally \ka-ary site, there is a covering
  family $R\colon U\To X$ and functional arrays $S$, $S'$, and $S''$
  such that $R = P S = P' S' = PHQS''$.  Thus we have~\eqref{eq:bfeqr}
  in which the left-hand quadrilateral commutes.  Let $\Psi'' = R\pb
  \Phi R\sb$; then $S$ is a tight map $\Psi'' \to P\pb \Phi P\sb$ and
  $S'$ is a tight map $\Psi''\to (P')\pb \Phi (P')\sb$, and we
  calculate
  \begin{align*}
    F\sbb S\sbb
    &= F\sbb (P\pb \Phi P\sb) S\sb\\
    &= F\sbb P\pb \Phi P\sb H\sb Q\sb (S'')\sb\\
    &= F\sbb H\sbb Q\sb (S'')\sb\\
    &= (F')\sbb (H')\sbb Q\sb (S'')\sb\\
    &= (F')\sbb (P')\pb \Phi (P')\sb (H')\sb Q\sb (S'')\sb\\
    &= (F')\sbb (P')\pb \Phi P\sb H\sb Q\sb (S'')\sb\\
    &= (F')\sbb (P')\pb \Phi (P')\sb (S')\sb\\
    &= (F')\sbb (S')\sbb.
  \end{align*}
  By \autoref{thm:sbbeq}, this is equivalent to $\bigvee \big((F S)\sb
  (F'S')\pb\big) \le \Theta$, so~\ref{item:fc3} holds.
\end{proof}

\begin{thm}\label{thm:ex-concrete}
  The exact completion of a \ka-ary site \bC can be described as
  follows.
  \begin{enumerate}[nolistsep]
  \item Its objects are \ka-ary congruences \Phi in
    \bC.\label{item:exc1}
  \item Each morphism $\Phi\to \Psi$ is represented by a span of
    functional arrays $X \xLeftarrow{P} W \xRightarrow{F} Y$ such that
    $P$ is a covering family and $P^*\Phi \;\lle\; F^*
    \Psi$.\label{item:exc2}
  \item Two such spans $(P,F)$ and $(Q,G)$ determine the same morphism
    $\Phi\to \Psi$ if there is a covering family $R\colon U\To X$ and
    functional arrays $H$ and $K$ such that $R = P H = Q K$ and for
    all $u\in U$, $\{ f_{h(u)} h_u, g_{k(u)} k_u \} \;\le\;
    \Psi(fh(u), gk(u))$.\label{item:exc3}
  \end{enumerate}
\end{thm}
\begin{proof}
  This is a direct translation of \autoref{thm:frac-concrete}, except
  that in~\ref{item:exc3} we assert $\le$ rather than $\lle$.
  However, this can easily be obtained by passing to an extra covering
  family.
\end{proof}

This description of $\exk(\bC)$ is a decategorification of the
bicategory of \emph{internal anafunctors}, as described
in~\cite{bartels:hgt,roberts:ana}.  This approach to exact completion
does not seem as popular as the relational one, but one case of it
appears in the literature.

\begin{example}
  If \bC has weak finite limits and the trivial unary topology, then a
  unary congruence in \bC reduces precisely to a
  \emph{pseudo-equivalence relation} as defined
  in~\cite[Def.~6]{cv:reg-exact-cplt}.  The tight maps in
  $\lMod_{\un}(\lRel_{\un}(\bC))$ similarly reduce to morphisms of
  pseudo-equivalence relations, and every surjective equivalence has a
  section.  This implies that every span as in
  \autoref{thm:ex-concrete}\ref{item:exc2} is equivalent to one where
  $P$ is the identity, and likewise in~\ref{item:exc3} we may assume
  $R$ is an identity.  Thus, in this case the above construction of
  $\exu (\bC)$ yields precisely the exact completion of \bC as
  constructed in~\cite[Def.~14]{cv:reg-exact-cplt}.
\end{example}

\section{Exact completion and sheaves}
\label{sec:exact-compl-sheav}


Suppose that \bC is a \emph{small} \ka-ary site.  Since the category
$\nSh(\bC)$ of (small) sheaves on \bC is \KA-ary exact, hence also
\ka-ary exact, and the sheafified Yoneda embedding $\by\colon \bC \to
\nSh(\bC)$ is a morphism of sites, we have an induced \ka-ary regular
functor
\[ \bytil\colon \exk(\bC) \to \nSh(\bC). \]

\begin{lem}\label{thm:yimg}
  Given $\cF\in \nSh(\bC)$, consider the following statements.
  \begin{enumerate}[nolistsep]
  \item $\cF$ is in the image of $\bytil$.\label{item:y1}
  \item $\cF$ is the colimit, in $\nSh(\bC)$, of the \by-image of some
    \ka-ary congruence in \bC.\label{item:y2}
  \item $\cF$ is the colimit in $\nSh(\bC)$ of a \ka-small diagram of
    sheafified representables.\label{item:y3}
  \end{enumerate}
  Then~\ref{item:y1}$\Leftrightarrow$\ref{item:y2} always,
  while~\ref{item:y2}$\Rightarrow$\ref{item:y3} if $2\in \ka$,
  and~\ref{item:y3}$\Rightarrow$\ref{item:y2} if $\om\in\ka$.
\end{lem}
\begin{proof}
  Firstly, every object of $\exk(\bC)$ is the colimit of some \ka-ary
  congruence in \bC, and \bytil preserves such colimits.  This
  immediately gives~\ref{item:y1}$\Rightarrow$\ref{item:y2}.
  Conversely, if~\ref{item:y2} holds, say $\cF\cong \colim \by(\Phi)$,
  then $\fy(\Phi)$ has a colimit in $\exk(\bC)$ and this colimit is
  preserved by \bytil, hence its image is isomorphic to \cF;
  thus~\ref{item:y2}$\Rightarrow$\ref{item:y1}.

  If $2\in \ka$, then \ka-small sets are closed under binary
  coproducts, and thus a \ka-ary congruence is a \ka-small diagram;
  hence~\ref{item:y2}$\Rightarrow$\ref{item:y3}.  Conversely, suppose
  that $\om\in\ka$ and that \cF satisfies~\ref{item:y3}.  Then we can
  present \cF as a coequalizer
  \[ \xymatrix {\sum \by(Y)
    \ar@<-1mm>[r]_{\sum T} \ar@<1mm>[r]^{\sum S} &
    \sum \by(X) \ar[r] & \cF.}
  \]
  where $X$ and $Y$ are \ka-ary families of objects of \bC and
  $S,T\colon Y\To X$ are functional arrays.  We define a congruence
  $\Phi$ on $X$ as follows.  Given $x,x' \in X$, consider zigzags of
  the form
  \[\xymatrix@-.5pc{
    & y_1  \ar[dl] \ar[dr] && y_2 \ar[dl] \ar[dr]
    && y_n \ar[dl] \ar[dr] \\
    \mathllap{x =\, }x_0 && x_1 && \cdots
    && x_n \mathrlap{\,= x',}}
  \]
  in which each span $x_{i-1} \ot y_{i} \to x_i$ is either $(s_{y_i},
  t_{y_i})$ or $(t_{y_i}, s_{y_i})$.  For each $n$, there are a
  \ka-small number of such zigzags, so since $\om\in\ka$ there are
  overall a \ka-small number of them.  Since \bC is a \ka-ary site,
  each zigzag has a local \ka-prelimit; let $\Phi(x,x')$ be the
  disjoint union of one local \ka-prelimit of each zigzag.  Then \Phi
  is a \ka-ary congruence, and the colimit of $\by(\Phi)$ is also
  $\cF$.  Thus,~\ref{item:y3}$\Rightarrow$\ref{item:y2} when
  $\om\in\ka$.
\end{proof}

We write $\shk(\bC)$ for the full image of $\exk (\bC)$ in
$\nSh(\bC)$.

\begin{thm}\label{thm:ex-sheaves}
  The functor $\bytil \colon \exk(\bC) \to \shk(\bC)$ is an
  equivalence of categories.
\end{thm}
\begin{proof}
  It remains only to show that it is fully faithful.  Moreover, since
  every object of $\exk(\bC)$ is the colimit of some diagram in \bC
  (specifically, a \ka-ary congruence), and these colimits are
  preserved by \bytil, it suffices to prove that
  \[ \exk(\bC)\Big(\fy(y), \Phi\Big) \too
  \shk(\bC)\Big(\by(y), \bytil(\Phi)\Big) \cong \bytil(\Phi)(y)
  \]
  is a bijection, for any object $y$ and \ka-ary congruence
  $\Phi\colon X\Rto X$ in \bC.

  Now by \autoref{thm:ex-concrete}, a morphism $\fy(y) \to \Phi$ in
  $\exk (\bC)$ is determined by a covering family $P\colon W\To y$ and
  a functional array $F\colon W\To X$ such that the kernel of $P$
  locally refines $F^* \Phi$.  More concretely, this is a covering
  $P\colon W\To y$ and for each $w\in W$, a morphism $f_w\colon w\to
  f(w)$ for some $f(w)\in X$, such that for any $a\colon u\to w_1$ and
  $b\colon u\to w_2$ with $p_{w_1} a = p_{w_2} b$, there is a covering
  family $Q\colon V\To u$ such that $\{f_{w_1} a q_v, f_{w_2} b q_v\}$
  factors through $\Phi(f(w_1),f(w_2))$ for all $v\in V$.  Two such
  collections define the same morphism $\fy(y) \to \Phi$ if there is a
  covering $R\colon Z\To y$, which refines both $P$ and $P'$ as $R = P
  S = P' S'$, such that $\{ f_{s(z)} s_z, f'_{s'(z)} s'_z \}$ factors
  through $\Phi(fs(z),f's'(z))$ for any $z\in Z$.

  On the other hand, let \cF denote the colimit of the congruence
  $\by(\Phi)$ in the presheaf category $[\bC\op,\bSet]$.  Then
  \begin{equation}\label{eq:hypercover-pshf}
    \cF(w) = \coprod_{x\in X} \bC(w,x) \Big/ \sim,
  \end{equation}
  where the equivalence relation $\sim$ relates $\alpha_1\colon w\to
  x_1$ and $\alpha_2\colon w\to x_2$ if $\{\alpha_1,\alpha_2\}$
  factors through $\Phi(x_1,x_2)$.  Then $\bytil(\Phi)$ is the
  sheafification of \cF.  Using the 1-step construction of
  sheafification via hypercoverings
  (e.g.~\cite[Prop.~7.9]{dhi:hypercovers}
  or~\cite[\S6.5.3]{lurie:higher-topoi}), we can describe this as
  follows.  An element of $\bytil(\Phi)(y)$ is determined by a
  covering family $P\colon W\To y$ together with for each $w\in W$, an
  element $f_w \in \cF(w)$, such that for any $a\colon u\to w_1$ and
  $b\colon u\to w_2$ with $p_{w_1} a = p_{w_2} b$, there is a covering
  family $Q\colon V\To u$ such that $(q_v)^*a^*(f_{w_1}) = (q_v)^*
  b^*(f_{w_2})$ for all $v\in V$.  Two such collections of data define
  the same element of $\bytil(\Phi)$ if there is a covering family
  $R\colon Z\To y$, which refines both $P$ and $P'$ as $R = P S = P'
  S'$, and such that for any $z\in Z$ we have $s_z^*(f_{s(z)}) =
  (s')_z^*(f'_{s'(z)})$.

  These descriptions are essentially identical.  The only difference
  is that there is a bit more identification at first in an element of
  $\bytil(\Phi)$ (the elements $f_w$ are only specified up to $\Phi$
  to begin with), but this disappears after we quotient by the full
  equivalence relations.
\end{proof}

\begin{example}
  Suppose \bC has a trivial unary topology.  Then $\nSh(\bC) =
  [\bC\op,\bSet]$.  And if $X_1 \toto X_0$ is a unary congruence in
  \bC, i.e.\ a pseudo-equivalence relation, then the image of
  $\by(X_1) \to \by(X_0) \times \by(X_0)$ is an equivalence relation
  on $\by(X_0)$ whose quotient is the colimit of $\by(X_1 \toto X_0)$.
  Thus, every presheaf in the image of \bytil admits a surjection from
  a representable, such that the kernel of the surjection also admits
  a surjection from a representable.  Conversely, given a presheaf
  with this property, the two resulting representables give a
  pseudo-equivalence relation.  Thus we reproduce the characterization
  of the exact completion of a weakly lex category
  from~\cite{ht:free-regex} in terms of presheaves.
\end{example}

\begin{example}
  If \bC is a (unary) regular category with its regular unary
  topology, then we have seen that every unary congruence in \bC is
  equivalent to an internal equivalence relation.  Equivalent
  congruences have isomorphic colimits in $\nSh(\bC)$, so a sheaf on
  \bC lies in the image of \bytil just when it is the quotient of an
  equivalence relation in \bC.  Thus we also reproduce the
  characterization of the exact completion of a regular category
  from~\cite{lack:exreg-inf} in terms of sheaves.
\end{example}

Finally, the following example is important enough to call a theorem.

\begin{thm}\label{thm:topos}
  If \bC is a small \KA-ary site, then $\nEx_{\KA}(\bC)\simeq
  \nSh(\bC)$.\qed
\end{thm}

Combining this with \autoref{thm:adjoint}, we obtain a new proof of
the classical theorem that for a small site \bC and a Grothendieck
topos \bE, geometric morphisms $\bE\to \nSh(\bC)$ are equivalent to
(what we call) morphisms of \KA-ary sites $\bC\to\bE$.

\subsection{Small sheaves on large sites}

Now suppose that \bC is a \emph{large} (but moderate) \KA-ary site.
We write \bSET for the very large category of moderate sets.
Similarly, we write $\nSH(\bC)$ for the very large category of
\bSET-valued sheaves on \bC.  As before, we can show:

\begin{prop}
  For any \KA-ary site \bC, we have a full embedding
  $\bytil\colon\exK(\bC)\into \nSH(\bC)$, whose image consists of
  those \bSET-valued sheaves which are colimits in $\nSH(\bC)$ of
  small diagrams in \bC.
\end{prop}

\noindent
By a \textbf{small sheaf} we mean a sheaf in the image of \bytil, or
equivalently an object of $\exK(\bC)$.

\begin{example}\label{eg:small-presheaves}
  If \bC is a moderate category with finite \KA-prelimits and a
  trivial \KA-ary topology, then $\exK(\bC)$ is equivalent to the
  category $\mathcal{P}(\bC)$ of \emph{small presheaves} on \bC in the
  sense of~\cite{dl:lim-smallfr}: the colimits in $[\bC\op,\bSET]$ of
  small diagrams in \bC.  Thus, under these hypotheses,
  $\mathcal{P}(\bC)$ is \KA-ary exact, and in particular has finite
  limits.  This last conclusion is the finitary version of one
  direction of a theorem of~\cite{dl:lim-smallfr}; we will also deduce
  the converse in \autoref{eg:small-presheaves-2}.
\end{example}

\begin{remark}
  While a small presheaf on a locally small category necessarily takes
  values in small sets (since colimits in $[\bC\op,\bSET]$ are
  pointwise), the same is not true of a small sheaf.  One virtue of
  our approach is that we have defined $\exK(\bC)$, for a large
  \KA-ary site \bC, without needing the whole very-large category
  $\nSH(\bC)$.
\end{remark}

When \bC is large, $\exK(\bC)$ is not, in general, a Grothendieck
topos: it lacks a small generating set.  However, we have shown that
it is an infinitary-pretopos.  Conversely,
\autoref{thm:ex-prop}\ref{item:ep3} implies that any
infinitary-pretopos is equivalent to $\nEx_{\KA}(\bC)$ for some
\KA-ary site \bC, namely the infinitary-pretopos itself with its
\KA-canonical topology.  Thus we have a ``purely size-free'' version
of Giraud's theorem: a category is an infinitary-pretopos if and only
if it is the \KA-ary exact completion of a \KA-ary site.  (This
viewpoint also shows that there is really nothing special in this
about the case $\ka=\KA$.)

Moreover, \autoref{thm:adjoint} implies that the \KA-ary exact
completion also satisfies a suitable version of Diaconescu's theorem:
for any \KA-ary site \bC and infinitary-pretopos \bD, functors
$\nEx_\KA(\bC)\to \bD$ which preserve finite limits and small colimits
are naturally equivalent to morphisms of \KA-ary sites $\bC\to \bD$.
It is natural to think of such functors between infinitary pretoposes
as ``the inverse image parts of geometric morphisms,'' although in the
absence of smallness hypotheses, we have no adjoint functor theorem to
ensure the existence of a ``direct image part.''  In particular, if
\bC is the syntactic category of a ``large geometric theory,'' then
$\nEx_\KA(\bC)$ might naturally be considered the ``classifying
(pre)topos'' of that theory.

\begin{remark}
  In our presentation, the objects of $\nEx_\KA(\bC)$ are
  transparently seen as ``objects of \bC glued together.''  This makes
  it obvious, for instance, how to express schemes as objects of
  $\nEx_\KA(\bRing\op)$.  Namely, let $X$ be a family of rings
  covering a scheme $S$ by open affines, and for each $x_1,x_2\in X$
  let $\Phi(x_1,x_2)$ be an open affine cover of $\nSpec(x_1) \cap
  \nSpec(x_2) \subseteq S$.  Then \Phi is a \KA-ary congruence which
  presents $S$ as a small sheaf on $\bRing\op$.

  It should be possible to axiomatize further ``open map structure''
  on a \KA-ary site, along the lines of~\cite{jm:open-maps}
  and~\cite{lurie:strsp}, enabling the identification of a general
  class of ``schemes'' in $\nEx_\KA(\bC)$ as the congruences where
  gluing happens only along ``open subspaces.''
\end{remark}

\section{Postulated and lex colimits}
\label{sec:post-lex-colim}

In this section we consider how our notion of exact completion is
related to the \emph{postulated colimits} of~\cite{kock:postulated}
and the \emph{lex colimits} of~\cite{gl:lex-colimits}.  We will also
obtain cocompletions of sites with respect to weaker exactness
properties, generalizing those of~\cite{gl:lex-colimits}.

Let \bC be a (moderate) category, \bSET the category of moderate sets,
and $\by\colon \bC\to [\bC\op,\bSET]$ the Yoneda embedding.  By a
\textbf{realization} of a presheaf $\cJ\colon \bC\op \to \bSET$ we
will mean a morphism $\cJ\to \by(z)$ into a representable through
which any map from $\cJ$ into a representable factors uniquely.  In
other words, $\cJ\to \by(z)$ is a reflection of $\cJ$ into
representables.

For instance, for any functor $\bg\colon \bD\to\bC$, colimits of \bg
are equivalent to realizations of $\colim (\by\circ\bg)$.  More
generally, colimits of \bg weighted by $\cI\colon \bD\op\to\bSET$ are
equivalent to realizations of $\lan_{\bg\op} \cI$.

\begin{defn}
  Let \bC be a \ka-ary site.  A presheaf $\cJ\colon \bC\op \to \bSET$
  is \textbf{\ka-admissible} if there exists a \ka-ary congruence
  $\Phi$ in \bC such that $\cJ \cong \colim \by(\Phi)$ in
  $[\bC\op,\bSET]$.
\end{defn}

Since a \ka-ary congruence is a small diagram, any \ka-admissible
weight must be a \textbf{small presheaf}, i.e.\ a small colimit of
representables.  If $2\in \ka$, then a \ka-ary congruence is a
\ka-small diagram, and so any \ka-admissible weight is a
\textbf{\ka-small presheaf}, i.e.\ a \ka-small colimit of
representables.  Conversely, the construction used for
\autoref{thm:yimg}\ref{item:y3}$\Rightarrow$\ref{item:y2} gives:

\begin{lem}\label{thm:small-adm}
  If $\om\in\ka$, then any \ka-small presheaf on \bC is
  \ka-admissible.\qed
\end{lem}

In particular, every presheaf on a small site is \KA-admissible.

\begin{example}
  Let \bC have finite limits and let $R\toto X$ be an internal
  equivalence relation in \bC.  Then the quotient of $\by (R)\toto
  \by(X)$ in $[\bC\op,\bSET]$ is \un-admissible.  A realization of
  this presheaf is a quotient of $R\toto X$.
\end{example}

\begin{example}\label{thm:lext-adm}
  Given a \ka-ary family $X$ of objects in \bC, the coproduct $\sum
  \by(X)$ in $[\bC\op,\bSET]$ is \ka-small.  Moreover, it is
  \ka-admissible regardless of what \ka may be, since it is the
  colimit of the discrete congruence $\Delta_X$.  Of course,
  realizations of $\sum \by(X)$ are coproducts of $X$ in \bC.
\end{example}

\begin{example}\label{thm:adh-adm}
  Suppose \bC has finite limits and that we have a span $y
  \hookleftarrow x \to z$ in which $x\into y$ is monic.  If $\cJ$
  denotes the pushout of $\by(y) \hookleftarrow \by(x) \to \by(z)$,
  then a realization of $\cJ$ is a pushout of the given span in \bC.
  This $\cJ$ is \om-small, but as $\om\notin\om$,
  \autoref{thm:small-adm} does not apply.  Nevertheless, $\cJ$ is
  still \om-admissible.  For if we define $X = \{y,z\}$, we have a
  congruence $\Phi$ on $X$ where:
  \begin{itemize}[nolistsep]
  \item $\Phi(z,z) = \Delta_z$,
  \item $\Phi(y,z)$ and $\Phi(z,y)$ are the given span and its
    opposite, and
  \item $\Phi(y,y)$ consists of $\Delta_y$ together with the span $y
    \ot x\times_z x \to y$.
  \end{itemize}
  The fact that $x\into y$ is monic makes this a congruence, and it is
  clearly \om-ary.  It is easy to check that the colimit of
  $\by(\Phi)$ is $\cJ$.
\end{example}

If $\cJ$ is \ka-admissible, then a congruence $\Phi$ with
$\colim\by(\Phi) \cong \cJ$ is of course not uniquely determined.
However, it is determined up to a suitable sort of equivalence.

\begin{lem}\label{lem:adm-cong-uniq}
  Suppose \Phi and \Psi are \ka-ary congruences in a \ka-ary site on
  families $X$ and $Y$, with $\colim\by(\Phi) \cong \colim\by(\Psi)$.
  Then $\Phi$ and $\Psi$ are isomorphic in $\exk (\bC)$.
\end{lem}
\begin{proof}
  By \autoref{thm:ex-sheaves}, $\exk (\bC)$ is equivalent to a
  subcategory of $\nSH(\bC)$, via a functor $\bytil$ which preserves
  colimits of \ka-ary congruences and extends the sheafified Yoneda
  embedding of \bC.  Since sheafification preserves colimits,
  $\colim\by(\Phi) \cong \colim\by(\Psi)$ implies $\bytil(\Phi) \cong
  \bytil(\Psi)$, hence $\Phi \cong \Psi$ in $\exk (\bC)$.
\end{proof}

Since a collage of \Phi in $\lrelk (\bC)$ is equivalently an object of
\bC together with an isomorphism to \Phi in $\exk (\bC)$, it follows
that under the hypotheses of \autoref{lem:adm-cong-uniq}, giving a
collage of \Phi is equivalent to giving a collage of \Psi.  This
justifies considering the following definition as a property of $\cJ$
alone.

\begin{definition}
  Let \bC be a \ka-ary site and $\cJ\colon \bC\op\to\bSET$ be
  \ka-admissible.  A morphism $\cJ \to \by(z)$ is \textbf{postulated}
  if for some (hence any) \ka-ary congruence $\Phi\colon X\Rto X$ with
  an isomorphism $\colim\by(\Phi) \cong \cJ$, the induced cocone
  $F\colon X\To z$ has the property that $F\sb$ is a (loose) collage
  of $\Phi$ in $\lrelk (\bC)$.
\end{definition}

Postulated colimits are defined in~\cite{kock:postulated} for conical
colimits in a site with finite limits, using the internal logic.  The
notion was reformulated in~\cite[Prop.~6.5]{gl:lex-colimits}, also
assuming finite limits, in terms of a presentation of $\cJ$ as a
coequalizer:
\begin{equation}
  \sum \by(Y) \toto \sum \by(X) \to \cJ.\label{eq:postcoeqpres}
\end{equation}
In these definitions, there are two conditions on $\cJ\to \by(z)$ to
be postulated:
\begin{enumerate}[nolistsep]
\item The induced family $X\To z$ is covering.\label{item:pc1}
\item For each $x, x'\in X$, the induced family of maps into $x
  \times_z x'$ out of pullbacks of zigzags built out of spans with
  vertices in $Y$ (as in the proof of \autoref{thm:yimg}) is
  covering.\label{item:pc2}
\end{enumerate}
If~\eqref{eq:postcoeqpres} exhibits $\cJ$ as the colimit of a
congruence \Phi, then the pullback of any such zigzag factors locally
through $\Phi(x,x')$.  Hence, condition~\ref{item:pc2} above is
equivalent to asking that $x \times_z x'$ be covered by $\Phi(x,x')$
--- which in turn is equivalent to saying that the kernel of $X\To z$
is equivalent to \Phi.  Thus, in this case the above two conditions
reduce precisely to the two conditions in \autoref{thm:coll-char}
characterizing $z$ as a collage of \Phi.

Conversely, for an arbitrary~\eqref{eq:postcoeqpres}, the congruence
constructed in \autoref{thm:yimg} is built out of zigzags, so Kock's
notion of postulatedness reduces to being its collage.  Therefore,
modulo \ka-admissibility, our notion of postulatedness coincides with
Kock's.

We remark in passing that~\cite[Prop.~1.1]{kock:postulated} now
follows easily.

\begin{prop}
  In a subcanonical site, any postulated morphism is a realization.
\end{prop}
\begin{proof}
  According to our definition, a postulated morphism $\cJ\to\by(z)$
  exhibits $z$ as the collage of some congruence $\Phi$ with $\colim
  \by(\Phi)\cong \cJ$.  But if \bC is subcanonical, then $\lrelk
  (\bC)$ is chordate by \autoref{thm:chordate1}, so any collage is a
  tight collage.  Thus, by \autoref{thm:tcoll-colim}, any such collage
  is also a colimit of $\Phi$, hence a realization of $\cJ$.
\end{proof}

Thus, in a subcanonical site we may refer to a postulated morphism
$\cJ \to\by(z)$ as a \textbf{postulated realization}.

In~\cite{gl:lex-colimits}, Garner and Lack define a \emph{lex-weight}
to be a functor $\cJ\colon \bD\op\to \bSet$ where \bD is small and has
finite limits, and a $\cJ$-weighted \emph{lex-colimit} in a category
\bC with finite limits to be a $\cJ$-weighted colimit of a
finite-limit-preserving functor $\bg\colon \bD\to\bC$.  Let \sJ be a
class of lex-weights; by~\cite[6.4]{gl:lex-colimits}, the following
definition is equivalent to theirs.

\begin{definition}\label{def:Wexact}
  A category \bC is \textbf{\sJ-exact} if it has finite limits, and
  there exists a subcanonical topology on \bC such that for any
  $\cJ\colon \bD\op\to\bSet$ in \sJ and $\bg\colon \bD\to\bC$
  preserving finite limits, the presheaf $\lan_{\bg\op} \cJ$ on \bC
  has a postulated realization.
\end{definition}

Now let \ka be an arity class and suppose that each $\cJ\colon
\bD\op\to\bSet$ in $\sJ$ is \ka-admissible for the trivial topology on
$\bD$.  For each $\cJ\in \sJ$, let $\Phi_\cJ\colon X_\cJ \Rto X_\cJ$
be a \ka-ary congruence in \bD with $\colim \by(\Phi_\cJ) \cong \cJ$.
Note that if $\bg\colon \bD\to\bC$ preserves finite limits, then it is
a morphism of sites for any topology on $\bC$, and hence
$\bg(\Phi_\cJ)$ is a congruence in $\bC$.

Thus, \bC is \sJ-exact if and only if it admits some subcanonical
topology such that for any $\cJ\colon \bD\op\to\bSet$ in \sJ and
$\bg\colon \bD\to\bC$ preserving finite limits, this congruence
$\bg(\Phi_\cJ)$ has a collage.  (In particular, any \ka-ary exact
category is \cJ-exact.)  This is equivalent to asking that there exist
a cocone $\bg(X_\cJ) \To \by(z_{\cJ,\bg})$ under $\bg(\Phi_\cJ)$ which
is covering, and such that all the induced cocones $\bg(\Phi(x,x'))
\To \bg(x) \times_z \bg(x')$ are also covering.

In particular, if \bC is \sJ-exact, then the cocones $\bg(X_\cJ) \To
\by(z_{\cJ,\bg})$ and $\bg(\Phi(x,x')) \To \bg(x) \times_z \bg(x')$
are all universally effective-epic, so the topology that they generate
is subcanonical.  Since these cocones are \ka-ary, the topology they
generate is weakly \ka-ary, hence \ka-ary (since \bC has finite
limits).  We call it the \textbf{\sJ-exact topology} on \bC.  It makes
an implicit appearance in~\cite[\S7]{gl:lex-colimits}, where its
category of sheaves is characterized directly as those presheaves
which send ``$\sJ^*$-lex colimits'' to limits in \bSet.

\begin{definition}\label{def:Wsuperexact}
  A topology on a \sJ-exact category is \textbf{\sJ-superexact} if it
  contains the \sJ-exact topology.
\end{definition}

Thus, the topologies on \bC which exhibit it as \sJ-exact as in
\autoref{def:Wexact} are precisely the subcanonical and \sJ-superexact
ones.  For instance, the \ka-canonical topology on a \ka-ary exact
category is \sJ-superexact.  The prefix ``super-'' in
``\sJ-superexact'' is intended to dualize the prefix ``sub-'' in
``subcanonical''.

\begin{remark}
  For many familiar classes \sJ of lex-weights, the \sJ-exact topology
  is already generated by the coprojections $\bg(X_\cJ) \To
  z_{\cJ,\bg}$.  This is related to the remark
  in~\cite{gl:lex-colimits} that \sJ-lex-cocompleteness often, but not
  always, coincides with ``$\sJ^*$-lex-cocompleteness''.  In
  particular, the example in~\cite[\S5.11]{gl:lex-colimits} also gives
  a class \sJ for which the \sJ-exact topology is not generated merely
  by these coprojections.
\end{remark}

Recall from~\cite{gl:lex-colimits} that a functor \f between \sJ-exact
categories is called \textbf{\sJ-exact} if it preserves finite limits
and \sJ-lex-colimits.

\begin{lem}\label{thm:Wexfr}
  If \bC and \bD are \sJ-superexact subcanonical \ka-ary sites, then
  any morphism of sites $\f\colon \bC\to\bD$ is \sJ-exact.  The
  converse holds if \bC has the \sJ-exact topology.
\end{lem}
\begin{proof}
  By \autoref{thm:fc-morsite}, $\f\colon \bC\to\bD$ is a morphism of
  sites if and only if it preserves finite limits and covering
  families.  On the other hand, since \bD is \sJ-superexact, \f is
  \sJ-exact if and only if it preserves finite limits and also the
  covering families that exhibit \sJ-lex-colimits as collages.  If \f
  is a morphism of sites then it clearly does this.  Conversely, since
  the \sJ-exact topology is generated by these covering families, if
  \f preserves them then it preserves all covering families in that
  topology.
\end{proof}

Let $\SITEk^\sJ$ denote the full sub-2-category of \SITEk determined
by the \sJ-exact categories equipped with \sJ-superexact subcanonical
\ka-ary topologies.  Let $\EXk^\sJ$ denote its full sub-2-category
determined by the \sJ-exact topologies.

By \autoref{thm:Wexfr}, the morphisms in $\EXk^\sJ$ are the \sJ-exact
functors, so it is equivalent to the category $\sJ$-$\mathbf{EX}$
of~\cite{gl:lex-colimits}.  In particular, $\EXk^\sJ$ is independent
of \ka, as long as \sJ is \ka-admissible (and every \sJ is
\KA-admissible).

\begin{thm}\label{thm:Wadjoint}
  The inclusion $\SITEk^\sJ \into \SITEk$ is reflective.
\end{thm}
\begin{proof}
  By applying chordate reflection to framed allegories first, we may
  assume all our sites are subcanonical.  Let $\exk^\sJ (\bC)$ denote
  the smallest full subcategory of $\exk(\bC)$ which contains \bC and
  is closed under finite limits and \sJ-lex-colimits, and let
  $\fy^\sJ_\bC \colon \bC \to \exk^\sJ (\bC)$ be the inclusion.  Since
  $\exk^\sJ (\bC)$ is closed under finite limits in $\exk(\bC)$, the
  $\ka$-canonical topology of $\exk(\bC)$ restricts to it, making it a
  \ka-ary site such that both $\fy^\sJ_\bC$ and the inclusion
  $\exk^\sJ(\bC) \to \exk(\bC)$ are morphisms of sites.  Moreover, for
  any $\cJ\colon \bD\op\to\bSet$ in \sJ and any $\bg\colon \bD\to
  \exk(\bC)$ preserving finite limits, a postulated realization of
  $\lan_{\bg\op} \cJ$ is in particular a $\sJ$-lex-colimit.  Thus
  $\exk^\sJ (\bC)$ is closed under these, hence inherits
  $\sJ$-exactness from $\exk(\bC)$.

  Now suppose \bD is a \sJ-superexact subcanonical \ka-ary site with
  finite limits.  Then $\exk(\bD)$ is \ka-ary exact, hence also
  \sJ-superexact and subcanonical.  Moreover, $\fy_{\bD}\colon \bD \to
  \exk (\bD)$ is \sJ-exact by \autoref{thm:Wexfr}, which is to say
  that \bD is closed in $\exk (\bD)$ under finite limits and
  \sJ-lex-colimits.  We must show that
  \begin{equation}
    (-\circ \fy^\sJ_\bC) \colon \SITEk(\exk^\sJ (\bC), \bD)
    \too \SITEk (\bC,\bD)\label{eq:Wexcpltmap}
  \end{equation}
  is an equivalence.
  Firstly, any morphism of sites $\f\colon \bC\to\bD$ induces a
  morphism of sites $\exk (\f)\colon \exk (\bC) \to \exk (\bD)$.
  Since $\exk(f)$ is \sJ-exact, $(\exk (\f))^{-1}(\bD)\subseteq \exk
  (\bC)$ is closed under finite limits and \sJ-lex-colimits.
  Therefore, it contains $\exk^\sJ(\bC)$, which is to say that the
  composite morphism of sites $\exk^\sJ(\bC) \into \exk(\bC)
  \xto{\exk(f)} \exk(\bD)$ corestricts to \bD.  Since $\fy_\bD\colon
  \bD \to \exk(\bD)$ is fully faithful and creates both finite limits
  and covering families, the corestriction is again a morphism of
  sites (this might be regarded as a trivial special case of
  \autoref{thm:fy-dense} combined with \autoref{thm:dense-cancel} from
  the next section).  Thus~\eqref{eq:Wexcpltmap} is essentially
  surjective.

  Secondly, suppose $\f_1, \f_2\colon \exk^\sJ(\bC) \toto \bD$ are
  morphisms of sites and $\al\colon \f_1 \circ \fy^\sJ_\bC \to \f_2
  \circ \fy^\sJ_\bC$ is a transformation.  Then we have an induced
  transformation between functors $\exk (\bC) \toto \exk (\bD)$.
  Since $\f_1$ and $\f_2$ map $\exk^\sJ(\bC)$ into $\bD$ and the
  inclusion $\bD \into \exk(\bD)$ is fully faithful, this
  transformation restricts and corestricts to a transformation
  $\f_1\to \f_2$, which in turn also restricts to \al.
  Hence~\eqref{eq:Wexcpltmap} is full.

  Finally, every object of $\exk^\sJ(\bC)$ is a colimit of objects of
  \bC (a collage of a congruence), and morphisms of sites
  $\exk^\sJ(\bC) \to \bD$ preserve these colimits.  Thus, a
  transformation between such functors is determined by its
  restriction to \bC; hence~\eqref{eq:Wexcpltmap} is faithful.
\end{proof}

\begin{example}
  For any \ka, there is a \sJ for which \sJ-exact categories are
  exactly \ka-ary exact categories.  In fact, there are many.  One is
  the class of \emph{all} small lex-weights.  Another is the union
  of~\cite[\S5.2]{gl:lex-colimits} with the \ka-analogue
  of~\cite[\S5.3]{gl:lex-colimits}.
  
  For such a \sJ, the \sJ-exact topology on a \ka-ary exact category
  is the \ka-canonical one, so it is the \emph{only} \sJ-superexact
  subcanonical \ka-ary topology.  Thus $\SITEk^\sJ = \EXk^\sJ = \EXk$
  and hence $\exk^\sJ = \exk$.
\end{example}

The most-studied case of $\exk^\sJ$ other than $\exk$ itself is the
(\ka-ary) \emph{regular completion}.  By an evident \ka-ary
generalization of~\cite[\S5.6]{gl:lex-colimits}, there exists a \sJ
such that \sJ-exact categories coincide with \ka-ary regular
categories.  As with \ka-ary exact categories, for such a \sJ the only
subcanonical \sJ-superexact \ka-ary topology on a \ka-ary regular
category is the \ka-canonical one, so we have $\SITEk^\sJ = \EXk^\sJ =
\REGk$.  Thus we deduce:

\begin{thm}\label{thm:reg-adjt}
  The inclusion $\REGk\into \SITEk$ is reflective.\qed
\end{thm}

Let $\regk (\bC)$ denote this left biadjoint, the \emph{\ka-ary
  regular completion}.  We can describe it more explicitly:

\begin{lem}
  The following are equivalent for $\Phi \in \exk (\bC)$:
  \begin{enumerate}[nolistsep]
  \item $\Phi$ lies in $\regk (\bC)$.
  \item \Phi is a kernel of some \ka-to-finite array in \bC.
  \item \Phi admits a monomorphism to some finite product of objects
    of \bC.
  \end{enumerate}
\end{lem}
\begin{proof}
  It is easy to see that the second two conditions are equivalent, and
  that they define a subcategory of $\exk (\bC)$ which is \ka-ary
  regular and hence contains $\regk (\bC)$.  The converse containment
  follows since $\regk (\bC)$ contains \bC and is closed under
  quotients of kernels of \ka-to-finite arrays.
\end{proof}

We can now identify $\regk (\bC)$ with various regular completions in
the literature.

\begin{example}
  When $\ka=\un$ and \bC has (weak) finite limits and a trivial unary
  topology, the above description and universal property of $\regu(C)$
  are equivalent to those of the \emph{regular completion} of a
  category with (weak) finite limits, as in~\cite{cv:reg-exact-cplt}.
  If we instead identify $\exu(\bC)$ with a subcategory of
  $\nSh_{\un}(\bC)$, as in \S\ref{sec:exact-compl-sheav}, we obtain
  the characterization of the regular completion
  from~\cite{ht:free-regex}.
\end{example}

\begin{example}\label{eg:factsys2}
  Suppose \bC has finite limits and a pullback-stable factorization
  system $(\cE,\cM)$ which is \emph{proper} (i.e.\ \cE consists of
  epis and \cM of monos).  As in \autoref{eg:factsys}, \cE is then a
  unary topology on \bC.  Since \bC has products, any unary kernel in
  \bC is the kernel of a single morphism of \bC.  Moreover, if \Phi is
  the kernel of $f\colon x\to z$ and we factor $f$ as $x \xto{e} y
  \xto{m} z$ with $m\in \cM$ and $e\in \cE$, then \Phi is also the
  kernel of $e$ since $m$ is monic.  But since $e\in \cE$ is a cover,
  the induced tight morphism $e\colon \Phi \to \fy(y)$ in $\lRel_{\un}
  (\bC)$ is a weak equivalence, so that $\Phi\cong \fy(y)$ in
  $\exu(\bC)$.  Therefore, we can identify $\regu(\bC)$ with the full
  subcategory of $\exu(\bC)$ determined by objects of the form
  $\fy(x)$.

  Moreover, because we have $(\cE,\cM)$-factorizations, any relation
  $\fy(x) \rto \fy(y)$ between such congruences is equivalent to an
  \cM-morphism $z \to x\times y$.  Thus, the full subcategory of
  $\nMod_{\un}(\bC)$ on the objects of $\regu(\bC)$ is precisely the
  bicategory of relations defined in~\cite{kelly:rel-factsys}, and
  hence $\regu(\bC)$ is precisely the \emph{regular reflection} of \bC
  as defined there.  Its universal property is likewise the same: for
  regular \bD, regular functors $\regu(\bC) \to \bD$ correspond to
  functors $\bC\to\bD$ preserving finite limits (hence taking
  \cM-morphisms, which are monic, to monics) and taking \cE-morphisms
  to regular epis.
\end{example}

\begin{example}
  Let \bE be (unary) regular, \bC finitely complete, and $\f\colon
  \bE\to\bC$ finitely continuous.  By \autoref{eg:gentop}, the
  \f-images of all regular epis in \bE generate a smallest unary
  topology on \bC, such that \f becomes a morphism of sites.
  Moreover, by \autoref{thm:morsite-flim}, if \bD is regular, then a
  functor $\bg\colon \bC\to\bD$ is a morphism of sites if and only if
  it preserves finite limits and the composite $\bg\f$ is a regular
  functor.  Therefore, the unary regular completion $\regu(\bC)$ of
  \bC with this topology must be the \emph{relative regular
    completion} of~\cite{hofstra:relcpltn}.  Similarly, $\exu(\bC)$ is
  the relative exact completion.
\end{example}

Higher-ary regular completions are not well-studied, but one example
is worth noting.

\begin{example}
  Let $\bC_{\mathbb{T}}^{\mathrm{reg}}$ be the syntactic category of a
  regular theory $\mathbb{T}$, which is unary regular.  We can also
  regard it as an \om-ary site, with \om-ary topology generated by its
  canonical unary topology.  Then its \om-ary regular completion can
  be identified with the syntactic category of $\mathbb{T}$ regarded
  as a coherent theory.  Similarly, the \KA-ary regular completion of
  $\bC_{\mathbb{T}}^{\mathrm{reg}}$ is the syntactic category of
  $\mathbb{T}$ regarded as a geometric theory, and likewise if we
  started with a coherent theory and its coherent syntactic category.
\end{example}

We should mention a few other examples of \sJ-superexact completion.

\begin{example}\label{eg:superextensive}
  Generalizing~\cite[\S5.3]{gl:lex-colimits}, for any \ka we have a
  class \sJ for which \sJ-exactness coincides with \emph{\ka-ary
    lextensivity}~\cite{clw:ext-dist}, i.e.\ having finite limits and
  disjoint stable \ka-ary coproducts.  In \autoref{thm:lext-adm} we
  observed that this \sJ is always \ka-admissible, so any \ka-ary site
  has a \textbf{\ka-superextensive} completion.  If \bC is a trivial
  \ka-ary site with finite limits, its \ka-superextensive completion
  is its free \ka-ary coproduct completion $\nFam_\ka (\bC)$.
\end{example}

\begin{example}\label{eg:superadhesive}
  In~\cite[\S5.7]{gl:lex-colimits} is described a (singleton) class
  \sJ for which \sJ-exactness coincides with \emph{adhesivity} as
  in~\cite{ls:adhesive}.  In \autoref{thm:adh-adm} we saw that this
  class is \om-admissible, so any \om-ary site has a
  \textbf{superadhesive} completion.
\end{example}

We now intend to generalize the relative exact completions
of~\cite[\S7]{gl:lex-colimits}.  Following~\cite{gl:lex-colimits}, we
write $\sJ_1 \le \sJ_2$ if every $\sJ_2$-exact category or functor is
also $\sJ_1$-exact.

\begin{lem}\label{thm:W12superex}
  If $\sJ_1 \le \sJ_2$ and both are \ka-admissible, then any
  subcanonical $\sJ_2$-superexact \ka-ary topology is also
  $\sJ_1$-superexact.
\end{lem}
\begin{proof}
  Suppose \bC is a subcanonical $\sJ_2$-superexact \ka-ary site.  Then
  \bC is a $\sJ_2$-exact category (and hence in particular has finite
  limits).  By assumption, \bC is also $\sJ_1$-exact.  But $\exk(\bC)$
  is also $\sJ_2$-exact, and by \autoref{thm:Wexfr}, the embedding
  $\fy\colon \bC \into \exk(\bC)$ is $\sJ_2$-exact; hence it is also
  $\sJ_1$-exact.  In other words, \bC is closed in $\exk(\bC)$ under
  $\sJ_1$-lex-colimits.  But $\sJ_1$-lex-colimits in $\exk(\bC)$ are
  collages of congruences, and $\fy$ reflects collages of all
  congruences.  Thus, \bC is $\sJ_1$-superexact.
\end{proof}

Therefore, we have $\SITEk^{\sJ_2}\subseteq \SITEk^{\sJ_1}$ as
subcategories of $\SITEk$.  Restricting the domain of $\exk^{\sJ_2}$,
we obtain:

\begin{cor}
  If $\sJ_1 \le \sJ_2$, then the inclusion $\SITEk^{\sJ_2}\into
  \SITEk^{\sJ_1}$ is reflective.\qed
\end{cor}

This is close to~\cite[Theorem~7.7]{gl:lex-colimits}, but not quite
there, since for general \sJ we have $\SITEk^\sJ \neq \EXk^\sJ$.  If
$\sJ_1 \le \sJ_2$ then we have an obvious forgetful functor
$\EXk^{\sJ_2}\to \EXk^{\sJ_1}$, but just as in Examples~\ref{eg:exlex}
and~\ref{eg:exwlex}, we do not have $\EXk^{\sJ_2}\subseteq
\EXk^{\sJ_1}$ as subcategories of \SITEk.  However, as in
\autoref{eg:exlex} (but not \autoref{eg:exwlex}) we do have:

\begin{thm}
  If $\sJ_1 \le \sJ_2$, the forgetful functor $\EXk^{\sJ_2}\to
  \EXk^{\sJ_1}$ has a left biadjoint.
\end{thm}
\begin{proof}
  Suppose \bC is $\sJ_1$-exact with its $\sJ_1$-exact topology, and
  \bD is $\sJ_2$-exact with its $\sJ_2$-exact topology.  By
  \autoref{thm:W12superex}, \bD is also $\sJ_1$-superexact.  Thus, by
  \autoref{thm:Wexfr}, morphisms of sites $\bC\to\bD$ are the same as
  $\sJ_1$-exact functors.  Therefore, it suffices to show that
  $\exk^{\sJ_2}(\bC)$ lies in $\EXk^{\sJ_2}$, i.e.\ that its topology
  is the $\sJ_2$-exact one.

  Now the universal property of $\exk^{\sJ_2}(\bC)$ says that if \bE
  is any other subcanonical $\sJ_2$-superexact site, then morphisms of
  sites $\exk^{\sJ_2}(\bC) \to \bE$ are equivalent to morphisms of
  sites $\bC\to\bE$.  Let \bE be the category $\exk^{\sJ_2}(\bC)$ with
  its $\sJ_2$-exact topology.  Then by \autoref{thm:Wexfr}, the
  embedding $\fy^{\sJ_2}$ is a morphism of sites $\bC \to\bE$, and
  hence there is a morphism of sites $\f\colon \exk^{\sJ_2}(\bC) \to
  \bE$ such that the composite $\f\circ \fy^{\sJ_2}$ is isomorphic to
  $\fy^{\sJ_2}$.  But the identity functor $\bE \to \exk^{\sJ_2}(\bC)$
  is also a morphism of sites by \autoref{thm:Wexfr}, and so the
  universal property of $\exk^{\sJ_2}(\bC)$ implies that the composite
  $\exk^{\sJ_2}(\bC) \xto{\f} \bE \xto{1} \exk^{\sJ_2}(\bC)$ is
  isomorphic to the identity functor.  This implies that \f itself is
  isomorphic to the identity, and hence the topology of
  $\exk^{\sJ_2}(\bC)$ must be the $\sJ_2$-exact one.
\end{proof}

If \bC is small and we identify $\exk (\bC)$ with a subcategory of
$\nSh(\bC)$ as in \S\ref{sec:exact-compl-sheav}, then we can do the
same for these relative exact completions.  This is how they are
described in~\cite[\S7]{gl:lex-colimits}.  However, as in the absolute
case, our relative exact completions require no smallness hypotheses.

\section{Dense morphisms of sites}
\label{sec:dense}

The following definition is essentially standard.

\begin{defn}\label{def:dense}
  Let \bD be a site and let \bC be a category with a notion of
  ``covering family''.  We say that a functor $\f\colon \bC\to\bD$ is
  \textbf{dense} if the following hold.
  \begin{enumerate}[noitemsep]
  \item $P$ is a covering family in \bC if and only if $\f(P)$ is a
    covering family in \bD.\label{item:dense0}
  \item For every $u\in \bD$, there exists a covering family $P\colon
    V\To u$ in \bD such that each $v\in V$ is in the image of
    \f.\label{item:dense1}
  \item For every $x,y\in \bC$ and $g\colon \f(x)\to \f(y)$ in \bD,
    there exists a covering family $P\colon Z\To x$ and a cocone
    $H\colon Z\To y$ in \bC such that $g \circ \f(P) =
    \f(H)$.\label{item:dense2}
  \item For every $x,y\in \bC$ and morphisms $h,k\colon x\toto y$ in
    \bC such that $\f(h)=\f(k)$, there exists a covering family
    $P\colon Z\To x$ in \bC such that $h P = k P$.\label{item:dense3}
  \end{enumerate}
\end{defn}

Of course,~\ref{item:dense0} just means the covering families in \bC
are determined by those of \bD.  Denseness is usually defined only for
inclusions of subcategories, in which case~\ref{item:dense3} is
unnecessary, as is~\ref{item:dense2} if the subcategory is full.

\begin{thm}\label{thm:dense-redundant}
  Suppose \bD is a weakly \ka-ary site and $\f\colon \bC\to\bD$ is
  dense.  Then:
  \begin{enumerate}[label=(\alph*),nolistsep]
  \item \bC is a weakly \ka-ary site.\label{item:dr1}
  \item \f is a morphism of sites.\label{item:dr2}
  \item If \bD is \ka-ary, so is \bC.\label{item:dr3}
  \item For any family $X$ of objects in \bC, \f induces an
    equivalence between the preorders of \ka-sourced arrays over $X$
    and over $\f(X)$ (under the relation $\lle$).\label{item:dr4}
  \end{enumerate}
\end{thm}
\begin{proof}
  We prove~\ref{item:dr4} first.  Since \f preserves covering
  families, it preserves $\lle$.  Conversely, suppose $S\colon U\To X$
  and $T\colon V\To X$ are arrays in \bC with $\f(S)\lle \f(T)$.  Then
  we have a covering family $P\colon W\To \f(U)$ and a functional
  array $F\colon W\To \f(V)$ with $\f(S)\circ P= \f(T)\circ F$.
  By~\ref{item:dense1}, we have a covering family $Q\colon \f(Y) \To
  W$, and $\f(S)\circ P\circ Q = \f(T) \circ F\circ Q$.
  Applying~\ref{item:dense2} twice to $P\circ Q$ and $F\circ Q$ and
  passing to a common refinement, we obtain a covering family $R\colon
  Z\To Y$ and functional arrays $G\colon Z\To U$ and $H\colon Z\To X$
  such that $\f(G) = P\circ Q\circ \f(R)$ and $\f(H) = F\circ Q\circ
  \f(R)$.  Hence,
  \[\f(S \circ G) = \f(S)\circ P\circ Q\circ \f(R)
  = \f(T)\circ F\circ Q\circ \f(R) = \f(T\circ H).
  \]
  Thus, by~\ref{item:dense3} we have a covering family $N\colon M \To
  Z$ such that $S G N = T H N$.  Finally, since $\f(G) = P \circ
  Q\circ \f(R)$ and $P$, $Q$, and $R$ are covering,
  by~\ref{item:dense0} $G$ is also covering.  Therefore, the equality
  $S G N = T H N$ exhibits $S\lle T$.

  Thus, \f reflects as well as preserves $\lle$, so for it to be an
  equivalence of preorders it suffices for it to be essentially
  surjective.  But for any array $S\colon U\To \f(X)$, we can find a
  covering family $P\colon \f(Y)\To U$ in \bD, and a covering family
  $Q\colon Z\To Y$ and an array $T\colon Z\To X$ with $\f(T) = S \circ
  P \circ \f(Q)$.  Hence $S$ is locally equivalent to $\f(T)$.

  Next we prove~\ref{item:dr1}.  The only axiom of a weakly \ka-ary
  site not obviously implied by~\ref{item:dense0} is
  pullback-stability.  Suppose, therefore, that $P\colon V\To u$ is a
  covering family in \bC and $g\colon x\to u$ is a morphism; then we
  have a covering family $R\colon Z\To \f(x)$ with $\f(g) \circ R \le
  \f(P)$.  By~\ref{item:dr4}, $R$ is locally equivalent to $\f(S)$ for
  some $S$, and $\f(g S) \lle \f(P)$ implies $g S \lle P$.  By
  definition of $\lle$, we have a covering $T$ with $g S T \le P$;
  hence $S T$ is what we want.

  Now we prove~\ref{item:dr2}.  Suppose $\bg\colon \bE \to\bC$ is a
  finite diagram and $T\colon u \To \f\bg(\bE)$ a cone over $\f\bg$ in
  \bD.  Then by~\ref{item:dr4}, $T$ is locally equivalent to $\f(S)$
  for some array $S\colon V \To \bg(\bE)$ in \bC.
  Applying~\ref{item:dense3} some finite number of times (once for
  each morphism in \bE) and passing to a common refinement, we obtain
  a covering family $P\colon W\To V$ such that $S P$ is an array over
  \bg.  Then $T$ is locally equivalent to, hence factors locally
  through, $\f(S P)$.

  Finally, we prove~\ref{item:dr3}.  Let \bg, $T$, $S$, and $P$ be as
  in the previous paragraph, and suppose $R\colon x\To \bg(\bE)$ is
  some cone over \bg in \bC.  Then $\f(R) \lle T \lle \f(S P)$, hence
  $R \lle S P$.  Thus, $S P$ is a local \ka-prelimit of \bg.
\end{proof}

Because of \autoref{thm:dense-redundant}\ref{item:dr2}, we may speak
of a \textbf{dense morphism of sites}.  Note that~\ref{item:dr2} fails
for the classical notion of ``morphism of sites''.

\begin{remark}
  In particular, if \bC is a site which admits a dense morphism of
  sites $\bC\to\bD$, where \bD is a subcanonical weakly \ka-ary site
  with finite limits, then \bC has local \ka-prelimits.  Thus, the
  ``solution-set condition'' of having local \ka-prelimits is a
  \emph{necessary} condition for a site to map densely to a \ka-ary
  exact category.  In \autoref{thm:fy-dense} we will show that any
  \ka-ary site \bC is dense in $\exk(\bC)$, so that this condition is
  also \emph{sufficient}.  This may be regarded as a generalization of
  the observation in~\cite[Prop.~4]{cv:reg-exact-cplt} that any
  projective cover of an exact category must have weak finite limits
  (see also \autoref{eg:projectives}).
\end{remark}

\begin{example}\label{eg:small-presheaves-2}
  If the category $\mathcal{P}(\bC)$ of small presheaves on \bC is
  finitely complete, then its \KA-canonical topology is \KA-ary and
  induces the trivial \KA-ary topology on \bC, while every small
  presheaf is covered by a small family of representables.  Thus,
  \autoref{thm:dense-redundant} implies that \bC has finite
  \KA-prelimits.  This is the other direction of the theorem mentioned
  in \autoref{eg:small-presheaves}.
\end{example}

\begin{remark}\label{rmk:locsite-dense}
  For locally \ka-ary sites as in \autoref{rmk:locsite}, we can
  consider the analogous notion of a \emph{dense pre-morphism of
    sites}.  The analogue of \autoref{thm:dense-redundant} then holds.
\end{remark}

\begin{corollary}\label{thm:dense-cancel}
  Let $\bC_1$, $\bC_2$, and $\bC_3$ be sites and $\f_1\colon \bC_1\to
  \bC_2$ and $\f_2\colon \bC_2\to\bC_3$ functors such that $\f_2 \f_1$
  is a morphism of sites and $\f_2$ is a dense morphism of sites.
  Then $\f_1$ is a morphism of sites.
\end{corollary}
\begin{proof}
  Since $\f_2 \f_1$ preserves covering families and $\f_2$ reflects
  them, $\f_1$ must preserve them.  Now suppose $\bg\colon \bE\to\bC$
  is a finite diagram and $T$ a cone over $\f_1 \bg$ with vertex $u$.
  Then $\f_2(T)$ is a cone over $\f_2 \f_1 \bg$ with vertex $\f_2(u)$,
  so it factors locally through the $\f_2\f_1$-image of some array
  over \bg.  By \autoref{thm:dense-redundant}\ref{item:dr4}, this
  implies that $T$ factors locally through the $\f_1$-image of some
  array over \bg; hence $\f_1$ is covering-flat.
\end{proof}

We now show that exact completion interacts as expected with dense
morphisms of sites.  On the one hand, dense functors become
equivalences on exact completion.

\begin{lem}\label{thm:dense-ff}
  If $\f\colon \bC\to\bD$ is a dense pre-morphism of sites, with \bC
  and \bD locally \ka-ary, then the induced functor $\lrelk(\f)\colon
  \lrelk(\bC)\to\lrelk(\bD)$ is 2-fully-faithful (an isomorphism on
  hom-posets) on underlying allegories.  Moreover, every object of
  $\lrelk(\bD)$ is the (loose) collage of the image of some congruence
  in $\lrelk(\bC)$.
\end{lem}
\begin{proof}
  The first statement follows immediately from
  \autoref{thm:dense-redundant}\ref{item:dr4}.  For the second, we
  observe that any object $u\in \bD$ admits a covering family $P\colon
  \f(V)\To u$.  Then $P\pb P\sb$ is a congruence in $\lrelk (\bD)$
  whose collage is $u$, and by 2-fully-faithfulness it is the image of
  some congruence in \bC.
\end{proof}

\begin{thm}\label{thm:dense}
  If $\f\colon \bC\to\bD$ is a dense pre-morphism of locally \ka-ary
  sites, then the induced functor $\exk(\f)\colon \exk(\bC)\to
  \exk(\bD)$ is an equivalence.
\end{thm}
\begin{proof}
  Since $\lrelk(\f)$ is 2-fully-faithful on allegories, by
  \autoref{thm:cocompletion}, so is $\lrelk(\exk(\f))$.  Hence, so is
  $\exk(\f)$.  And since every object of $\lrelk(\bD)$ is a collage of
  a congruence in $\lrelk(\bC)$, by \autoref{thm:coll-coll}, so is
  every object of $\lrelk (\exk (\bD))$.  Hence $\exk (\f)$ is also
  essentially surjective.
\end{proof}

\begin{example}\label{eg:projectives}
  We say that an object $z$ of a \ka-ary exact category is
  \emph{\ka-ary projective} if every \ka-ary extremal-epic cocone with
  target $z$ contains a split epic.  If \bC has a trivial \ka-ary
  topology, then every object of the form $\fy(x)$ is \ka-ary
  projective in $\exk(\bC)$.  Moreover, in this case every object of
  $\exk(\bC)$ is covered by a \ka-ary family of \ka-ary projectives,
  namely the family of objects on which it is a congruence.

  On the other hand, if \bC is \ka-ary exact and every object of \bC
  is covered by a \ka-ary family of \ka-ary projectives, then the full
  subcategory \bP of \ka-ary projectives satisfies
  \ref{def:dense}\ref{item:dense1} (and,
  obviously,~\ref{item:dense2}--\ref{item:dense3}), so that
  $\bP\to\bC$ is a dense morphism of sites and hence $\bC \simeq
  \exk(\bC) \simeq \exk(\bP)$.  Moreover, the induced topology on \bP
  is trivial.  Thus, a \ka-ary exact category is the \ka-ary exact
  completion of a trivial \ka-ary site exactly when every object is
  covered by a \ka-ary family of \ka-ary projectives, and in this case
  the category of \ka-ary projectives has \ka-prelimits.

  For $\ka=\un$, this was observed in~\cite{cv:reg-exact-cplt}.  For
  $\ka=\KA$, we obtain a characterization of small-presheaf categories
  of categories with finite \KA-prelimits.  If we add the additional
  assumption that there is a \emph{small} generating set of \KA-ary
  projectives, we recover a well-known characterization of presheaf
  toposes.
\end{example}

On the other hand, every site is dense in its own exact completion.

\begin{thm}\label{thm:fy-dense}
  For any (locally) \ka-ary site \bC, the functor $\fy\colon \bC \to
  \exk(\bC)$ is a dense (pre-)morphism of sites.
\end{thm}
\begin{proof}
  For \autoref{def:dense}\ref{item:dense0}, since the embedding of a
  \ka-ary allegory in its cocompletion under \ka-ary congruences is
  fully faithful, it reflects as well as preserves the property
  $\bigvee(P\sb P\pb) = 1_u$ for a cocone of maps.  Since this
  characterizes covering families in \bC and $\exk (\bC)$, it follows
  that \fy reflects as well as preserves covering families.

  For~\ref{item:dense1}, since every object of $\exk (\bC)$ is a
  quotient of the image of a \ka-ary congruence in \bC, in particular
  it admits a covering family whose domains are in the image of \fy.

  For~\ref{item:dense2}, note that since the chordate reflection of
  $\lrelk(\bC)$ embeds fully faithfully in its collage cocompletion,
  to give a morphism $g\colon \fy(x)\to \fy(y)$ in $\exk(\bC)$ is the
  same as to give a loose map $\phi\colon x\lto y$ in $\lrelk(\bC)$.
  By weak \ka-tabularity, for any such $\phi$ we have $\phi = \bigvee
  (F\sb P\pb)$ for $P\colon Z\To x$ and $F\colon Z\To y$ in \bC.
  Since $\phi$ is a map, by \autoref{thm:entire-detect2} $P$ is
  covering, and we have
  \[ F\sb \le \bigvee (F\sb P\pb P\sb) = \phi P\sb. \]
  Since maps are discretely ordered, $F\sb = \phi P\sb$,
  hence $\fy(F) = g \circ \fy(P)$.

  Finally, for~\ref{item:dense3}, if $f,g\colon x\toto y$ are
  morphisms in \bC with $\fy(f)=\fy(g)$, then $f\sb = g\sb$ as loose
  maps, hence (by the familiar weak \ka-tabularity of $\lrelk(\bC)$)
  there is a covering family $P$ with $f P = g P$.
\end{proof}

\begin{example}
  Recall (e.g. from~\cite[D3.3]{ptj:elephant}) that a Grothendieck
  topos is called \emph{coherent} if it is the topos of sheaves for
  the \om-canonical topology on an (\om-ary) pretopos, or equivalently
  if it has a small, finitely complete, and \om-ary site of
  definition.  Panagis Karazeris has pointed out that in fact, the
  topos of sheaves on \emph{any} small \om-ary site is coherent.  For
  by \autoref{thm:fy-dense}, the canonical functor $\fy\colon \bC \to
  \nEx_{\om}(\bC)$ is dense, hence induces an equivalence of sheaf
  toposes (a.k.a.\ \KA-ary exact completions); but $\nEx_{\om}(\bC)$
  is a pretopos with its \om-canonical topology, so its sheaf topos is
  coherent.  Similarly, the topos of sheaves on any unary site is a
  regular topos.
\end{example}

\begin{example}\label{eg:superext-un}
  Let \bC be a \ka-ary extensive category with a \ka-superextensive
  \ka-ary topology, as in \autoref{eg:superextensive}.  (Actually, we
  can be more general here, not requiring \bC to have all finite
  limits; the \ka-extensive topology exists on any \ka-extensive
  category.)  Then a \ka-ary cocone $V\To u$ is covering if and only
  if the single morphism $\sum V \to u$ is covering.  Thus, the
  topology of \bC is uniquely determined by its singleton covers,
  which themselves form a unary topology.  Let $\bC_{\un}$ and $\exk
  (\bC)_{\un}$ denote the categories \bC and $\exk (\bC)$ equipped
  with their unary topologies of singleton covers.  Note that this
  topology on $\exk (\bC)_{\un}$ is the \un-canonical one.

  Now since \ka-ary coproducts are postulated in any
  \ka-superextensive site, they are preserved by $\fy\colon \bC \to
  \exk(\bC)$.  Thus, since we can replace any covering family in \bC
  by a singleton cover, the denseness of $\fy\colon \bC \to \exk(\bC)$
  (\autoref{thm:fy-dense}) implies that it is also dense as a functor
  $\bC_{\un} \to \exk(\bC)_{\un}$.  But since $\exk(\bC)_{\un}$ is
  unary exact with its \un-canonical topology, it is its own unary
  exact completion.  Thus, by \autoref{thm:dense}, we have an
  equivalence $\exu (\bC_{\un})\simeq \exk (\bC)_{\un}$.
\end{example}

\begin{example}\label{thm:superext-fam}
  Let \bC be any \ka-ary site, and $\nFam_\ka (\bC)$ its free
  completion under \ka-ary coproducts.  Recall that the objects of
  $\nFam_\ka (\bC)$ are \ka-ary families of objects of \bC and its
  morphisms are functional arrays.  We remarked after
  \autoref{thm:fam-pb} that $\nFam_\ka (\bC)$ inherits a weakly unary
  topology, whose covers are covering families as in
  \autoref{def:covfam}.

  In fact, it is easy to see that this topology is unary.  This
  generalizes the observation of~\cite[4.1(ii)]{carboni:free-constr}
  that $\nFam_\ka (\bC)$ has finite limits when \bC does.  Since
  $\nFam_\ka (\bC)$ is \ka-extensive by~\cite[2.4]{clw:ext-dist}, this
  unary topology corresponds to a \ka-superextensive \ka-ary topology
  on $\nFam_\ka (\bC)$.  We notate the two resulting sites as
  $\nFam_\ka (\bC)_{\un}$ and $\nFam_\ka (\bC)_{\ka}$ respectively.
  By \autoref{eg:superext-un}, we have $\exu(\nFam_\ka
  (\bC)_{\un})\simeq \exk(\nFam_\ka (\bC)_{\ka})$.

  However, it is also easy to verify that the inclusion $\bC \into
  \nFam_\ka (\bC)_{\ka}$ is dense, so that $\exk (\bC)\simeq
  \exk(\nFam_\ka (\bC)_{\ka})$, and hence $\exk (\bC)\simeq
  \exu(\nFam_\ka (\bC)_{\un})$.  This justifies our earlier comments
  that the \ka-ary exact completion can be obtained as a \ka-ary
  coproduct completion followed by a unary exact completion.
\end{example}

\autoref{thm:fy-dense} also implies a generalization
of~\cite[C2.3.8]{ptj:elephant} to arbitrary \ka-ary sites (which would
not be true for the classical notion of ``morphism of sites'').

\begin{prop}\label{thm:msinduce}
  Let \bC and \bD be \ka-ary sites and $\bg\colon
  \exk(\bC)\to\exk(\bD)$ a morphism of sites.  For a functor $\f\colon
  \bC\to\bD$, the following are equivalent.
  \begin{enumerate}[nolistsep]
  \item \f is a morphism of sites and $\bg \cong
    \exk(\f)$.\label{item:msi1}
  \item The following square commutes up to
    isomorphism:\label{item:msi2}
    \[\vcenter{\xymatrix{
        \bC\ar[r]^\f\ar[d]_{\fy_\bC} &
        \bD\ar[d]^{\fy_\bD}\\
        \exk(\bC)\ar[r]_{\bg} &
        \exk (\bD)
      }}\]
  \end{enumerate}
\end{prop}
\begin{proof}
  By naturality of $\fy$, we
  have~\ref{item:msi1}$\Rightarrow$\ref{item:msi2}.  Conversely,
  if~\ref{item:msi2}, then $\bg \circ \fy_{\bC} \cong \fy_\bD \circ
  \f$ is a morphism of sites, and since $\fy_\bD$ is dense by
  \autoref{thm:fy-dense}, \autoref{thm:dense-cancel} implies that \f
  is a morphism of sites.  The isomorphism $\bg \cong \exk(\f)$ then
  follows by the universal property of exact completion.
\end{proof}

Finally, we can prove \autoref{thm:morsite-sheaves}.

\begin{prop}\label{thm:morsite-sheaves2}
  For a small category \bC and a small site \bD, a functor $\f\colon
  \bC\to\bD$ is covering-flat if and only if the composite
  \begin{equation}
    [\bC\op,\bSet] \xto{\lan_\f} [\bD\op,\bSet]
    \xto{\ba} \nSh(\bD)\label{eq:mslan2}
  \end{equation}
  preserves finite limits, where \ba denotes sheafification.  If \bC
  is moreover a site and \f a morphism of sites, then
  \[\f^*\colon [\bD\op,\bSet] \to [\bC\op,\bSet]\]
  takes $\nSh(\bD)$ into $\nSh(\bC)$, so \f induces a geometric
  morphism $\nSh(\bD) \to \nSh(\bC)$.
\end{prop}
\begin{proof}
  We regard \bD as a \KA-ary site, and \bC as a \KA-ary site with
  trivial topology.  Then by \autoref{thm:topos} we have
  $\exK(\bC)\simeq [\bC\op,\bSet]$ and $\exK(\bD)\simeq \nSh(\bD)$.

  Moreover, \f is covering-flat just when it is a morphism of sites.
  By \autoref{thm:msinduce}, this holds exactly when there is a
  morphism of sites \bg such that
  \[\vcenter{\xymatrix{
      \bC\ar[r]^\f\ar[dd]_{\by} &
      \bD\ar[d]^{\by}\\
      & [\bD\op,\bSet] \ar[d]^\ba \\
      [\bC\op,\bSet]\ar[r]_-{\bg} &
      \nSh(\bD)
    }}\]
  commutes up to isomorphism.  Now recall that morphisms of sites
  between Grothendieck toposes are precisely finite-limit-preserving
  and cocontinuous functors.  Thus, since $[\bC\op,\bSet]$ is the free
  cocompletion of \bC, if \bg exists it must be~\eqref{eq:mslan2}.
  Since~\eqref{eq:mslan2} is always cocontinuous, this holds precisely
  when~\eqref{eq:mslan2} preserves finite limits.

  Finally, if \bC has instead a nontrivial topology for which \f is a
  morphism of sites, then $\exK(\f)\colon \nSh(\bC) \to \nSh(\bD)$ is
  cocontinuous, hence has a right adjoint.  It is easy to verify that
  this right adjoint must be $\f^*$.
\end{proof}


\begin{thebibliography}{BCSW83}

\bibitem[AR94]{ar:loc-pres}
Ji{\v{r}}{\'{\i}} Ad{\'a}mek and Ji{\v{r}}{\'{\i}} Rosick{\'y}.
\newblock {\em Locally presentable and accessible categories}, volume 189 of
  {\em London Mathematical Society Lecture Note Series}.
\newblock Cambridge University Press, Cambridge, 1994.

\bibitem[Bar06]{bartels:hgt}
Toby Bartels.
\newblock {\em Higher gauge theory I: 2-Bundles}.
\newblock PhD thesis, University of California, Riverside, 2006.
\newblock arXiv:math/0410328.

\bibitem[BCSW83]{bcsw:variation-enr}
Renato Betti, Aurelio Carboni, Ross Street, and Robert Walters.
\newblock Variation through enrichment.
\newblock {\em J. Pure Appl. Algebra}, 29(2):109--127, 1983.

\bibitem[BLS12]{bls:weak-aspects}
Gabriella B{\"o}hm, Stephen Lack, and Ross Street.
\newblock Idempotent splittings, colimit completion, and weak aspects of the
  theory of monads.
\newblock {\em J. Pure Appl. Algebra}, 216(2):385--403, 2012.

\bibitem[Car95]{carboni:free-constr}
A.~Carboni.
\newblock Some free constructions in realizability and proof theory.
\newblock {\em J. Pure Appl. Algebra}, 103(2):117--148, 1995.

\bibitem[CKW87]{ckw:axiom-mod}
Aurelio Carboni, Stefano Kasangian, and Robert Walters.
\newblock An axiomatics for bicategories of modules.
\newblock {\em J. Pure Appl. Algebra}, 45(2):127--141, 1987.

\bibitem[CLW93]{clw:ext-dist}
Aurelio Carboni, Stephen Lack, and R.F.C. Walters.
\newblock Introduction to extensive and distributive categories.
\newblock {\em J. Pure Appl. Algebra}, 84(2):145--158, 1993.

\bibitem[CM82]{cm:ex-lex}
A.~Carboni and R.~Celia Magno.
\newblock The free exact category on a left exact one.
\newblock {\em J. Austral. Math. Soc. Ser. A}, 33(3):295--301, 1982.

\bibitem[CV98]{cv:reg-exact-cplt}
A.~Carboni and E.~M. Vitale.
\newblock Regular and exact completions.
\newblock {\em J. Pure Appl. Algebra}, 125(1-3):79--116, 1998.

\bibitem[CW87]{cw:cart-bicats-i}
A.~Carboni and R.F.C. Walters.
\newblock Cartesian bicategories. {I}.
\newblock {\em J. Pure Appl. Algebra}, 49(1-2):11--32, 1987.

\bibitem[CW02]{cw:regcplt}
C.~Centazzo and R.~J. Wood.
\newblock An extension of the regular completion.
\newblock {\em J. Pure Appl. Algebra}, 175(1-3):93--108, 2002.
\newblock Special volume celebrating the 70th birthday of Professor Max Kelly.

\bibitem[CW05]{cw:factreg}
C.~Centazzo and R.~J. Wood.
\newblock A factorization of regularity.
\newblock {\em J. Pure Appl. Algebra}, 203(1-3):83--103, 2005.

\bibitem[DHI04]{dhi:hypercovers}
Daniel Dugger, Sharon Hollander, and Daniel~C. Isaksen.
\newblock Hypercovers and simplicial presheaves.
\newblock {\em Math. Proc. Cambridge Philos. Soc.}, 136(1):9--51, 2004.

\bibitem[DL07]{dl:lim-smallfr}
Brian~J. Day and Stephen Lack.
\newblock Limits of small functors.
\newblock {\em J. Pure Appl. Algebra}, 210(3):651--663, 2007.

\bibitem[Fef69]{feferman:fdns-of-ct}
Solomon Feferman.
\newblock Set-theoretical foundations of category theory.
\newblock In {\em Reports of the Midwest Category Seminar. III}, pages
  201--247. Springer, Berlin, 1969.

\bibitem[Fre12]{frey:mfpo-pca}
Jonas Frey.
\newblock Multiform preorders and partial combinatory algebras.
\newblock \url{http://www.pps.jussieu.fr/~frey/li2012.pdf}, February 2012.
\newblock Talk at \emph{Logic and interactions 2012}.

\bibitem[FS90]{fs:catall}
Peter~J. Freyd and Andre Scedrov.
\newblock {\em Categories, allegories}, volume~39 of {\em North-Holland
  Mathematical Library}.
\newblock North-Holland Publishing Co., Amsterdam, 1990.

\bibitem[GL12]{gl:lex-colimits}
Richard Garner and Stephen Lack.
\newblock Lex colimits.
\newblock {\em Journal of Pure and Applied Algebra}, 216(6):1372 -- 1396, 2012.
\newblock arXiv:1107.0778.

\bibitem[Hof04]{hofstra:relcpltn}
P.~J.~W. Hofstra.
\newblock Relative completions.
\newblock {\em J. Pure Appl. Algebra}, 192(1-3):129--148, 2004.

\bibitem[HT96]{ht:free-regex}
Hongde Hu and Walter Tholen.
\newblock A note on free regular and exact completions and their infinitary
  generalizations.
\newblock {\em Theory and Applications of Categories}, 2:113--132, 1996.

\bibitem[JM94]{jm:open-maps}
A.~Joyal and I.~Moerdijk.
\newblock A completeness theorem for open maps.
\newblock {\em Ann. Pure Appl. Logic}, 70(1):51--86, 1994.

\bibitem[Joh02]{ptj:elephant}
Peter~T. Johnstone.
\newblock {\em Sketches of an Elephant: A Topos Theory Compendium: Volumes 1
  and 2}.
\newblock Number~43 in Oxford Logic Guides. Oxford Science Publications, 2002.

\bibitem[Kar04]{karazeris:flatness}
Panagis Karazeris.
\newblock Notions of flatness relative to a {G}rothendieck topology.
\newblock {\em Theory Appl. Categ.}, 11(5):225--236, 2004.

\bibitem[Kel82]{kelly:enriched}
G.~M. Kelly.
\newblock {\em Basic concepts of enriched category theory}, volume~64 of {\em
  London Mathematical Society Lecture Note Series}.
\newblock Cambridge University Press, 1982.
\newblock Also available in Reprints in Theory and Applications of Categories,
  No. 10 (2005) pp. 1-136, at
  \url{http://www.tac.mta.ca/tac/reprints/articles/10/tr10abs.html}.

\bibitem[Kel91]{kelly:rel-factsys}
G.~M. Kelly.
\newblock A note on relations relative to a factorization system.
\newblock In {\em Category theory ({C}omo, 1990)}, volume 1488 of {\em Lecture
  Notes in Math.}, pages 249--261. Springer, Berlin, 1991.

\bibitem[Koc89]{kock:postulated}
Anders Kock.
\newblock Postulated colimits and left exactness of {K}an-extensions.
\newblock Aarhus Preprint 1989/90 no. 9, Retyped in TeX in the fall of 2003.
  Available at \url{http://home.imf.au.dk/kock/}, 1989.

\bibitem[Lac99]{lack:exreg-inf}
Stephen Lack.
\newblock A note on the exact completion of a regular category, and its
  infinitary generalizations.
\newblock {\em Theory Appl. Categ.}, 5:No.\ 3, 70--80 (electronic), 1999.

\bibitem[Law74]{lawvere:metric-spaces}
F.~William Lawvere.
\newblock Metric spaces, generalized logic, and closed categories.
\newblock {\em Rend. Sem. Mat. Fis. Milano}, 43:135--166, 1974.
\newblock Reprinted as Repr. Theory Appl. Categ. 1:1--37, 2002.

\bibitem[LS02]{ls:ftm2}
Stephen Lack and Ross Street.
\newblock The formal theory of monads. {II}.
\newblock {\em J. Pure Appl. Algebra}, 175(1-3):243--265, 2002.
\newblock Special volume celebrating the 70th birthday of Professor Max Kelly.

\bibitem[LS04]{ls:adhesive}
Stephen Lack and Pawe{\l} Soboci{\'n}ski.
\newblock Adhesive categories.
\newblock In {\em Foundations of software science and computation structures},
  volume 2987 of {\em Lecture Notes in Comput. Sci.}, pages 273--288. Springer,
  Berlin, 2004.

\bibitem[LS12]{ls:limlax}
Stephen Lack and Michael Shulman.
\newblock Enhanced 2-categories and limits for lax morphisms.
\newblock {\em Advances in Mathematics}, 229(1):294--356, 2012.
\newblock arXiv:1104.2111.

\bibitem[Lur09a]{lurie:strsp}
Jacob Lurie.
\newblock Derived {A}lgebraic {G}eometry {V}: Structured spaces.
\newblock arXiv:0905.0459, 2009.

\bibitem[Lur09b]{lurie:higher-topoi}
Jacob Lurie.
\newblock {\em Higher topos theory}.
\newblock Number 170 in Annals of Mathematics Studies. Princeton University
  Press, 2009.

\bibitem[Rob]{roberts:ana}
David~M. Roberts.
\newblock Internal categories, anafunctors and localisations.
\newblock arXiv:1101.2363.

\bibitem[RR90]{rr:colim-eff}
Edmund Robinson and Giuseppe Rosolini.
\newblock Colimit completions and the effective topos.
\newblock {\em J. Symbolic Logic}, 55(2):678--699, 1990.

\bibitem[SC75]{c:exreg}
Rosanna Succi~Cruciani.
\newblock La teoria delle relazioni nello studio di categorie regolari e di
  categorie esatte.
\newblock {\em Riv. Mat. Univ. Parma (4)}, 1:143--158, 1975.

\bibitem[Shu08]{shulman:frbi}
Michael Shulman.
\newblock Framed bicategories and monoidal fibrations.
\newblock {\em Theory and Applications of Categories}, 20(18):650--738
  (electronic), 2008.
\newblock arXiv:0706.1286.

\bibitem[Str81a]{street:cauchy-enr}
Ross Street.
\newblock Cauchy characterization of enriched categories.
\newblock {\em Rend. Sem. Mat. Fis. Milano}, 51:217--233 (1983), 1981.
\newblock Reprinted as Repr. Theory Appl. Categ. 4:1--16, 2004.

\bibitem[Str81b]{street:topos}
Ross Street.
\newblock Notions of topos.
\newblock {\em Bull. Austral. Math. Soc.}, 23(2):199--208, 1981.

\bibitem[Str83a]{street:absolute}
Ross Street.
\newblock Absolute colimits in enriched categories.
\newblock {\em Cahiers Topologie G\'eom. Diff\'erentielle}, 24(4):377--379,
  1983.

\bibitem[Str83b]{street:enr-cohom}
Ross Street.
\newblock Enriched categories and cohomology.
\newblock In {\em Proceedings of the Symposium on Categorical Algebra and
  Topology (Cape Town, 1981)}, volume~6, pages 265--283, 1983.
\newblock Reprinted as Repr. Theory Appl. Categ. 14:1--18, 2005.

\bibitem[Str84]{street:family}
Ross Street.
\newblock The family approach to total cocompleteness and toposes.
\newblock {\em Trans. Amer. Math. Soc.}, 284(1):355--369, 1984.

\bibitem[T{\etalchar{+}}11]{nlab:bicat-rel}
Todd Trimble et~al.
\newblock Bicategories of relations.
\newblock \url{http://ncatlab.org/nlab/revision/bicategory+of+relations/14},
  August 2011.

\bibitem[Wal81]{walters:sheaves-cauchy-1}
R.F.C. Walters.
\newblock Sheaves and {C}auchy-complete categories.
\newblock {\em Cahiers Topologie G\'eom. Diff\'erentielle}, 22(3):283--286,
  1981.
\newblock Third Colloquium on Categories, Part IV (Amiens, 1980).

\bibitem[Wal82]{walters:sheaves-cauchy-2}
R.F.C. Walters.
\newblock Sheaves on sites as {C}auchy-complete categories.
\newblock {\em J. Pure Appl. Algebra}, 24(1):95--102, 1982.

\bibitem[Wal09]{walters:relations}
R.F.C. Walters.
\newblock Categorical algebras of relations.
\newblock
  \url{http://rfcwalters.blogspot.com/2009/10/categorical-algebras-of-relation%
s.html}, October 2009.

\bibitem[Woo82]{wood:proarrows-i}
R.~J. Wood.
\newblock Abstract proarrows. {I}.
\newblock {\em Cahiers Topologie G\'eom. Diff\'erentielle}, 23(3):279--290,
  1982.

\bibitem[Woo85]{wood:proarrows-ii}
R.~J. Wood.
\newblock Proarrows. {II}.
\newblock {\em Cahiers Topologie G\'eom. Diff\'erentielle Cat\'eg.},
  26(2):135--168, 1985.

\end{thebibliography}

\newcommand{\etalchar}[1]{$^{#1}$}

\end{document}